\documentclass[a4paper]{article}

\usepackage{amsmath}%
\usepackage{amssymb}%
\usepackage[latin1]{inputenc}
\usepackage{enumitem}
\usepackage{euscript}

\usepackage{geometry}
\newtheorem{theorem}{Theorem}[section]

\newtheorem{corollary}[theorem]{Corollary}

\newtheorem{definition}[theorem]{Definition}

\newtheorem{lemma}[theorem]{Lemma}

\newtheorem{proposition}[theorem]{Proposition}
\newtheorem{remark}[theorem]{Remark}

\newenvironment{proof}[1][Proof]{\textbf{#1.} }{\ \rule{0.5em}{0.5em}}

\newcommand{\refeqn}[1]{(\ref{#1})}
\newcommand{\virgolette}[1]{``#1''}
\newcommand{\cinf}[0]{C^{\infty}}

\newcommand{\matr}[0]{\operatorname{Mat}}
\newcommand{\spann}[0]{\operatorname{span}}
\newcommand{\Iso}[0]{\operatorname{Iso}}
\newcommand{\acc}[1]{\`{#1}}
\newcommand{\argg}[0]{\operatorname{arg}}

\begin{document}

\title{{\bf Symmetries and invariance properties of stochastic differential equations driven by semimartingales with jumps}}
\author{Sergio Albeverio\thanks{Institut f\"ur Angewandte Mathematik, Rheinische Friedrich-Wilhelms-Universit\"at Bonn, Endenicher Allee 60, Bonn, Germany, \emph{email: albeverio@uni-bonn.de}}, Francesco C. De Vecchi\thanks{Dipartimento di Matematica, Universit\`a degli Studi di Milano, via Saldini 50, Milano, Italy, \emph{email: francesco.devecchi@unimi.it}}, Paola Morando\thanks{DISAA, Universit\`a degli Studi di Milano, via Celoria 2, Milano, Italy, \emph{email: paola.morando@unimi.it}} and Stefania Ugolini\thanks{Dipartimento di Matematica, Universit\`a degli Studi di Milano, via Saldini 50, Milano, Italy, \emph{email: stefania.ugolini@unimi.it}}}
\date{}

\maketitle

\begin{abstract}
Stochastic symmetries and related invariance properties of finite dimensional SDEs driven by general  c\acc{a}dl\acc{a}g  semimartingales taking values in Lie groups are defined
and investigated.  In order to enlarge the class of possible symmetries of SDEs, the new concepts of gauge and time symmetries for semimartingales on Lie
groups are introduced. Markovian and non-Markovian examples of gauge and time symmetric processes are  provided. The considered set of SDEs
includes affine and Marcus type SDEs as well as smooth SDEs driven by  L\'evy  processes. Non trivial invariance results concerning a class of
iterated random maps are obtained as special cases.
\end{abstract}

\bigskip



\section{Introduction}

The study of symmetries and invariance properties of ordinary and partial differential equations (ODEs and PDEs, respectively) is a classical and
well-developed area of
research  (see \cite{Bluman1989,Gaeta1994,Olver1993,Stephani1989})
and provides a powerful tool for computing some explicit solutions to the equations and analysing their qualitative behaviour.\\
The study of invariance properties of finite or infinite dimensional stochastic differential equations (SDEs)
is, in comparison, less developed and a systematic study could be fruitful from both the practical and the theoretical points of view.\\
The knowledge of some closed formulas is important in many applications of stochastic processes since it  permits to
develop faster and cheaper numerical algorithms for the simulation of the process or to evaluate interesting quantities related to it. Moreover,
the use of closed formulas
 allows the application of simpler statistical methods for the calibration of models (this is the reason for the popularity of affine
models in mathematical finance, see e.g. \cite{Cuchiero2011,Filipovic2003}, or of the Kalman filter and its generalizations in the theory of
stochastic filtering, see e.g. \cite{Bain2009}). The presence of symmetries and invariance properties is
 a strong clue for the possibility of closed formulas (see, for example, \cite{Craddock2009,Craddock2007,Craddock2012},
 where classical infinitesimal symmetry techniques are used for finding fundamental solutions of some diffusion processes applied in mathematical
 finance, or \cite{DeLara1997(1),DeLara1997(2),DM2016}, where geometrical methods based on Lie algebras are used to find new finite
 dimensional stochastic filters).\\
The investigation of invariance properties is relevant also from a theoretical point of view, in particular when stochastic processes are discussed in a
geometrical framework. Some interesting examples of this approach are the study of  L\'evy  processes on Lie groups \cite{Albeverio2007,Liao2004},
the geometric description of
stochastic filtering (see \cite{Elworthy2010}, where invariant diffusions on fibred bundles are discussed), and the study of variational
stochastic systems (\cite{Cruzeiro2016,Holm2015,Zambrini2015}).\\

In this paper we apply Sophus Lie original ideas to the study of stochastic symmetries of a finite dimensional SDE driven by
general c\acc{a}dl\acc{a}g semimartingales taking values in a Lie group. In particular, we introduce a group of transformations which change both
the family of processes, solutions to the considered SDE, and its driving noise, and we transform correspondingly the coefficients of the SDE. Therefore,  we look for the subgroup of these transformations which leave
invariant the set of solutions to the SDE.\\
 In order to clarify the novelty of our study we describe, without claiming to be exhaustive, some
previous approaches to the same problem.
There are essentially two natural approaches to the description of the symmetries of an SDE. The first one, applied when the solution processes are
Markovian semimartingales, consists in studying the invariance properties of the generator of the SDE solutions (which is an analytical object).
This approach, used by Glover et al. \cite{Glover1991,Glover1998,Glover1990}, Cohen de Lara \cite{DeLara1991,DeLara1995} and Liao
\cite{Liao1992,Liao2009}, deals with a large group of transformations involving both a general spatial transformation and a
solution-dependent stochastic time change.\\
The second approach, mainly applied to Brownian-motion-driven SDEs, consists in restricting the attention to a suitable set of transformations and directly apply
a natural notion of symmetry, closely inspired by the ODEs case (see Gaeta et al.
\cite{Gaeta2000,Gaeta1999}, Unal \cite{Unal2004}, Srihirun, Meleshko and Schulz
\cite{Srihirun2006}, Fredericks and Mahomed \cite{Fredericks2007}, Kozlov \cite{Kozlov2010(1),Kozlov2010(2)} for SDEs driven by Brownian motion (see also \cite{Gaeta2017} for a review on this subject) and L\'azaro-Cam\'i and Ortega \cite{Cami2009} for SDEs driven by general continuous semimartingales). \\
Both approaches have their strengths and weaknesses. For example, the first method  permits to treat a very general family of processes (all
Markovian processes on a metric space) and a large class of transformations with interesting applications  (see
\cite{Glover1990,Liao2009}), but the explicit calculation of the symmetries is quite difficult in the non-diffusive case.
Conversely, the second approach
 allows us to face the non-Markovian setting (see \cite{Cami2009}) and  permits easy explicit calculations. In particular, in the latter framework, it
 is possible   to
get the \emph{determining equations}, that are a set of first order PDEs which uniquely characterizes the  symmetries of the SDE. As for the first approach,
also the second one has interesting
applications (see, e.g., \cite{DMU2,Cami2009,Misawa1999}), even though,  until now, it has been confined to the case of continuous semimartingales
and usually considers a  family of transformations which is smaller than the family considered by the first method.\\

In this paper, aiming at  reducing the gap between the two approaches, we propose a possible foundation of the concept of symmetry for general SDEs
and we extend the methods introduced in \cite{DMU1} where, despite working in the setting of the second approach,
we introduced a large family of transformations which allows us to obtain all the symmetries of the first setting for Brownian-motion-driven SDEs.\\
In particular in \cite{DMU1} we  considered an SDE as a pair $(\mu(x),\sigma(x))$, where $\mu$ is the drift and $\sigma$ is the diffusion coefficient
defined on a manifold $M$ and we called solution to the SDE $(\mu,\sigma)$ a pair $(X,W)$, where $X$ is a semimartingale on $M$ and $W$ is an
 $n$-dimensional
 Brownian
motion. A stochastic transformation is a triad  $T=(\Phi,B,\eta)$, where $\Phi$ is a diffeomorphism of $M$, $B$ is a $X_t$-dependent rotation and
$\eta$ is a $X_t$-dependent density of a stochastic time change. The transformation $T$ induces an action $E_T$ on the SDE $(\mu,\sigma)$ and an action
$P_T$ on the process $(X,W)$. The operator $P_T$ acts on the process $(X,W)$ changing the semimartingale
$X$ by the diffeomorphism $\Phi$ and the time change $\int_0^t{\eta dt}$, and on the Brownian motion $W$ by the rotation $B$ and the
same time change. Since the Brownian motion is invariant with respect to both rotations and time rescaling, the process $P_T(X,W)$ is composed by
a semimartingale on $M$ and a new $n$-dimensional Brownian motion. The action $E_T$ of the stochastic transformation on $(\mu,\sigma)$ is the
unique way of changing the SDE so that, if $(X,W)$ is a solution to $(\mu,\sigma)$, then $P_T(X,W)$ is a solution to $E_T(\mu,\sigma)$. \\
In this framework a symmetry is defined as a transformation $T$ which leaves the SDE $(\mu,\sigma)$ invariant. These transformations are the only ones
which preserve the set of solutions to the SDE $(\mu,\sigma)$. Since all actions $P_T$ and $E_T$ are explicitly determined in terms of $T=(\Phi,B, \eta)$,
it is possible to write the  determining equations satisfied by $T$ which can be solved explicitly with a
 computer algebra software (see \cite{DMU2}).\\
The main aim of the present paper is to generalize this approach  from Brownian-motion-driven SDEs to SDEs driven by general c\acc{a}dl\acc{a}g semimartingales taking values in  Lie
groups. There are two main differences with respect to the Brownian motion setting. The first one is the lack of a natural geometric transformation
rule for processes with jumps replacing the It\^o transformation rule for continuous processes. This fact makes the action of a
diffeomorphism $\Phi$ on an SDE more difficult to be described. The second is the fact that a general semimartingale has not the
symmetry properties of Brownian motion in the sense that we cannot \virgolette{rotate} it or make general time changes.\\

In order to address the first problem we restrict ourselves to a particular family of SDEs (that we call \emph{canonical SDEs}) introduced by Cohen
in \cite{Cohen1996,Cohen1996(2)} (see also \cite{Applebaum1997,Cohen1995}). In particular, we consider SDEs defined by a map
$\Psi:M \times N \rightarrow M$, where $M$ is the manifold where the solution lives and $N$ is the Lie group where the driving process takes values. This
definition simplifies the description of the transformations of the solutions $(X,Z) \in M \times N$. In fact, if $(X,Z)$ is a solution to the SDE
$\Psi(x,z)$  then, for any diffeomorphism $\Phi$, $(\Phi(X),Z)$ is a solution to the SDE $\Phi(\Psi(\Phi^{-1}(x),z))$ (see Theorem
\ref{theorem_geometrical1} and Theorem \ref{theorem_gauge1}). We remark that the family of canonical SDEs is not too restrictive: in fact it includes  affine
types SDEs, Marcus type SDEs, smooth SDEs driven by  L\'evy  processes and a class of iterated random maps (see Subsection \ref{subsection_comparison}
for further details).\\
The second problem is addressed by  introducing two new notions of invariance of a semimartingale defined on a Lie group. These two notions are
extensions of predictable transformations which preserve the law of $n$ dimensional Brownian motion and $\alpha$-stable processes studied for
example in \cite[Chapter 4]{Kallenberg2005}. The first notion, which we call \emph{gauge symmetry}, generalizes the rotation invariance of
Brownian motion, while the second one, which we call \emph{time symmetry}, is an extension of the time rescaling invariance of Brownian motion.
The concept of gauge symmetry group is based on the action $\Xi_g$ of a Lie group $\mathcal{G}$ ($g$ is an element of $\mathcal{G}$) on the Lie
group $N$ which preserves the identity $1_N$ of $N$. A semimartingale $Z$ admits $\mathcal{G}$ as gauge symmetry group if, for any locally
bounded predictable process $G_t, t\in \mathbb{R}_+$ taking values in $\mathcal{G}$, the well defined transformation $d\tilde{Z}=\Xi_{G_t}(dZ)$
has the same probability law of $Z$ (see Section
\ref{section_gauge}). A similar definition is given for the time symmetry, where $\Xi_g$ is replaced by an $\mathbb{R}_+$ action $\Gamma_r$ and the
process $G_t$ is replaced by an absolutely continuous time change $\beta_t$ (see Section \ref{section_time}).\\
Given an SDE $\Psi$ and a driving process $Z$ with gauge symmetry group $\Xi_g$ and time symmetry $\Gamma_r, r\in \mathbb{R}_+$, we are able to define a
stochastic transformation $T=(\Phi,B,\eta)$, where $\Phi$ and $\eta$ are a diffeomorphism respectively a density of a time change as in the Brownian
setting, while $B$ is a function taking values in $\mathcal{G}$ (in the Brownian setting $\mathcal{G}$ is the group of rotations in
$\mathbb{R}^n$). In order to generalize \cite{DMU1}, using the properties of canonical SDEs and of gauge and time symmetries, we define an action $E_T$ of $T$ on
the SDE $\Psi$ as well as an action $P_T$  of $T$ on the solutions $(X,Z)$.\\

In this paper there are three main novelties. The first one is that, for the first time, the notion of symmetry of an SDE driven by a
general c\acc{a}dl\acc{a}g, in principle non-Markovian, semimartingale is studied in full detail. The analysis is based on the introduction of
a group of transformations which permits both the space transformation $\Phi$ and the gauge and time transformations $\Xi_g,\Gamma_r$. In this
way our approach extends the results of \cite{Cami2009}, where only general continuous semimartingales $Z$ and space transformations $\Phi$ were
considered. We also generalize the results to the case of a Markovian process  on a manifold $M$ and with a regular generator. Indeed, due to the
introduction of gauge and time symmetries, we recover all smooth symmetries of a Markovian process which would be lost if
we had just considered  the space transformation $\Phi$.\\
The second novelty is the introduction of the notion of gauge symmetry group and time symmetry and the careful analysis of their properties. Predictable
transformations which preserve the law of a process have already been considered  for special classes of processes as the $n$ dimensional
Brownian motion, $\alpha$-stable processes or Poisson processes (see \cite{Kallenberg2005,Privault2012}), but it seems the first time that the
invariance with respect to transformations depending on general predictable processes is studied for general semimartingales taking values in 
Lie groups. Furthermore, proving Theorem \ref{theorem_characteristic2} and Theorem \ref{theorem_time2}, we translate the notion of gauge and time
symmetries into the language of characteristics of a semimartingale (see \cite{Jacod2003} for the characteristics of a  $\mathbb{R}^n$
semimartingale and Theorem \ref{theorem_characteristic1} for our extension to general Lie groups). This translation permits to see gauge and
time symmetries as special
examples of predictable transformations preserving the characteristics (and so the law) of a process. This new insight is certainly interesting in itself and, in our opinion, deserves a deeper investigation.\\
The third novelty of the paper is given by our explicit approach: indeed, we provide many results which permit to check explicitly whether a
semimartingale admits  given gauge and time symmetries and to compute stochastic transformations which are symmetries of a given SDE. In
particular, Theorem \ref{theorem_characteristic3} and Corollary \ref{corollary_characteristic1} provide easily applicable criteria to construct
gauge symmetric  L\'evy  processes (see also the corresponding Theorem \ref{theorem_time3} and Theorem \ref{theorem_time4} for time symmetries).
Analogously, Theorem \ref{theorem_characteristic4} permits to construct non-Markovian processes with a gauge symmetry group. Finally we obtain
the determining equations \refeqn{equation_determining2} which are satisfied, under some additional hypotheses on the jumps of the driving
process $Z$, by any infinitesimal symmetry. The possibility of providing explicit determining equations is the main reason to restrict our
attention to canonical SDEs instead of considering  more general classes of SDEs. Indeed, an interesting consequence of our study is that we
provide a black-box method, applicable in a several different situations,  which permits to explicitly compute symmetries of a given SDE or to construct
all the canonical SDEs  admitting a given symmetry.
For these reasons, in order to show the generality and the
  user-friendliness of our theory, we conclude the paper with an example inspired by the iterated random maps theory.  \\

The paper is organized as follows. In Section \ref{section_geometrical} we introduce both the notions of geometrical SDE and of canonical SDE, and we discuss
their transformation properties. In Section \ref{section_gauge} and in Section \ref{section_time} we
give the definition of gauge and time symmetries and we study their properties. Finally, in Section \ref{section_symmetry}, we
extend the study of symmetries of Brownian-motion-driven SDEs to SDEs driven by general c\acc{a}dl\acc{a}g semimartingales.

\section{Stochastic differential equations with jumps on manifolds}\label{section_geometrical}

\subsection{Geometrical SDEs with jumps}

\begin{definition}
An adapted c\acc{a}dl\acc{a}g stochastic process $X$ on a smooth manifold $M$ is a semimartingale if, for any smooth function $f \in
\cinf(M)$, the real-valued process $f(X)$ is a real-valued semimartingale.
\end{definition}

Simplifying the setting of  \cite{Cohen1996},  a stochastic differential equation (SDE) defined on a smooth manifold $M$ and driven by a general c\acc{a}dl\acc{a}g
semimartingale  on a  smooth manifold $N$ can be described in terms  of a smooth function
\begin{equation}\label{equation_manifold0}
\overline{\Psi}:M \times N \times N \rightarrow M.
\end{equation}
In particular, let $\overline{\Psi}(x,z',z)$  be a smooth function such that, for any $z \in N$, $\overline{\Psi}(\cdot,z,z)=id_M$ (the identity map on $M$). \\
We first consider the case where the manifolds $M,N$ are open subsets of $\mathbb{R}^m$ and $\mathbb{R}^n$ and we take two global coordinate systems
$x^i$ and $z^{\alpha}$ of $M$ and $N$ respectively. The semimartingale $X$ with values in $ M$ is a solution to the SDE defined by the map
$\overline{\Psi}$ and driven by the semimartingale $Z$ defined on $N$ if, for $t \ge 0$,
\begin{equation}\label{equation_manifold1}
\begin{array}{ccl}
X^i_t-X^i_0&=&\int_0^t{\partial_{z'^{\alpha}}(\overline{\Psi}^i)(X_{s_-},Z_{s_-},Z_{s_-})dZ^{\alpha}_s}+\frac{1}{2}\int_0^t{\partial_{z'^{\alpha}z'^{\beta}}(\overline{\Psi}^i)(X_{s_-},Z_{s_-},Z_{s_-})d[Z^{\alpha},Z^{\beta}]_s}\\
&+&\sum_{0\leq s \leq
t}\{\overline{\Psi}^i(X_{s_-},Z_{s},Z_{s_-})-\overline{\Psi}^i(X_{s_-},Z_{s_-},Z_{s_-})-\partial_{z'^{\alpha}}(\overline{\Psi}^i)(X_{s_-},Z_{s_-},Z_{s_-})\Delta
Z^{\alpha}_s\},
\end{array}
\end{equation}
where $\overline{\Psi}^i:=x^i(\overline{\Psi})$, the derivation $\partial_{z'^{\alpha}}$ is the derivative of $\overline{\Psi}^i(x,z',z)$ with respect to the second set $z'$ of variables on $N$ and with respect to the coordinates system $z^{\alpha}$, $X^i:=x^i(X)$, $Z^{\alpha}:=z^{\alpha}(Z)$ and $\Delta Z^{\alpha}_s:=Z^{\alpha}_s-Z^{\alpha}_{s_-}$. \\
In order to extend the previous definition to the case of two general smooth
manifolds $M,N$ we  introduce two embeddings $i_1:M
\rightarrow \mathbb{R}^{k_M}$ and $i_2:N \rightarrow
\mathbb{R}^{k_N}, k_M, k_N \in \mathbb{N}$, and an extension
$$\tilde{\Psi}: \mathbb{R}^{k_M} \times \mathbb{R}^{k_N} \times \mathbb{R}^{k_N} \rightarrow \mathbb{R}^{k_M},$$
of the map $\overline{\Psi}$ such that
$$\tilde{\Psi}(i_1(x),i_2(z'),i_2(z)))=i_1(\overline{\Psi}(x,z',z)).$$
A semimartingale $X$ defined on $M$ solves the SDE defined by $\overline{\Psi}$ with respect
to the noise $Z$ defined on $N$ if $i_1(X) \in \mathbb{R}^{k_M}$ solves the integral problem \refeqn{equation_manifold1} where the map
$\overline{\Psi}$ is replaced by $\tilde{\Psi}$ and the noise $Z$ is replaced by $i_2(Z)$.\\

We generalize  \refeqn{equation_manifold0}  by considering a map $\overline{\Psi}_k$ of the form
$$\overline{\Psi}_{\cdot}(\cdot,\cdot,\cdot):M \times N \times N \times \mathcal{K} \rightarrow M,$$
where $\mathcal{K}$ is a (general) metric space (although in this paper we mostly take $ \mathcal{K}$ as a finite dimensional smooth manifold),
$\overline{\Psi}_k$ is smooth in the $M,N$ variables, and $\overline{\Psi}_k$ and all its derivatives with respect to the $M,N$ variables are
continuous in all their arguments. Let $K$ be a predictable locally bounded process taking values in $\mathcal{K}$. If $M,N$ are two open
subsets of $\mathbb{R}^n,\mathbb{R}^m$, we say that $(X,Z)$ solves the SDE $\overline{\Psi}_{K_t}$ if, for $t\ge 0,$ \small
\begin{equation}\label{equation_manifold2}
\begin{array}{ccl}
X^i_t-X^i_0&=&\int_0^t{\partial_{z'^{\alpha}}(\overline{\Psi}^i_{K_s})
(X_{s_-},Z_{s_-},Z_{s_-})dZ^{\alpha}_s}+\frac{1}{2}\int_0^t{\partial_{z'^{\alpha}z'^{\beta}}(\overline{\Psi}^i_{K_s})(X_{s_-},Z_{s_-},Z_{s_-})d[Z^{\alpha},Z^{\beta}]_s}\\
&+&\sum_{0\leq s \leq
t}\{\overline{\Psi}^i_{K_s}(X_{s_-},Z_{s},Z_{s_-})-\overline{\Psi}^i_{K_s}(X_{s_-},Z_{s_-},Z_{s_-})-\partial_{z'^{\alpha}}(\overline{\Psi}^i_{K_s})(X_{s_-},Z_{s_-},Z_{s_-})\Delta
Z^{\alpha}_s\}.
\end{array}
\end{equation}
\normalsize
 The extension to the case where $M,N$ are general manifolds can be easily obtained as before by using embeddings $i_1,i_2$ and an
extension $\tilde{\Psi}_k$ of $\overline{\Psi}_k$ which is continuous in the $M,N,\mathcal{K}$ variables and smooth in the $N,M$ variables.

\begin{definition}\label{definition_solution}
Let $M,N$ be two subsets of $\mathbb{R}^m$ and $\mathbb{R}^n$ respectively,  $\mathcal{K}$ be  a metric space and  $K$ be a predictable locally
bounded process taking values in $\mathcal{K}$. A pair of semimartingales $(X,Z)$ on $M$ and $N$ respectively is a solution to the \emph{geometrical
SDE} defined by $\overline{\Psi}_{K_t}$ until the stopping time $\tau$ if $X$ and $Z$, stopped at the stopping time $\tau$, solve the integral equation
\refeqn{equation_manifold2}. If $N,M$ are two general manifolds, $(X,Z)$ solves the geometrical SDE defined by $\overline{\Psi}_{K_t}$ until the stopping time $\tau$ if, for any couple of embeddings $i_1,i_2$ of $M,N$ in $\mathbb{R}^{k_M},\mathbb{R}^{k_N}$ respectively and for
any extension $\tilde{\Psi}_k$ of $\overline{\Psi}_k$, the pair $(i_1(X),i_2(Z))$ is a solution to the SDE $\tilde{\Psi}_{K_t}$ until the stopping time $\tau$. If $(X,Z)$ is
a solution to the SDE $\overline{\Psi}_{K_t}$ until the stopping time $\tau$ we write
$$dX_t=\overline{\Psi}_{K_t}(dZ_t).$$
\end{definition}

When not strictly necessary, we omit the stopping time $\tau$ from the definition of solution to a SDE.

\begin{theorem}\label{theorem_manifold1}
Given two open subsets $M$ and $N$  of  $\mathbb{R}^m$ and
$\mathbb{R}^n$ respectively, for any semimartingale $Z$ on
$N$ and any $x_0 \in M$, there exist a stopping time $\tau$, almost surely strictly positive, and a semimartingale $X$ on $M$, uniquely defined until $\tau$ and such that $X_0=x_0$ almost surely, such that $(X,W)$ is a solution of the SDE $\Psi_{K_t}$ until the stopping time $\tau$.
Furthermore, if $M,N$ are two general manifolds, $Z$ is a semimartingale on $N$, $i_1,i_2$ are  two embeddings of $N,M$ in
$\mathbb{R}^{k_M}$ and $\mathbb{R}^{k_N}$ and $\tilde{\Psi}_k$ is any extension of $\overline{\Psi}_k$, then the unique solution
$(\tilde{X},i_2(Z))$ to the SDE $\tilde{\Psi}_k$ is of the form $(i_1(X),i_2(Z))$ for a unique semimartingale $X$ on $M$. Finally, the process
$X$ does not depend on the embeddings $i_1,i_2$ and on the extension $\tilde{\Psi}_k$.
\end{theorem}
\begin{proof}
Since the process $K$ is locally bounded, the function $\tilde{\Psi}_{K_t}$, up to a sequence of stopping times $\tau_n \rightarrow + \infty$,
is locally Lipschitz with Lipschitzianity constant uniform with respect to $\omega$. The proof of this fact can be found in \cite{Cohen1996},
Theorem 2.
${}\hfill$ \end{proof}\\

\subsection{Geometrical SDEs and diffeomorphisms}

The notion of geometrical SDE introduced  in Definition \ref{definition_solution} naturally suggests to consider transformations of solutions to an SDE.

\begin{theorem}\label{theorem_geometrical1}
Let $\Phi:M \rightarrow M'$ and $\tilde{\Phi}:N \rightarrow N'$ be two diffeomorphisms. If $(X,Z)$ is a solution to the geometrical SDE
$\overline{\Psi}_{K_t}$, then $(\Phi(X),\tilde{\Phi}(Z))$ is a solution of the geometrical SDE $\overline{\Psi}'_{K_t}$ defined by
$$\overline{\Psi}'_{K_t}(x,z',z)=\Phi(\overline{\Psi}_{K_t}(\Phi^{-1}(x),\tilde{\Phi}^{-1}(z'),\tilde{\Phi}^{-1}(z))).$$
\end{theorem}

In order to prove Theorem \ref{theorem_geometrical1} we start by  establishing the following  lemma.

\begin{lemma}\label{lemma_geometrical1}
Given $k$  c\acc{a}dl\acc{a}g semimartingales  $X^1,...,X^k$, let $H^{\alpha}_1,...,H^{\alpha}_k$, for $\alpha=1,...,r$, be predictable processes
which can be integrated along $X^1,...,X^k$ respectively. If $\Phi^{\alpha}(t,\omega,x^1,x'^1,...,x^k,x'^k):\mathbb{R}_+ \times \Omega
\times \mathbb{R}^{2k} \rightarrow \mathbb{R}$ are some progressively measurable random functions continuous in
$x^1,x'^1,...,x^k,x'^k$ and such that $|\Phi^{\alpha}(t,\omega,x^1,x'^1,...,x^k,x'^k)| \leq O((x^1-x'^1)^2+...+(x^k-x'^k)^2)$ as $x^i
\rightarrow x'^i$, for  almost every fixed $\omega \in \Omega$ and uniformly on compact subsets of $\mathbb{R}_+ \times \mathbb{R}^{2k}$, the processes
$$Z^{\alpha}_t=\int_0^t{H^{\alpha}_{i,s}dX^i_s}+\sum_{0 \leq s \leq t}\Phi^{\alpha}(s,\omega,X^1_{s_-},X^1_s,...,X^k_{s_-},X^k_s)$$
are semimartingales. Furthermore
\begin{eqnarray}
&\Delta Z^{\alpha}_t=H^{\alpha}_{i,t} \Delta X^i_t + \Phi^{\alpha}(t,\omega,X^1_{t_-},X^1_t,...,X^k_{t_-},X^k_t),\label{equation_gauge4}\\
&[Z^{\alpha},Z^{\beta}]^c_t=\int_0^t{H^{\alpha}_{i,s}H^{\beta}_{j,s}d[X^i,X^j]^c_s},&\label{equation_gauge2}\\
&\int_0^t{K_{\alpha,s}dZ^{\alpha}_s}=\int_0^t{K_{\alpha,s}H^{\alpha}_{i,s}dX^i_s}+\sum_{0 \leq s \leq t}K_{\alpha,s}\Phi^{\alpha}(s,\omega,X^1_{s_-},X^1_s,....,X^k_{s_-},X^k_s).&\label{equation_gauge3}
\end{eqnarray}
\end{lemma}
\begin{proof}
Since $\int_0^t{H^{\alpha}_{i,s}dX^i_s}$ are semimartingales, we only need  to prove that \\
$\tilde{Z}^{\alpha}_t=\sum_{0 \leq s \leq t}\Phi^{\alpha}(s,\omega,X^1_{s_-},X^1_s,...,X^k_{s_-},X^k_s)$ is a
c\acc{a}dl\acc{a}g process of bounded variation.\\
If $\tilde{Z}^{\alpha}$ are of bounded variation, then we can prove  \refeqn{equation_gauge4}, \refeqn{equation_gauge2} and \refeqn{equation_gauge3}.
In fact, if $\tilde{Z}^{\alpha}$ are of bounded variation, they do not contribute to the brackets $[Z^{\alpha},Z^{\beta}]^c$. Thus
$[Z^{\alpha},Z^{\beta}]^c=[Z^{\alpha}-\tilde{Z}^{\alpha},Z^{\beta}-\tilde{Z}^{\beta}]^c$ and we obtain equation \refeqn{equation_gauge2}.
Furthermore, since $\tilde{Z}^{\alpha}$ is a sum of pure jumps processes, $\tilde{Z}^{\alpha}$  is a pure jump process. Then we get equations
\refeqn{equation_gauge4} and \refeqn{equation_gauge3} by using that $\tilde{Z}^{\alpha}$ are  pure jump processes of bounded variation and that
the measures $d \tilde{Z}^{\alpha}$ are pure atomic measures.\\
The fact that $\tilde{Z}^{\alpha}$ is of bounded variation can be established by exploiting the standard argument used for proving It\^o formula
(see, e.g., \cite[Chapter II, Section 7]{Protter1990}).\\
Indeed, if $[X^1,X^1]_t(\omega),...,[X^k,X^k]_t(\omega) < + \infty$ for all $t \in
\mathbb{R}_+$, then $\sum_i\sum_{0 \leq s \leq t}(\Delta X^i_s)^2(\omega) \leq \sum_i [X^i,X^i]_T(\omega) < + \infty$. Since $X^i$ are
c\acc{a}dl\acc{a}g they are locally bounded and so, for all $t < T$ and for almost every $\omega \in \Omega$, there exists a $C(T,\omega)$ such
that
\begin{eqnarray*}
\operatorname{var}_{[0,T]}(\tilde{Z}^{\alpha}_t(\omega))& \leq & \sum_{0 \leq s \leq t} |\Phi^{\alpha}(s,\omega,X^1_{s_-},X^1_s,...,X^k_{s_-},X^k_s)|\\
&\leq & C(T,\omega) \left(\sum_i\sum_{0 \leq s \leq t}(\Delta
X^i_s)^2 \right) < + \infty.
\end{eqnarray*}
\hfill\end{proof}\\

\begin{remark}\label{remark_geometrical1}
Let $ \mathcal{K}$ be a metric space, $K \in \mathcal{K}$ be a locally bounded predictable process and $\tilde{\Phi}: \mathbb{R}_+ \times
\mathcal{K} \times \mathbb{R}^{2k} \rightarrow \mathbb{R}$ be a $C^2$ function in $\mathbb{R}^{2k}$ variables such that
$\tilde{\Phi}$ and all its derivatives are continuous in all their arguments. If
$\tilde{\Phi}(\cdot,\cdot,x^1,x^1,...,x^k,x^k)=\partial_{x'^i}(\tilde{\Phi})(\cdot,\cdot,x^1,x^1,...,x^k)=0$
 for $i=1,...k$, then $\Phi(t,\omega,...)=\tilde{\Phi}(t,K_t(\omega),...)$ satisfies the
hypothesis of Lemma \ref{lemma_geometrical1}.
\end{remark}

\begin{proof}[Proof of Theorem \ref{theorem_geometrical1}]
The proof is given for $M,M'=\mathbb{R}^m,N,N'=\mathbb{R}^n$ (or more generally $M,M',N,N'$ open subsets of
$\mathbb{R}^m,\mathbb{R}^n$). The general case follows exploiting an embedding of $M,M',N,N'$ in $\mathbb{R}^{k_M},\mathbb{R}^{k_N}$
and extending $\Phi,\tilde{\Phi}$ to diffeomorphisms defined in a neighbourhood of the image of $M,M',N,N'$ with respect to these embeddings. \\
In order to simplify the proof we consider the two special cases $M=M'$, $\Phi=Id_M$ and $N=N'$, $\tilde{\Phi}=Id_N$.
The general case can be obtained combining  these two cases.\\
If $M=M'$ and $\Phi=Id_M$,  putting
$\tilde{Z}=\tilde{\Phi}(Z)$, so that  $Z=\tilde{\Phi}^{-1}(\tilde {Z})$,
by It\^o's formula for semimartingales with jumps, Lemma
\ref{lemma_geometrical1} and Remark \ref{remark_geometrical1} we
have
\begin{eqnarray*}
Z^{\alpha}_t-Z^{\alpha}_0&=&\int_0^t{\partial_{\tilde{z}^{\beta}}(\tilde{\Phi}^{-1})^{\alpha}(\tilde{Z}_{s_-})d\tilde{Z}^{\beta}_s+\frac{1}{2}\partial_{\tilde{z}^{\beta}\tilde{z}^{\gamma}}(\tilde{\Phi}^{-1})^{\alpha}(\tilde{Z}_{s_-}) d[\tilde{Z}^{\beta},\tilde{Z}^{\gamma}]^c_s}+\\
&&+\sum_{0\leq s \leq t}((\tilde{\Phi}^{-1})^{\alpha}(\tilde{Z}_s)-(\tilde{\Phi}^{-1})^{\alpha}(\tilde{Z}_{s_-})-\partial_{\tilde{z}^{\beta}}(\tilde{\Phi}^{-1})^{\alpha}(\tilde{Z}_{s_-})\Delta\tilde{Z}^{\beta}_s)\\
d[Z^{\alpha},Z^{\beta}]^c_t&=&\partial_{\tilde{z}^{\gamma}}(\tilde{\Phi}^{-1})^{\alpha}(Z_{s_-})
\partial_{\tilde{z}^{\delta}}(\tilde{\Phi}^{-1})^{\beta}(Z_{s_-})d[\tilde{Z}^{\gamma},\tilde{Z}^{\delta}]^c_t\\
\Delta
Z^{\alpha}_t&=&(\tilde{\Phi}^{-1})^{\alpha}(\tilde{Z}_t)-(\tilde{\Phi}^{-1})^{\alpha}(\tilde{Z}_{t_-}).
\end{eqnarray*}
The conclusion of  Theorem \ref{theorem_geometrical1} follows using the definition of solution of the geometrical SDE $\overline{\Psi}_{K_t}$, Lemma \ref{lemma_geometrical1} and the chain rule
for derivatives.\\
Suppose now that $N=N'$ and $\tilde{\Phi}=Id_N$. Putting $X'=\Phi(X)$,
by It\^o's formula we obtain
\begin{eqnarray*}
X'^i_t-X'^i_0&=&\int_0^t{\partial_{x^j}(\Phi^i)(X_{s_-})dX^j_s}+\frac{1}{2}\int_0^t{\partial_{x^jx^h}(\Phi^i)(X_{s_-})d[X^j,X^h]^c_s}+\\
&&+\sum_{0\leq s \leq
t}(\Phi^i(X_s)-\Phi^i(X_{s_-})-\partial_{x^j}(\Phi^i)(X_{s_-})\Delta
X^j_s).
\end{eqnarray*}
Furthermore, by definition of solutions to the geometrical SDE
$\overline{\Psi}$ and by  Lemma \ref{lemma_geometrical1} we have
\begin{eqnarray*}
dX^i_s&=&\partial_{z'^{\alpha}}(\overline{\Psi}^i_{K_s})(X_{s_-},Z_{s_-},Z_{s_-})
dZ^{\alpha}_s+\frac{1}{2}\partial_{z'^{\alpha}z'^{\beta}}(\overline{\Psi}_{K_s}^i)(X_{s_-},Z_{s_-},Z_{s_-})d[Z^{\alpha},Z^{\beta}]^c_s +\\
&&+\overline{\Psi}_{K_s}^i(X_{s_-},Z_{s},Z_{s_-})-\overline{\Psi}_{K_s}^i(X_{s_-},Z_{s_-},Z_{s_-})-
\partial_{z'^{\alpha}}(\overline{\Psi}_{K_s}^i)(X_{s_-},Z_{s_-},Z_{s_-})\Delta Z^{\alpha}_s,\\
d[X^i,X^j]^c_s&=&\partial_{z'^{\alpha}}(\overline{\Psi}_{K_s}^i)(X_{s_-},Z_{s_-},Z_{s_-})
\partial_{z'^{\beta}}(\overline{\Psi}_{K_s}^i)(X_{s_-},Z_{s_-},Z_{s_-})d[Z^{\alpha},Z^{\beta}]^c_s\\
\Delta X^i_s&=&\overline{\Psi}_{K_s}^i(X_{s_-},Z_s,Z_{s_-})-\overline{\Psi}_{K_s}^i(X_{s_-},Z_{s_-},Z_{s_-}).
\end{eqnarray*}
Using the previous relations, the fact that $X=\Phi^{-1}(X')$ and the chain rule for derivatives we get the thesis. ${}\hfill$ \end{proof}

\subsection{A comparison with other approaches}\label{subsection_comparison}

Since the geometrical approach of \cite{Cohen1996} is not widely known, but nevertheless it  is essential in our investigation of symmetries, in
this subsection we compare the definition of geometrical SDEs driven by semimartingales with jumps with some more usual definitions of SDEs
driven by c\acc{a}dl\acc{a}g processes appearing in the literature. We make the comparison with different kinds of SDEs with jumps:
\begin{itemize}
\item affine-type SDEs of the type studied in \cite[Chapter
V]{Protter1990} and \cite[Chapter 5]{Bichteler2002}, \item Marcus-type SDEs (see \cite{Protter1995,Marcus1978,Marcus1981}), \item SDEs driven
by L\'evy processes with smooth coefficients (see, e.g., \cite{Applebaum2004,Kunita2004}), \item smooth iterated random functions (see, e.g.,
\cite{Arnold1998,Diaconis1999}).
\end{itemize}
In the following we assume,  for simplicity, that $M$ and $N$ are
open subsets of $\mathbb{R}^m$ and $\mathbb{R}^n$ respectively.

\subsubsection{Affine-type SDEs}

We briefly describe the affine type SDEs as proposed, e.g., in \cite[Chapter V]{Protter1990}. In particular we show how it is possible to rewrite them according to our geometrical setting.\\
Let $(Z^1,...,Z^n)$ be a semimartingale in $N$ and let $\sigma:M
\rightarrow \matr(m,n)$ be a smooth function taking values in the
set of $m \times n$ matrices with real elements. We consider the SDE defined by
\begin{equation}\label{equation_affine}
dX^i_t=\sigma^i_{\alpha}(X_t) dZ^{\alpha}_t,
\end{equation}
where $\sigma^i_j$ are the components of the matrix $\sigma$. If $Z^1_t=t$  and $Z^2,...,Z^n$ are independent Brownian motions, we have the usual diffusion processes with drift $(\sigma^1_1,...,\sigma^m_1)$ and diffusion matrix $(\sigma^i_{\alpha})|_{\stackrel{i=1,...,m}{\alpha=2,...,n}}$.\\
The previous affine-type SDE can be rewritten as a geometrical SDE
defined by the function $\overline{\Psi}$
$$\overline{\Psi}(x,z',z)=x+\sigma(x) \cdot (z'-z),$$
or, in coordinates,
$$\overline{\Psi}^i(x,z',z)=x^i+\sigma^i_{\alpha}(x)(z'^{\alpha}-z^{\alpha}).$$
In fact, by definition of geometrical SDE
$\overline{\Psi}$, we have
\begin{eqnarray*}
X^i_t-X^i_0&=&\int_0^t{\partial_{z'^{\alpha}}(\overline{\Psi}^i)(X_{s_-},Z_{s_-},Z_{s_-})dZ^{\alpha}_s}+\frac{1}{2}\int_0^t{\partial_{z'^{\alpha}z'^{\beta}}(\overline{\Psi}^i)(X_{s_-},Z_{s_-},Z_{s_-})d[Z^{\alpha},Z^{\beta}]_s}\\
&+&\sum_{0\leq s \leq t}\{\overline{\Psi}^i(X_{s_-},Z_{s},Z_{s_-})-\overline{\Psi}^i(X_{s_-},Z_{s_-},Z_{s_-})-\partial_{z'^{\alpha}}(\overline{\Psi}^i)(X_{s_-},Z_{s_-},Z_{s_-})\Delta Z^{\alpha}_s\},\\
&=&\int_0^t{\sigma^i_{\alpha}(X_{s_-})dZ^{\alpha}_s}+\sum_{0\leq s \leq t}\{\sigma^i_{\alpha}(X_{s_-})(Z^{\alpha}_s-Z^{\alpha}_{s_-})-\sigma^i_{\alpha}(X_{s_-})\Delta Z^{\alpha}_s\}\\
&=&\int_0^t{\sigma^i_{\alpha}(X_{s_-})dZ^{\alpha}_s}.
\end{eqnarray*}

\subsubsection{Marcus-type SDEs}

The Marcus-type SDEs with jumps, initially  proposed by Marcus in \cite{Marcus1978,Marcus1981} for semimartingales with finitely many jumps in
any compact interval, have been extended to the case of general real semimartingales in \cite{Protter1995}. The special property of this family of
SDEs is their natural behaviour with respect to diffeomorphisms. \\
Given a manifold $M$ and a global cartesian coordinate system $x^i$ on $M$, we consider $n$ smooth vector fields $Y_1,...,Y_n$ on $M$ of the form  $Y_{\alpha}=Y^i_{\alpha} \partial_{x^i}, \alpha=1,...,n$. If the
functions $Y^i_{\alpha}$ grow at most linearly at infinity, the flow of $Y_{\alpha}$ is defined for any time. Therefore,  for any
$z=(z^1,...,z^n) \in \mathbb{R}^n$,  we introduce the function
$$\Psi(x,z)=\exp(z^{\alpha}Y_{\alpha})(x),$$
where $\exp(Y)$ is the exponential map with respect to the vector field $Y$, i.e. the map associating with  any $x \in M$ its  evolute at  time $1$  with respect to the flow defined by the vector field $Y$.\\
The solution $X $ with values in $ M$ (we shall shortly write $X\in M$) to the Marcus-type SDE defined by the
vector fields $Y_1,...,Y_n$ with respect to the semimartingales
$(Z^1,...,Z^n)$ is the unique semimartingale $X \in M$ such that
\begin{eqnarray*}
X^i_t-X^i_0&=&\int_0^t{Y^i_{\alpha}(X_{s_-})dZ^{\alpha}_s}+\frac{1}{4}\int_0^t{(Y_{\beta}(Y^i_{\alpha})(X_{s_-})+Y_{\alpha}(Y^i_{\beta})(X_{s_-}))d[Z^{\alpha},Z^{\beta}]_s}+\\
&&\sum_{0 \leq s \leq
t}\{\Psi^i(X_{s_-},Z_s-Z_{s_-})-X^i_{s_-}-Y^i_{\alpha}(X_{s_-})\Delta
Z^{\alpha}_s\}.
\end{eqnarray*}
We note that the previous equation depends only on $Y_1,...,Y_n$, which means that if $\Phi:M \rightarrow M'$ is a diffeomorphism, the semimartingale $\Phi(X)$ solves the Marcus-type SDE defined by the vector fields $\Phi_*(Y_1),...,\Phi_*(Y_n)$ (see \cite{Protter1995}). \\
The Marcus-type SDE is a special form of geometrical SDE with
defining map given by
$$\overline{\Psi}(x,z',z)=\Psi(x,z'-z).$$
Indeed, by definition of $\Psi$ and $\overline{\Psi}$, we have
\begin{eqnarray*}
&\partial_{z'^{\alpha}}(\overline{\Psi}^i)(x,z,z)=\partial_{z^{\alpha}}(\Psi^i)(x,0)=Y^i_{\alpha}&\\
&\partial_{z'^{\alpha}z'^{\beta}}(\overline{\Psi}^i)(x,z,z)=\partial_{z^{\alpha}z^{\beta}}(\Psi^i)(x,0)=\frac{1}{2}(Y_{\beta}(Y^i_{\alpha})+Y_{\alpha}(Y_{\beta}^i)).&
\end{eqnarray*}

\subsubsection{Smooth SDEs driven by a L\'evy process}\label{subsubsection_Levy}

In this section we describe a particular form of SDEs driven by $\mathbb{R}^n$-valued L\'evy processes (see, e.g.,
\cite{Applebaum2004,Kunita2004}). By definition, an $\mathbb{R}^n$-valued L\'evy process $(Z^1,...,Z^n)$ can be decomposed into the sum of Brownian
motions and compensated Poisson processes defined on $\mathbb{R}^n$. In particular, a L\'evy process on $\mathbb{R}^n$ can be identified by a
vector $b_0=(b_0^1,...,b_0^n) \in \mathbb{R}^n$, an $n \times n$ matrix $A_0^{\alpha\beta}$ (with real elements) and a positive $\sigma$-finite measure $\nu_0$
defined on $\mathbb{R}^n$ (called L\'evy measures, see, e.g., \cite{Applebaum2004,Sato1999}) such that
$$\int_{\mathbb{R}^n}{\frac{|z|^2}{1+|z|^2}\nu_0(dz)} < + \infty.$$
By  the L\'evy-It\^o decomposition, the triplet $(b,A,\nu)$ is such that there exist an $n$ dimensional Brownian motion $(W^1,...,W^n)$ and a
Poisson measure $P(dz,dt)$ defined on $\mathbb{R}^n$ such that
\begin{eqnarray*}
Z^{\alpha}_t&=&b^{\alpha}_0 t+ C^{\alpha}_{\beta} W^{\beta}_t+\int_0^t{\int_{|z| \leq 1}{z^{\alpha}(P(dz,ds)-\nu_0(dz)ds)}}+\\
&&+\int_0^t{\int_{|z|>1}{z^{\alpha}P(dz,ds)}}.
\end{eqnarray*}
where $A^{\alpha\beta}_0=\sum_{\gamma}C^{\alpha}_{\gamma}C^{\beta}_{\gamma}$. Henceforth we suppose for simplicity that $b^1=1$ and
$b^{\alpha}_0=0$ for $\alpha>1$, that there exists
 $n_1$ such that $A^{\alpha\beta}_0=\delta^{\alpha\beta}$ for $1< \alpha ,\beta \leq n_1$ and
 $A^{\alpha\beta}_0=0$ for $\alpha$ or $\beta$ in $\{1,n_1+1,...,n\}$, and finally  that
 $\int_0^t{\int_{|z| \leq 1}{z^{\alpha}(P(dz,ds)-\nu_0(dz)ds)}}=0$ and $\int_0^t{\int_{|z|>1}{z^{\alpha}P(dz,ds)}}=0$ for $\alpha \leq n_1$. \\
Consider a vector field $\mu$ on $M$, a set of $n_1-1$ vector
fields $\sigma=(\sigma_2,...,\sigma_{n_{1}})$ on $M$ and a smooth
(both in $x $ and $z$) function $F:M \times \mathbb{R}^{n-n_1}
\rightarrow \mathbb{R}^m$ such that $F(x,0)=0$. We say that a
semimartingale $X \in M$ is a solution to the smooth SDE
$(\mu,\sigma,F)$ driven by the $\mathbb{R}^n$ L\'evy process
$(Z^1,...,Z^n)$ if
\begin{eqnarray*}
X^i_t-X^i_0&=&\int_0^t{\mu^i(X_{s_-})dZ^1_s}+\int_0^t{\sum_{\alpha=2}^{n_1}\sigma^i_{\alpha}(X_{s_-})dZ_s^{\alpha}}+\\
&&+\int_0^t{\int_{\mathbb{R}^{n-n_1}}{F^i(X_{s_-},z)(P(dz,ds)-I_{|z|\leq 1}\nu_0(dz)ds)}},
\end{eqnarray*}
where $I_{|z|\leq 1}$ is the indicator function of the set
$\{|z|\leq 1\}\subset \mathbb{R}^{n-n_1}$. Define the function
$$\overline{\Psi}^i(x,z',z)=x^i+\tilde{\mu}^i(x)(z'^1-z^1)+\sigma^i_{\alpha}(x)(z'^{\alpha}-z^{\alpha})+F^i(x,z'-z),$$
where
$$\tilde{\mu}^i(x)=\mu^i(x)-\int_{|z|\leq 1}{(F^i(x,z)-\partial_{z^{\alpha}}(F^i)(x,z)z^{\alpha})\nu_0(dz)}.$$
It is easy to see that any solution $X$ to the smooth SDE
$(\mu,\sigma,F)$ driven by the L\'evy process $(Z^1,...,Z^n)$ is
also solution to the geometrical SDE $\overline{\Psi}$ driven by
the $\mathbb{R}^n$ semimartingale $(Z^1,...,Z^n)$ and conversely.

\begin{remark}
In the theory of SDEs driven by $\mathbb{R}^n$-valued L\'evy processes the usual assumption  is
that $F$ is Lipschitz in $x$ and measurable in $z$. Our assumption on smoothness of $F$ in both $x,z$ is thus  a stronger requirement. For this reason
we say that $(\mu,\sigma,F)$ is a \emph{smooth SDE} driven by a L\'evy process.
\end{remark}

\subsubsection{Iterated random smooth functions}\label{subsection_iterated}

In the previous sections we have only  considered continuous time processes $Z_t$. Let us now take $Z$ as a discrete time
adapted process, i.e. $Z$ is a sequence of random variables $Z_0,Z_1,...,Z_n,...$ defined on $N$. We can consider $Z$ as a c\acc{a}dl\acc{a}g
continuous time process $Z_t$ defined by
$$Z_t=Z_n \text{ if } n \leq t < n+1.$$
Since the process $Z$ is a pure jump process with a finite number of jumps in any compact interval of $\mathbb{R}_+$, $Z$ is a
semimartingale. If $(X,Z)$ is a solution of the geometrical SDE  $\overline{\Psi}$, we have that
\begin{equation}\label{equation_difference}
X_n=\overline{\Psi}(X_{n-1},Z_n,Z_{n-1})
\end{equation}
and $X_t=X_n$ if $n \leq t < n+1$. The process $X$ can be viewed as a discrete time process defined by the recursive relation
\refeqn{equation_difference}. These processes are special forms of \emph{iterated random functions} (see, e.g.,
\cite{Arnold1998,Diaconis1999,Schreiber2012}) and  this kind of equations is very important in time series analysis (see, e.g.,
\cite{Brockwell1991,Shumway2006}) and in numerical simulation of SDEs (see, e.g., \cite{Kloeden1992} for simulation of SDEs and
\cite{DeVecchi2017} for the concept of strong symmetry of a discretization scheme). In this case we do not need that $\overline{\Psi}$ is smooth
in all its variables and that $\overline{\Psi}(x,z,z)=x$ for any $x \in M$ and $z \in N$. In the case of a discrete time semimartingale $Z_t$
these two conditions can be skipped and
we can consider more general iterated random functions defined by relation \refeqn{equation_difference}.\\
An important example of iterated random functions can be obtained by considering $M=\mathbb{R}^m$, $N=GL(m) \times
\mathbb{R}^m$ and the functions
$$\overline{\Psi}(x,z',z)=(z'_1 \cdot z_1^{-1}) \cdot x + (z'_2-z_2),$$
where $(z_1,z_2) \in GL(m) \times \mathbb{R}^m$. Moreover, taking  two sequences of random variables $A_0,...,A_n,... \in GL(n)$ and
$B_0,...,B_n,... \in \mathbb{R}^m$, we  define
$$Z_n=\left(A_n \cdot A_{n-1} \cdot ....\cdot A_0, B_n +B_{n-1} + ....+ B_0  \right).$$
The iterated random functions associated with the SDE $\overline{\Psi}$ is
$$X_n=A_n \cdot X_{n-1}+B_{n}.$$
This model is very well studied (see, e.g., \cite{Arnold1998,Babillot1997,Kesten1973}). In particular the well known  ARMA model is of this form
(see, e.g., \cite{Brockwell1991,Shumway2006}).

\subsection{Canonical SDEs}

In this section, in order to generalize  the well known noise change property of affine-type SDEs driven by c\acc{a}dl\acc{a}g semimartingales,
we introduce  the concept of canonical SDEs driven by a process on a Lie group $N$. If $M=\mathbb{R}^m$ and $N=\mathbb{R}^n$ and we consider the
affine SDE given by
$$dX^i_t=\sigma^i_{\alpha}(X_{t_-}) dZ^{\alpha}_t,$$
we can define a new semimartingale on $N$ given by
\begin{equation}\label{equation_semimartingales_change1}
d\tilde{Z}^{\alpha}_t=B^{\alpha}_{\beta,t}dZ^{\beta}_t,
\end{equation}
where $B=(B^{\alpha}_{\beta})$ is a locally bounded predictable
process taking values in $GL(n)$, and  rewrite the affine SDE in terms of the
semimartingale $\tilde{Z}$ in the following way
\begin{equation}\label{equation_property}
dX^i_t=\sigma^{i}_{\alpha}(X_{t_-})(B^{-1})^{\alpha}_{\beta,t} d
\tilde{Z}^{\beta}_t,
\end{equation}
where $B^{-1}$ is the inverse matrix of $B$. Since this property, essential in the
definition of symmetries of a canonical SDE,  has no counterpart for general geometrical  SDEs,  we restrict our attention to a special class of geometrical SDEs that we call canonical (geometrical) SDEs.
The first three kinds of SDEs proposed in Subsection \ref{subsection_comparison} are canonical
SDEs in the above sense.\\
Considering now a (general) Lie group $N$ and a semimartingale $Z$ on $N$,
a natural definition of jump can be given.
Indeed, if $\tau$ is a stopping time, we define the jump at time
$\tau$ as the random variable $\Delta Z_{\tau}$ taking values on
$N$ such that
$$\Delta Z_{\tau}=Z_{\tau} \cdot (Z_{\tau_-})^{-1},$$
where $\cdot$ is the multiplication in the group $N$. In order to define a special class of equations that, in some sense,
depends only on the
jumps $\Delta Z_t$ of a process $Z$ defined on a Lie group, we consider a function $\Psi$ of the form
$$\Psi_{\cdot}(\cdot,\cdot):M \times N \times \mathcal{K} \rightarrow M,$$
such that $\Psi_k(x,1_N)=x$ for any $k$ in a metric space $\mathcal{K}$ and $x \in M$, and we introduce the function $\overline{\Psi}_k$ defining the corresponding
 geometrical SDE as
$$\overline{\Psi}_k(x,z',z)=\Psi_k(x,z'\cdot z^{-1})={\Psi}_k(x,\Delta z).$$
If $(X,Z)$ solves the SDE defined by this $\overline{\Psi}_k$, we write
$$dX_t=\Psi_{K_t}(d Z_t),$$
and we say that $(X,Z)$ is a solution to the canonical SDE  $\Psi_{K_t}$.
For canonical SDEs it is possible to consider a sort of generalization of the semimartingales change rule \refeqn{equation_property}.\\
Suppose that $M=\tilde{N}$ for some Lie group $\tilde{N}$ and
consider the smooth function
$$\Xi_{\cdot}(\cdot):N \times \mathcal{G} \rightarrow \tilde{N},$$
where $\mathcal{G}$ is a Lie group, which satisfies the
relation $\Xi_g(1_N)=1_{\tilde{N}}, \forall g \in \mathcal{G}$. We define the map
$$\tilde{\Psi}_g(x,z)=\Xi_g(z) \cdot x.$$
If $Z$ is a semimartingale on $N$,  we define the transformed semimartingale on $\tilde{N}$ by
\begin{equation}\label{equation_semimartingales_change2}
d\tilde{Z_t}=\Xi_{G_t}(dZ_t)
\end{equation}
as the unique solution $(\tilde{Z},Z)$  to the equation
$$d\tilde{Z}_t=\tilde{\Psi}_{G_t}(d Z_t),$$
with initial condition $\tilde{Z}_0=1_{\tilde{N}}$. Before proving further results about  transformation \refeqn{equation_semimartingales_change2}, we
 show that the semimartingales change \refeqn{equation_semimartingales_change1} is a particular case of
\refeqn{equation_semimartingales_change2}. In fact, for $\tilde{N}=N=\mathbb{R}^n$, any map $\Xi_{\cdot}:\mathbb{R}^n \times \mathcal{G}
\rightarrow \mathbb{R}^n$ gives the canonical SDE defined by the function
$$\tilde{\Psi}_g(\tilde{z},z)=\tilde{z}+\Xi_g(z).$$
This means that equation \refeqn{equation_semimartingales_change2} is explicitly given by the relation
\begin{equation}\label{equation_semimartingales_change3}
\begin{array}{ccl}
\tilde{Z}_t&=&\int_0^t{\partial_{z^{\alpha}}(\Xi_{G_s})(0)dZ^{\alpha}_s}+\frac{1}{2}\int_0^t{\partial_{z^{\alpha}z^{\beta}}(\Xi_{G_s})(0)d[Z^{\alpha},Z^{\beta}]_s^c}+\\
&&+\sum_{0\leq s \leq t}(\Xi_{G_s}(\Delta
Z_s)-\partial_{z^{\alpha}}(\Xi_{G_s})(0)\Delta Z^{\alpha}_s).
\end{array}
\end{equation}
If $\mathcal{G}=GL(n)$ and $\Xi_g(z)=\Xi_B(z)=B \cdot z$, since both $\partial_{z^{\alpha}z^{\beta}}(\Xi_{B_t})(0)$ and $(\Xi_{G_s}(\Delta
Z_s)-\partial_{z^{\alpha}}(\Xi_{G_s})(0)\Delta Z^{\alpha}_s)$ are equal to zero, we obtain equation \refeqn{equation_semimartingales_change1}.

\begin{remark}
When $N=\mathbb{R}^n$ the
right-hand side of equation
\refeqn{equation_semimartingales_change3} does not depend on
$\tilde{Z}$.
\end{remark}

\begin{theorem}\label{theorem_gauge1}
Let $N,\tilde{N}$ be two Lie groups and suppose that
$(X,\tilde{Z})$ (where $\tilde{Z}$ is defined on $\tilde{N}$) is a
solution to the canonical SDE $\Psi_{K_t}$. If
$d\tilde{Z}_t=\Xi_{G_t}(dZ_t)$,  then $(X,Z)$ is a solution to the
canonical SDE defined by
$$\hat{\Psi}_{k,g}(x,z)=\Psi_k(x,\Xi_g(z)).$$
\end{theorem}
\begin{proof}
We prove the theorem when $N,\tilde{N},M$ are open subsets of $\mathbb{R}^{m},\mathbb{R}^{n}$.
The proof of the  general case can be obtained by using suitable embeddings. \\
Let $x^i$, $z^{\alpha}$ and $\tilde{z}^{\alpha}$ be some global
coordinate systems of $M,N, \tilde{N}$ respectively. By
definition $\tilde{Z}$ is such that
\begin{eqnarray*}
\tilde{Z}^{\alpha}_t-\tilde{Z}^{\alpha}_0&=&
\int_0^t{\partial_{z'^{\beta}}(\overline{\Xi}^{\alpha}_{G_s})(\tilde{Z}_{s_-},Z_{s_-},Z_{s_-})dZ^{\beta}_s+\frac{1}{2}
\partial_{z'^{\beta}z'^{\gamma}}(\overline{\Xi}^{\alpha}_{G_s})(\tilde{Z}_{s_-},Z_{s_-},Z_{s_-})d[Z^{\beta},Z^{\gamma}]^c_s}\\
&&+\sum_{0\leq s \leq t}\overline{\Xi}_{G_s}^{\alpha}(\tilde{Z}_{s_-},
Z_s,Z_{s_-})-\overline{\Xi}^{\alpha}_{G_s}(\tilde{Z}_{s_-},Z_{s_-},Z_{s_-})-
\partial_{z'^{\beta}}(\overline{\Xi}^{\alpha}_{G_s})(\tilde{Z}_{s_-},Z_{s_-},Z_{s_-})\Delta Z^{\beta}_s,
\end{eqnarray*}
where $\overline{\Xi}_g(\tilde{z},z',z)=\Xi_g(z'\cdot z^{-1})
\cdot \tilde{z}$. By the previous equation, Lemma
\ref{lemma_geometrical1} and Remark \ref{remark_geometrical1} we
obtain
\begin{eqnarray*}
[\tilde{Z}^{\alpha},\tilde{Z}^{\beta}]_t&=&\int_0^t{\partial_{z'^{\gamma}}
(\overline{\Xi}^{\alpha}_{G_s})(\tilde{Z}_{s_-},Z_{s_-},Z_{s_-})\partial_{z'^{\delta}}
(\overline{\Xi}^{\beta}_{G_s})(\tilde{Z}_{s_-},Z_{s_-},Z_{s_-})d[Z^{\gamma},Z^{\delta}]^c_s}\\
\Delta \tilde{Z}^{\alpha}_t&=&\overline{\Xi}^{\alpha}_{G_t}(\tilde{Z}_{t_-},
Z_t,Z_{t_-})-\overline{\Xi}^{\alpha}_{G_t}(\tilde{Z}_{t_-},Z_{t_-},Z_{t_-}).
\end{eqnarray*}
Therefore, since $(X,\tilde{Z})$ is a solution to the canonical SDE $\Psi_{K_t}$, using Lemma
\ref{lemma_geometrical1} and Remark \ref{remark_geometrical1}, we have
\small
\begin{eqnarray*}
X^i_t-X^i_0&=&
\int_0^t{\partial_{\tilde{z}'^{\alpha}}(\overline{\Psi}^i_{K_s})(X_{s_-},\tilde{Z}_{s_-},\tilde{Z}_{s_-})d\tilde{Z}^{\alpha}_s+\frac{1}{2}
\partial_{\tilde{z}'^{\alpha}\tilde{z}'^{\beta}}(\overline{\Psi}^i_{K_s})(X_{s_-},\tilde{Z}_{s_-},\tilde{Z}_{s_-})d[\tilde{Z}^{\alpha},\tilde{Z}^{\beta}]_s}\\
&&+\sum_{0\leq s \leq t}\{\Psi^i_{K_s}(X_{s_-},\Delta \tilde{Z}_s)-\Psi^i_{K_s}(X_{s_-},1_N)-
\partial_{\tilde{z}'^{\alpha}}(\overline{\Psi}^i_{K_s})(X_{s_-},\tilde{Z}_{s_-},\tilde{Z}_{s_-})\Delta \tilde{Z}^{\alpha}_s\}\\
&=&\int_0^t{\partial_{\tilde{z}'^{\alpha}}(\overline{\Psi}^i_{K_s})(X_{s_-},\tilde{Z}_{s_-},\tilde{Z}_{s_-})\partial_{z'^{\beta}}
(\overline{\Xi}^{\alpha}_{G_s})
(\tilde{Z}_{s_-},Z_{s_-},Z_{s_-})dZ^{\beta}_s}\\
&&+\frac{1}{2}\int_0^t{\partial_{\tilde{z}'^{\alpha}}(\overline{\Psi}^i_{K_s})(X_{s_-},\tilde{Z}_{s_-},\tilde{Z}_{s_-})
\partial_{z'^{\beta}z'^{\gamma}}(\overline{\Xi}^{\alpha}_{G_s})(\tilde{Z}_{s_-},Z_{s_-},Z_{s_-})d[Z^{\beta},Z^{\gamma}]^c_s}+\\
&&+\sum_{0\leq s \leq t}\partial_{\tilde{z}'^{\alpha}}(\overline{\Psi}^i_{K_s})(X_{s_-},\tilde{Z}_{s_-},\tilde{Z}_{s_-})
(\overline{\Xi}_{G_s}^{\alpha}(\tilde{Z}_{s_-},Z_s,Z_{s_-})-\overline{\Xi}_{G_s}(\tilde{Z}_{s_-},Z_{s_-},Z_{s_-})+\\
&&-\partial_{z'^{\beta}}(\overline{\Xi}^{\alpha}_{G_s})(\tilde{Z}_{s_-},Z_{s_-},Z_{s_-})\Delta Z^{\beta}_s)+\\
&&+\frac{1}{2}\int_0^t{\left(\partial_{\tilde{z}'^{\alpha}\tilde{z}'^{\delta}}(\overline{\Psi}^i_{K_s})(X_{s_-},\tilde{Z}_{s_-},\tilde{Z}_{s_-})
\partial_{z'^{\beta}}(\overline{\Xi}^{\alpha}_{G_s})(\tilde{Z}_{s_-},Z_{s_-},Z_{s_-})\cdot\right.}\\
&&\left.\cdot\partial_{z'^{\beta}}(\overline{\Xi}^{\delta}_{G_s})(\tilde{Z}_{s_-},Z_{s_-},Z_{s_-})\right)d[Z^{\beta},Z^{\gamma}]_{s}+\sum_{0
\leq s \leq t}\left(
\Psi^i(X_{s_-},\Xi_{G_s}(\Delta Z_s))-\Psi^i(X_{s_-},1_N)+\right.\\
&&\left.-\partial_{\tilde{z}'^{\alpha}}(\overline{\Psi}^i_{K_s})(X_{s_-},\tilde{Z}_{s_-},\tilde{Z}_{s_-})(\overline{\Xi}_{G_s}^{\alpha}
(\tilde{Z}_{s_-}, Z_s,Z_{s_-})-\overline{\Xi}_{G_s}^{\alpha}(\tilde{Z}_{s_-},Z_{s_-},Z_{s_-}))\right).
\end{eqnarray*}
\normalsize
By the chain rule for derivatives and the fact that
$\overline{\Xi}^{\alpha}(\tilde{Z}_{s_-},Z_{s_-},Z_{s_-})=\tilde{Z}^{\alpha}_{s_-}$ we have
\small
$$
\left.\partial_{z'^{\beta}}(\overline{\Psi}^i_{K_s}(x,\overline{\Xi}_{G_s}(\tilde{z},z',z),\tilde{z})) \right|_{\stackrel{
x=X_{s_-},\tilde{z}=\tilde{Z}_{s_-}}{z=z'=Z_{s_-}}}=
\partial_{\tilde{z}'^{\alpha}}(\overline{\Psi}^i_{K_s})(X_{s_-},\tilde{Z}_{s_-},\tilde{Z}_{s_-})\partial_{z'^{\beta}}(\overline{\Xi}^{\alpha}_{G_s})
(\tilde{Z}_{s_-},Z_{s_-},Z_{s_-})
$$
\begin{eqnarray*}
&\left.\partial_{z'^{\beta}z'^{\gamma}}(\overline{\Psi}^i_{K_s}(x,\overline{\Xi}_{G_s}(\tilde{z},z',z),\tilde{z})) \right|_{\stackrel{
x=X_{s_-},\tilde{z}=\tilde{Z}_{s_-}}{z=z'=Z_{s_-}}}=
\partial_{\tilde{z}'^{\alpha}}(\overline{\Psi}^i_{K_s})(X_{s_-},\tilde{Z}_{s_-},\tilde{Z}_{s_-})
\partial_{z'^{\beta}z'^{\gamma}}(\overline{\Xi}^{\alpha}_{G_s})(\tilde{Z}_{s_-},Z_{s_-},Z_{s_-})+&\\
&+\partial_{\tilde{z}'^{\alpha}\tilde{z}'^{\delta}}(\overline{\Psi}^i_{K_s})(X_{s_-},\tilde{Z}_{s_-},\tilde{Z}_{s_-})
\partial_{z'^{\beta}}(\overline{\Xi}^{\alpha}_{G_s})(\tilde{Z}_{s_-},Z_{s_-},Z_{s_-})\partial_{z'^{\gamma}}(\overline{\Xi}^{\delta}_{G_s})(\tilde{Z}_{s_-},Z_{s_-},Z_{s_-})&
\end{eqnarray*}
\normalsize
Using the fact that
\small
$$\overline{\Psi}_k(x,\overline{\Xi}_g(\tilde{z},z',z),\tilde{z})=\Psi_k(x,(\Xi_g(z' \cdot z^{-1})\cdot \tilde{z})\cdot \tilde{z}^{-1})=\hat{\Psi}_{k,g}(x,z' \cdot z^{-1})=\Psi_k(x,\Xi_g(z'\cdot z^{-1}))=
\overline{\hat{\Psi}}_{k,g}(x,z',z)$$ \normalsize  we obtain \footnotesize
\begin{eqnarray*}
X^i_t-X^i_0&=&\int_0^t{\partial_{z'^{\beta}}(\overline{\hat{\Psi}}^i_{K_s,G_s})(X_{s_-},Z_{s_-},Z_{s_-})dZ^{\beta}_s+\frac{1}{2}
\partial_{z'^{\beta}z'^{\gamma}}(\overline{\hat{\Psi}}^i_{K_s,G_s})(X_{s_-},Z_{s_-},Z_{s_-})d[Z^{\beta},Z^{\gamma}]_s}+\\
&&+\sum_{0\leq s \leq t}\overline{\hat{\Psi}}_{K_s,G_s}^i(X_{s_-},Z_{s},Z_{s_-})-\hat{\Psi}_{K_s,G_s}^i(X_{s_-},Z_{s_-},Z_{s_-})-
\partial_{z'^{\beta}}(\overline{\hat{\Psi}}^i_{K_s,G_s})(X_{s_-},Z_{s_-},Z_{s_-})\Delta Z^{\beta}_s,
\end{eqnarray*}
\normalsize and so $dX_t=\hat{\Psi}_{K_t,G_t}(dZ_t)$. ${}\hfill$ \end{proof}

\begin{corollary}\label{corollary_gauge1}
Suppose that $\mathcal{G}$ is a Lie group and  $\Xi$
is a Lie group action. If $(X,Z)$ is a solution to the
canonical SDE $\Psi_{K_t}$, then $(X,\tilde{Z})$ is a solution to
the canonical SDE defined by
$$\hat{\Psi}_{k,g}(x,z)=\Psi_k(x,\Xi_{g^{-1}}(z)).$$
\end{corollary}
\begin{proof}
The proof is an application of Theorem \ref{theorem_gauge1} and of the fact that $dZ_t=\Xi_{G_t^{-1}}(d\tilde{Z}_t)$. Indeed, defining
$d\hat{Z}_t=\Xi_{G_t^{-1}}(d\tilde{Z}_t)$, by  Theorem \ref{theorem_gauge1} we have that $d\hat{Z}_t=\Xi_{G_t^{-1}} \circ \Xi_{G_t}
(dZ_t)=\Xi_{1_{\mathcal{G}}}(dZ_t)=dZ_t$. The corollary follows directly from Theorem \ref{theorem_gauge1}. ${}\hfill$ \end{proof}

\section{Gauge symmetries of semimartingales on Lie groups}\label{section_gauge}

\subsection{Definition of gauge symmetries}

Let us consider the following well known
property of  Brownian motion. Consider a Brownian
motion $Z$ on $\mathbb{R}^n$ and let $B_t: \Omega \times [0,T]
\rightarrow O(n)$ be a predictable process, with respect to the natural filtration of $Z$, 
taking values in the Lie group $ O(n)$ of orthogonal matrices. Then the process defined by
\begin{equation}\label{equation_gauge1}
Z'^{\alpha}_t=\int_0^t{B^{\alpha}_{\beta,s}dZ^{\beta}_s}
\end{equation}
is a new $n$ dimensional Brownian motion.\\
We propose a generalization of this property to the case in which $Z$ is a c\acc{a}dl\acc{a}g semimartingale in a Lie group $N$ (see
\cite{Privault2012} for a similar result about Poisson measures). In the simple case $N=\mathbb{R}^n$, by replacing the Brownian motion with a
general semimartingale, the invariance property \refeqn{equation_gauge1}  is no longer true. So we need
\begin{itemize}
\item a method to generalize the integral relation to the case where $Z$ is no more a process on $\mathbb{R}^n$ and the Lie group valued
process is no more the $O(n)$-valued process $B$, \item a class of
semimartingales on a Lie group $N$ such that the generalization of the integral relation \refeqn{equation_gauge1} holds.
\end{itemize}

\begin{definition}\label{definition_gauge}
Let $Z$ be a semimartingale on a Lie group $N$ with respect to the filtration $\mathcal{F}_t$. Given a Lie group $\mathcal{G}$  and an element $ g\in \mathcal{G}$, we say that $Z$ admits $\mathcal{G}$, with action $\Xi_g$ and with respect to the filtration $\mathcal{F}_t$,  as \emph{gauge symmetry group} if, for any $\mathcal{F}_t$-predictable locally bounded process $G_t$ taking values in $\mathcal{G}$, the semimartingale $\tilde{Z}$ solution to the equation $d\tilde{Z}_t=\Xi_{G_t}(dZ_t)$ has the same law as $Z$.
\end{definition}

In the following we consider that the filtration $\mathcal{F}_t$ of the  probability space $(\Omega,\mathcal{F},\mathbb{P})$  is given and we omit to mention it if it is not strictly necessary.\\
Since  $\tilde{Z}$ in  Definition \ref{definition_gauge} solves the canonical equation
$d\tilde{Z}_t=\Xi_{G_t}(dZ_t)$, for all times, we are interested in characterizing  SDEs of the previous form with explosion time equal to $+ \infty$ for any $G_t$.
The following proposition gives us  a sufficient condition on the group $N$ such that, for any action $\Xi_g$, the corresponding canonical SDE has indeed
explosion time $+ \infty$.

\begin{proposition}\label{proposition_faithful}
Suppose that $N$ admits a faithful representation. Then, for any
locally bounded process $G_t$ in $\mathcal{G}$, the explosion time
of the SDE $d\tilde{Z}_t=\Xi_{G_t}(dZ_t)$ is $+
\infty$.
\end{proposition}
\begin{proof}
Let $K:N \rightarrow \matr(l_N,l_N)$ be a faithful representation
of $N$. In this representation, the geometrical SDEs associated
with $\Xi_g$, is defined by the map $\overline{\Xi}_g$ given by
$$\overline{\Xi}_g(\tilde{z},z',z)=K(\Xi_g(z' \cdot z^{-1})) \cdot K(\tilde{z}),$$
where $\cdot$ on the right-hand side denotes the usual matrix multiplication. If $k^i$ is the standard cartesian coordinate system in
$\matr(l_N,l_N)$, extending suitably $\Xi_g$ to all $\matr(l_N,l_N)$, we have that $\overline{\Xi}^i(\tilde{k},k',k)$,
$\partial_{k'^j}(\overline{\Xi}^i_g)(\tilde{k},k',k)$ and $\partial_{k'^jk'^l}(\overline{\Xi}^i_g)(\tilde{k},k',k)$ are linear in $\tilde{k}$.
So, putting $Z^i=k^i(Z)$ and $\tilde{Z}^i=k^i(\tilde{Z})$, the SDE
\begin{eqnarray*}
\tilde{Z}^i_t&=&K^i_0+\int_0^t{\partial_{k'^j}(\overline{\Xi}^i_{G_s})(\tilde{Z}_{s_-},Z_{s_-},Z_{s_-}) dZ^j_{s}}+\\
&&+\frac{1}{2}\int_0^t{\partial_{k'^jk'^l}(\overline{\Xi}^i_{G_s})(\tilde{Z}_{s_-},Z_{s_-},Z_{s_-})d[Z^j,Z^l]_s}+\\
&&+\sum_{0 \leq s \leq t}(\overline{\Xi}^i_{G_s})(\tilde{Z}_{s_-},Z_{s},Z_{s_-})-\tilde{Z}^i_s-\Delta Z_s^j
\partial_{k'^j}(\overline{\Xi}^{i}_{G_s})(\tilde{Z}_{s_-},Z_{s_-},Z_{s_-})),
\end{eqnarray*}
is linear in $\tilde{Z}$ and so, by well known results on SDEs with jumps in $\mathbb{R}^{l_N^2}$ (see, e.g., \cite[Chapter 5]{Bichteler2002}) the
solution has explosion time $\tau=+ \infty$ almost surely. ${}\hfill$ \end{proof}\\

In order to provide  a method to construct semimartingales  admitting gauge symmetry groups, we start by showing how it is possible to obtain, starting from martingales with gauge symmetries, new semimartingales with different gauge symmetries.

\begin{proposition}\label{proposition_gauge2}
Given two Lie groups $N$ and $N'$, let $Z$ be a semimartingale on $N$ with gauge symmetry group $\mathcal{G}$ and action $\Xi_g$. If $\Theta:N
\rightarrow N'$ is a diffeomorphism from $N$ onto $N'$ such that $\Theta(1_N)=1_{N'}$, then $d\tilde{Z}_t=\Theta(dZ_t)$ has gauge symmetry group
$\mathcal{G}$ with action $\Theta \circ \Xi_g \circ \Theta^{-1}$.
\end{proposition}
\begin{proof}
By Corollary \ref{corollary_gauge1}  $dZ_t=\Theta^{-1}(\tilde{dZ}_t)$, and since $Z$ has gauge symmetry group $\mathcal{G}$ with action $\Xi_g$,
by Theorem \ref{theorem_gauge1}, $\Xi_{G_t}(dZ_t)=\Xi_{G_t} \circ \Theta^{-1}(d\tilde{Z}_t)$ has the same distribution as $Z$ for any locally
bounded predictable process $G_t$. Moreover, by the uniqueness of the strong solution to a geometrical SDE, we have that $\Theta(\Xi_{G_t} \circ
\Theta^{-1}(d\tilde{Z}_t))= \Theta \circ \Xi_{G_t} \circ \Theta^{-1}(d\tilde{Z}_t)$ has the same distribution as $\tilde{Z}$. \hfill\end{proof}\\

In the following, in order to provide some explicit methods  to verify that a semimartingale on a Lie group $N$ has the gauge symmetry group
$\mathcal{G}$ with action $\Xi_g$,  we introduce the concept of characteristics of a semimartingale on a Lie
group. This allows us to formulate  a condition, equivalent to Definition \ref{definition_gauge}, that can be
directly applied to L\'evy processes on Lie groups providing a completely deterministic method to verify Definition \ref{definition_gauge} in this
case. Then we shall use this reformulation  to provide some examples of non-Markovian processes admitting  gauge
symmetry groups.\\

\subsection{Characteristics of a Lie group valued semimartingale}

In this section we  extend the well known concept of semimartingale characteristics  from the $\mathbb{R}^n$ setting
to the case of a semimartingale defined on a general finite dimensional Lie group $N$. \\
Given $n$ generators  $Y_1,...,Y_n$  of right-invariant vector fields on $N$ providing a global trivialization of the
tangent bundle $TN$, the corresponding Hunt functions $h^1,...,h^n$  are measurable, bounded positive functions, smooth in a neighbourhood of
the identity $1_N$, with compact support and such that $Y_{\alpha}(h^{\beta})(1_N)=\delta^{\beta}_{\alpha}$ (the existence of these functions is proved, for example, in
\cite{Hunt1956}). Generalizing \cite{Jacod2003} we give the following
\begin{definition}\label{definition_characteristic}
Let $b$ be a predictable semimartingale of bounded variation on $\mathbb{R}^n$, and let $A$ be a predictable continuous semimartingale taking values
in the set of semidefinite positive $n \times n$ matrices. Furthermore let  $\nu$ be a predictable random measure defined on $\mathbb{R}_+ \times N$.  If
 $Z$ is a semimartingale on a Lie group $N$, we say that $Z$ has characteristics $(b,A,\nu)$ with respect to $Y_1,...,Y_n$ and $h^1,...,h^n$ if,
for any smooth bounded functions $f,g \in \cinf(N)$ and for any smooth and bounded function $p$ which is identically $0$ in a neighbourhood of
$1_N$, we have that
\begin{eqnarray}
&\sum_{0 \leq s \leq t}p(\Delta Z_s)-\int_{0}^t{\int_N{p(z')\nu(ds,dz')}},&\label{equation_characteristic2}\\
&[f(Z),g(Z)]^c_t-g(Z_0)f(Z_0)-\int_0^t{Y_{\alpha}(f)(Z_{s_-})Y_{\beta}(g)(Z_{s_-})dA^{\alpha\beta}_s}&\label{equation_charcteristic3}\\
&\begin{array}{c}
f(Z_t)-f(Z_0)-\int_0^t{Y_{\alpha}(f)(Z_{s_-})db^{\beta}_s}-\frac{1}{2}\int_0^t{Y_{\alpha}(Y_{\beta}(f))(Z_{s_-})dA^{\alpha\beta}_s}+\\
-\sum_{0 \leq s \leq t}(f(Z_s)-f(Z_{s_-})-h^{\alpha}(\Delta
Z_s)Y_{\alpha}(f)(Z_{s_-}))
\end{array}&\label{equation_characteristic4}
\end{eqnarray}
are local martingales.
\end{definition}

\begin{remark}
We note that condition
\refeqn{equation_charcteristic3} is redundant, because it can be
deduced from  \refeqn{equation_characteristic2} and
\refeqn{equation_characteristic4}.
\end{remark}

The following theorem states that any semimartingale $Z$ defined on a Lie group $N$ admits  essentially a unique  characteristic triplet $(b,A,\nu)$.

\begin{theorem}\label{theorem_characteristic1}
If $Z$ is a semimartingale on a Lie group $N$, then $Z$ admits a
characteristic triplet $(b,A,\nu)$ with respect to $Y_1,...,Y_n$ and
$h^1,...h^n$, which is  unique up to
$\mathbb{P}$ null sets.
\end{theorem}
\begin{proof}
We first prove the existence.
Given a  semimartingale $Z$ on $N$, we can associate with $Z$ a unique random
measure on $N$ given by
$$\mu^Z(\omega,dt,dz)=\sum_{s \geq 0}I_{\Delta Z_s \not =1_N}\delta_{(s,\Delta Z_s(\omega))}(dt,dz),$$
where $\delta_a$ is the Dirac delta with mass in $a \in \mathbb{R}_+ \times N$. The random measure
$\mu^Z$ is an integer-valued random measure (see, e.g., \cite[Chapter II, Proposition 1.16]{Jacod2003}),
hence there exists a unique non-negative predictable random measure $\mu^{Z,p}$, which is the
compensator of $\mu^Z$ (see, e.g., \cite[Chapter II, Theorem 1.8]{Jacod2003}).\\
We prove that $\nu=\mu^{Z,p}$. Indeed, by definition of $\mu^Z$, we have  $\sum_{0 \leq s \leq t}h(\Delta
Z_s)=\int_0^t{\int_N{h(z')\mu^Z(ds,dz')}}$ and, by definition of compensator, we have that
$$\int_0^t{\int_N{h(z')\mu^Z(ds,dz')}}-\int_0^t{\int_N{h(z')\mu^{Z,p}(ds,dz')}}$$
is a local martingale.\\
In order to prove the existence of processes
$b^{\alpha},A^{\alpha\beta}$ we  introduce a Riemannian
embedding $K:N \rightarrow \mathbb{R}^{k_N}$ with respect to a
left invariant metric on $N$. Put
$$\tilde{Z}^{i}=k^i(Z),$$
where $K=(k^1,...k^{k_N})$, and write
$$Z^i_t=\tilde{Z}^i_t-\sum_{0 \leq s \leq t}\left(\Delta \tilde{Z}^i_s- h^{\alpha}(\Delta Z_s)
Y_{\alpha}(k^i)(Z_{s_-})\right).$$
Since $K$ is Riemannian, the norms of  $K_*(Y_{\alpha})(x)$
are constant and so $Y_{\alpha}(k^i)$ are bounded. Because of
$$\Delta Z^i_t= h^{\alpha}(\Delta Z_t)Y_{\alpha}(k^i(Z_{t_-})),$$
and  $h^{\alpha}$ being  bounded, $Z^i$ have bounded jumps and
so they are special semimartingales. This means that the processes
$Z^i$ can be decomposed in a unique way as
$$Z^i=B^i+M^{i,c}+M^{i,d},$$
where $B^i$ is a predictable process having a bounded variation, $M^{i,c}$ is a continuous local martingale and $M^{i,d}$ is a purely
discontinuous local martingale. If we consider the matrix
$$P=(Y_{\alpha}(k^i))|_{\stackrel{\alpha=1,...,n}{i=1,...,k_N}},$$
since $K$ is an immersion and
$Y_1,...,Y_n$ are point by point linearly independent,  $P$ is non-singular. Therefore there
exists a pseudoinverse
$\tilde{P}=(P^{\alpha}_i)|_{\stackrel{i=1,...,k_N}{\alpha=1,...,n}}$ such
that $\tilde{P} \cdot P=I_n$, $P \cdot \tilde{P}=Id|_{Im(P)}$. We can  choose, for example,
$\tilde{P}=(P^T \cdot P)^{-1} \cdot P^T$. Therefore, we can define
\begin{eqnarray*}
b^{\alpha}_t&=&\int_0^t{\tilde{P}^{\alpha}_i(Z_{s_-})dB^i_s+Y_{\beta}(\tilde{P}^{\alpha}_i)(Z_{s_-})\tilde{P}^{\beta}_j(Z_{s_-})d[M^{i,c},M^{j,c}]_s}\\
A^{\alpha\beta}_t&=&\int_0^t{\tilde{P}^{\alpha}_i(Z_{s_-})\tilde{P}^{\beta}_j(Z_{s_-})d[M^{i,c},M^{j,c}]_s}.
\end{eqnarray*}
Given $f,g \in \cinf(N)$ let us consider two extensions $\tilde{f},\tilde{g}$ in $\mathbb{R}^{k_N}$ which are constants with respect to a distribution $D \subset T\mathbb{R}^{k_N}|_{K(N)}$ which is transverse  to $TK(N)$, i.e.,  for any $Y,Y' \in D$, $Y(\tilde{f})=Y(\tilde{g})=0$ and $Y(Y'(f))=Y(Y'(g))=0$ (the existence of such kind of extensions is guaranteed by the existence of a tubular neighbourhood of $K(N)$).\\
By It\^o formula we have
\begin{equation}\label{equation_characteristic1}
\begin{array}{rcr}
f(Z_t)-f(Z_0)&=&\int_0^t{\partial_{k^i}(\tilde{f})(Z_{s_-})d\tilde{Z}^i_s}+\frac{1}{2} \int_0^t{\partial_{k^ik^j}(\tilde{f})(Z_{s_-})d[\tilde{Z}^i,\tilde{Z}^j]^c_s}+\\
&&+\sum_{0 \leq s \leq t}(f(Z_s)-f(Z_{s_-})-\Delta {Z}^i_s
\partial_{k^i}(\tilde{f})(Z_{s_-}))
\end{array}
\end{equation}
and the same formula holds for $g$. Recalling that
$[\tilde{Z}^i,\tilde{Z}^j]^c=[Z^i,Z^j]^c=[M^{i,c},M^{j,c}]$ and
that, for our choice of the extensions $\tilde{f},\tilde{g}$,
$$\partial_{k^i}(\tilde{f})=\tilde{P}_i^{\alpha}Y_{\alpha}(f),$$ we have
$$[f(Z),g(Z)]^c_t=\int_0^t{Y_{\alpha}(f)(Z_{s_-})Y_{\beta}(g)(Z_{s_-})dA^{\alpha\beta}_s}.$$
Finally, recalling that
\begin{eqnarray*}
\tilde{Z}^i_t=A^i_t+M^{i,c}_t+M^{i,d}_t+\sum_{0 \leq s \leq t}(\Delta \tilde{Z}_s^i-h^{\alpha}(\Delta Z_s)Y_{\alpha}(k^i)(Z_{s_-}))\\
\partial_{k^ik^j}(\tilde{f})=Y_{\beta}(Y_{\alpha}(f))\tilde{P}^{\alpha}_j\tilde{P}^{\beta}_i+Y_{\beta}(\tilde{P}^{\alpha}_j)
\tilde{P}^{\beta}_iY_{\alpha}(f),
\end{eqnarray*}
and using both equation \refeqn{equation_characteristic1} and Lemma
\ref{lemma_geometrical1} we obtain that
\begin{eqnarray*}
&f(Z_t)-f(Z_0)-\int_0^t{Y_{\alpha}(f)(Z_{s_-})db^{\alpha}_s}-\frac{1}{2}\int_0^t{Y_{\alpha}(Y_{\beta}(f))(Z_{s_-})dA^{\alpha\beta}_s}+&\\
&-\sum_{0 \leq s \leq t}(f(Z_s)-f(Z_{s_-})-h^{\alpha}(\Delta
Z_s)Y_{\alpha}(f)(Z_{s_-}))&
\end{eqnarray*}
is a local martingale.\\
The uniqueness of $\nu$ has already been proved using the uniqueness of the compensator of the random measure $\mu^Z$  (see \cite[Chapter II,
Theorem 1.8]{Jacod2003}). In order to prove the uniqueness of $b^{\alpha},A^{\alpha\beta}$ we use the  fact that a predictable
martingale of bounded variation is constant (see, e.g., \cite[Chapter III, Theorem 12]{Protter1990}). Indeed, if $(b',A',\nu)$ is another
characteristic triplet of $Z$, we  have that, for any $f,g \in \cinf(M)$,
\begin{eqnarray*}
&\int_0^t{Y_{\alpha}(f)(Z_{s_-})Y_{\beta}(g)(Z_{s_-})d(A^{\alpha\beta}_s-A'^{\alpha\beta}_s)}&\\
&\int_0^t{Y_{\alpha}(f)(Z_{s_-})d(b^{\alpha}_s-b'^{\alpha}_s)}-\int_0^t{Y_{\alpha}(Y_{\beta}(f))(Z_{s_-})d(A^{\alpha\beta}_s-A'^{\alpha\beta}_s)}&
\end{eqnarray*}
are local martingales. Since the processes involved in the previous integrals are predictable and $b,b',A,A'$ are of bounded variation, they are
local martingales having a vanishing bounded variation at the origin and so they are identically equal to $0$. Finally, by using the arbitrariness of $f,g$ and the
existence of a partition of unity for $N$, we find that $b-b'=0$ and $A-A'=0$ up to $\mathbb{P}$-null sets. ${}\hfill$ \end{proof}

\subsection{Gauge symmetries and semimartingales characteristics}\label{subsection_gauge_main}

In this subsection we provide  an equivalent method to verify the  conditions in Definition \ref{definition_gauge}
using the characteristics introduced in the previous subsection. \\
In particular, after introducing suitable geometric and probabilistic tools, we look for conditions to be satisfied by the characteristics of a
semimartingale in order to ensure that the semimartingale admits a gauge symmetry group.\\

First of all we need to study in more detail the role of the filtration $\mathcal{F}_t$ in Definition \ref{definition_gauge}. In fact, although the definition of gauge symmetry group apparently  concerns only the law of $Z$ and not the chosen filtration, there are examples of semimartingales $Z$ admitting a gauge symmetry group $\mathcal{G}$ with respect to a filtration $\mathcal{F}_t$ but such that $\mathcal{G}$ is no longer a gauge symmetry group for $Z$ if a different filtration $\mathcal{H}_t$ is chosen. For example, let $W$ be a
standard $n$ dimensional Brownian motion, let $\mathcal{F}_t$ be its natural filtration and let us put $\mathcal{H}_t=\mathcal{F}_t \vee \sigma(W_T)$. It is
well known that $W$ is a semimartingale with respect to both $\mathcal{F}_t$ and $\mathcal{H}_t$, but the rotations are a gauge group only with
respect to the filtration $\mathcal{F}_t$ and not with respect to $\mathcal{H}_t$. Indeed, let $B:\mathbb{R}^n \rightarrow O(n)$ be a measurable
map such that $B(x) \cdot x=(|x|,0,...,0)$. The constant process $B(W_T)$ is predictable with respect to the filtration $\mathcal{H}_t$ and it
is not adapted with respect to $\mathcal{F}_t$. On the other hand the semimartingale
$$\tilde{W}^{\alpha}_t=\int_0^t{B^{\alpha}_{\beta}(W_T)dW^{\beta}_s}=B^{\alpha}_{\beta}(W_T) W^{\beta}_t,$$
is not a Brownian motion since, for example, $\tilde{W}_T=(|W_T|,0,...,0)$ is not a Gaussian random variable. This phenomenon is due to the fact
that the family of the $\mathcal{H}_t$-predictable processes is too large for preserving the invariance property of Brownian motion. In order to
avoid this kind of phenomena, and ensuring that a gauge symmetry is a property of the law of the process $Z$ and not of
its filtration, we introduce the following definition.

\begin{definition}
Let $Z$ be a semimartingale with respect to the filtration $\mathcal{F}_t$. We say that the filtration $\mathcal{F}_t$ is a generalized natural
filtration if there exists a version of the characteristic triplet $(b,A,\nu)$ of $Z$ ( with respect to the filtration $\mathcal{F}_t$), which is
predictable with respect to the natural filtration $\mathcal{F}^Z_t \subset \mathcal{F}_t$ of the semimartingale $Z$.
\end{definition}

It is important to note that if $(b,A,\nu)$ are the characteristics of a semimartingale $Z$ with respect to its natural filtration, then they
are also the characteristics of $Z$ with respect to any generalized natural filtration for $Z$. For this reason, hereafter, whenever we consider
a generalized natural filtration $\mathcal{F}_t$ for $Z$ we can use the characteristics $(b,A,\nu)$ with respect to the natural filtration of
$Z$ as the characteristics of $Z$ with respect to $\mathcal{F}_t$.\\

Let us  consider the probability space
$$\Omega^c=\Omega_A \times \Omega_B,$$
where $\Omega_A=\mathcal{D}_{1_N}([0,+\infty),N)$ is the space of c\acc{a}dl\acc{a}g functions $\omega_A(t)$ taking values on $N$ and such that
$\omega_A(0)=1_N$, and $\Omega_B=L_{loc}^{\infty}([0,+\infty),\mathcal{G})$ is
the set of locally bounded and measurable functions taking values in $\mathcal{G}$.\\
On the set $\Omega_A$ we consider the standard filtration $\mathcal{F}^A_{t}$ of $\mathcal{D}_{1_N}([0,+\infty),N)$ and on $\Omega_B$  the filtration $\mathcal{F}^B_t$ generated by the standard filtration of $C^0([0,+\infty),\mathcal{G}) \subset \Omega_B$ (usually
called the predictable filtration). We denote by $\pi_A, \pi_B$ the projections of $\Omega$ on $\Omega_A$ and $\Omega_B$ respectively and so we define $\mathcal{F}^c_t$
$\mathcal{F}^c_t=\sigma(\pi_A^{-1}(\mathcal{F}^A_t),\pi^{-1}_B(\mathcal{F}^B_t))$. We call $\Omega^c$ the canonical probability space and
$\mathcal{F}^c_t$ the natural filtration on $\Omega^c$. \\
We need the space $\Omega_A$ in order  to define a semimartingale $Z$ on $N$, and the space $\Omega_B$ in order  to define a locally bounded predictable process
taking values on $\mathcal{G}$. Choosing a particular semimartingale $Z$ on $N$ and a predictable process $G_t$ on $\mathcal{G}$ is equivalent to
fixing a probability measure $\mathbb{P}$ on $\Omega$ such that $Z_t(\omega)=\pi_A(\omega)(t)$ is a semimartingale on $N$ (the fact that the
process $G_t(\omega)=\pi_B(\omega)(t)$ is a locally bounded predictable process
is automatically guaranteed by the choices of the space $\Omega_B$ and the filtration $\mathcal{F}^B_t$).\\
Given  an $N$ valued semimartingale $Z$ and  a generic predictable process $G_t$  taking values in $\mathcal{G}$, both defined on a
probability space $(\Omega,\mathcal{F},\mathcal{F}_t,\mathbb{P})$, there exist a natural
probability measure $\mathbb{P}^c=M_*(\mathbb{P})$ on the canonical probability space $\Omega^c$ and a
natural map
$$\begin{array}{rlcc}
M:& \Omega & \longrightarrow & \Omega^c \\
& \omega& \longmapsto & (Z_t(\omega),G_t(\omega))\end{array}$$ which puts the couple $(Z_t,G_t)$ in canonical form. Thus, fixing the process
$G_t$ and the law $\mathbb{P}^Z$ of the semimartingale $Z_t$ is equivalent to fixing the probability law $\mathbb{P}^c$ on $\Omega^c$ so that
the restriction of $\mathbb{P}^c$ to the $\Omega_A$ measurable subsets,  $\mathbb{P}=\mathbb{P}^c|_{\mathcal{F}^A}$, is exactly $\mathbb{P}^Z$.
As a consequence, proving a statement  involving only the measurable objects $Z_t,G_t$ which is independent from the choice of a specific
predictable process $G_t$, is equivalent to proving the same statement on the probability space $\Omega^c$ with respect to the canonical
processes $\omega_A(t),\omega_B(t)$ and for a suitable subset of probability laws $\mathbb{P}^c$ on $\Omega^c$ such that
$\mathbb{P}|_{\mathcal{F}^A}=\mathbb{P}^Z$. This subset depends on the filtration  $\mathcal{F}_t$ of the probability space chosen. In
particular if $\mathcal{F}_t$ is a generalized natural filtration for $Z$, then $\tilde{\mathcal{F}}^c_t$ is a generalized natural filtration for
$\omega_A(t)$ (where $\tilde{\mathcal{F}}^c_t$ is the completion of $\mathcal{F}^c_t$ with respect to $\mathbb{P}^c$). Since we consider only
generalized natural filtrations for the semimartingale $Z$, we suppose that $\mathbb{P}^c$ is such that $\tilde{\mathcal{F}}^c_t$ is a generalized natural
filtration.\\
For this reason, in the following  we shall only consider  the canonical probability space $\Omega^c$ with law $\mathbb{P}=\mathbb{P}^c$ and denote by $Z_t$
the canonical semimartingale $\omega_A(t)$ and by $G_t$ the canonical predictable process $\omega_B(t)$.\\
In the same way, we identify  the solution $\tilde{Z}$ to the SDE $d\tilde{Z}_t=\Xi_{G_t}(dZ_t)$ with the measurable map
$\Lambda_A:\Omega \rightarrow \Omega_A$ such that $\tilde{Z}_t(\omega)=\Lambda_{A}(\omega)(t)$. We can extend the map $\Lambda_A$ to a map
$\Lambda:\Omega \rightarrow \Omega$ given by
$$\Lambda(\omega)=(\Lambda_A(\omega),\pi_2(\omega)),$$
defining a new probability measure
$\mathbb{P}'=\Lambda_*(\mathbb{P})$. The map $\Lambda$  is $\mathbb{P}$ invertible, i.e. there exists a map $\Lambda'$ such that
$\Lambda \circ \Lambda'$ is equal to the identity map up to $\mathbb{P}'$ null sets and the map $\Lambda' \circ \Lambda$ is equal to the
identity up to $\mathbb{P}$ null sets. The construction of the  map $\Lambda'$ is similar to the construction of $\Lambda$ starting from  the stochastic
differential equation $dZ_t=\Xi_{(G_t)^{-1}}(d\tilde{Z}_t)$ and the measure $\mathbb{P}'$. The proof of the fact that $\Lambda'$ is the
$\mathbb{P}'$ inverse of $\Lambda$ and hence $\Lambda$ is the $\mathbb{P}$ inverse
of $\Lambda'$, is based on Theorem \ref{theorem_gauge1}. It is important to note that $\hat{\mathcal{F}}^c_t$, i.e.
the completion of $\mathcal{F}^c_t$ with respect to the probability $\mathbb{P}'$, could not be a generalized natural
filtration for $\omega_A(t)$ even if $\tilde{\mathcal{F}}^c_t$ is a generalized natural filtration for $Z_t$ under $\mathbb{P}$. \\
Given the probability law $\mathbb{P}^Z$ on $\Omega_A$, by Theorem \ref{theorem_characteristic1} there exist some measurable and predictable
functions $b^{\alpha},A^{\alpha\beta}:\Omega_A \times \mathbb{R}_+ \rightarrow \mathbb{R}$ and a random predictable measure $\nu:\Omega_A
\rightarrow \mathcal{M}(\mathbb{R}_+ \times N)$ which are the characteristics of the canonical process $Z_t(\omega)=\pi_A(\omega(t))$ and are
uniquely defined up to $\mathbb{P}^Z$ null-sets. The characteristic triplets $(b,A,\nu)$, seen as $\mathcal{F}^A$ measurable objects, are
uniquely determined by the probability measure $\mathbb{P}^Z$. The converse, namely the fact  that the $\mathcal{F}^A$ measurable objects
$(b,A,\nu)$ uniquely individuate the probability law $\mathbb{P}^Z$, is in general not true (the reader can think, for example, to  diffusion processes whose
martingale problem admits multiple solutions). In the case in which the triplet $(b,A,\nu)$ uniquely determines the probability law $\mathbb{P}^Z$ on
$\Omega_A$ we say that the triplet $(b,A,\nu)$ uniquely individuates the law of $Z$. Examples of this situation  are, e.g., the $\mathbb{R}^n$ Brownian
motion, $\mathbb{R}^n$ L\' evy processes,
diffusion processes with a unique solution to the associated martingale problem, and point processes. If
the law $\mathbb{P}$ on $\Omega^c$ is such that $\mathbb{P}|_{\mathcal{F}^A}=\mathbb{P}^Z$ and the filtration $\tilde{\mathcal{F}}^c_t$ is a
generalized natural filtration for $\omega_A(t)$, then the same $\mathcal{F}^A$ measurable characteristics $(b,A,\nu)$, viewed as $\Omega^c$
semimartingales,  are  characteristics of $Z$ with respect to $\tilde{\mathcal{F}}^c_t$ as well. Obviously it is possible to define other
characteristic triplets $(\bar{b},\bar{A},\bar{\nu})$ of $Z_t$ on $\Omega^c$ which are only $\tilde{\mathcal{F}}^c$ adapted and not
$\mathcal{F}^A_t$ adapted. The characteristics
$(\bar{b},\bar{A},\bar{\nu})$ are equal to $(b,A,\nu)$ up to $\mathbb{P}$ null sets (and not only up to $\mathbb{P}^Z$ null sets).\\

Let us now consider a map
$\Xi_g:N \rightarrow N$ such that $\Xi_g(1_N)=1_N$. This means
that the tangent map $T\Xi_g$ of $\Xi_g$ sends the tangent space
of the identity $TN|_{1_N}$ into itself. Recalling that the Lie
algebra $\mathfrak{n}$ associated with $N$ is exactly the tangent
space to the identity, we have that there exists a map
$$\Upsilon_g=T\Xi_g|_{1_N}:\mathfrak{n} \rightarrow \mathfrak{n}.$$
The map $\Upsilon$ has the following property: if $Y$ is any right invariant vector field on $N$ and  $\hat{\Xi}_g(\tilde{z},z)=\Xi_g(z)
\cdot \tilde{z}$, then, by  definition of right invariant vector fields, for any smooth function $f \in
\cinf(N)$, we have
$$Y^z(f \circ \hat{\Xi}_g)(\tilde{z},1_N)=\Upsilon_g(Y)(f)(\tilde{z}),$$
where $Y^z$ denotes the vector fields $Y$ applied to the $z^{\alpha}$ variables. Going further along in this way, instead of working with
first derivatives we can work with second derivatives and we can define a linear map
$$O_g:\mathfrak{n} \otimes \mathfrak{n} \rightarrow \mathfrak{n}$$
such that, for any two right invariant vector fields $Y,Y'$ defined on $N$, we have
$$Y'^z(Y^z(f \circ \hat{\Xi}_g))(\tilde{z},1_N)=\Upsilon_g(Y')(\Upsilon_g(Y)(f))(\tilde{z})+O_g(Y,Y')(f)(\tilde{z}).$$
If we fix a basis $Y_1,...,Y_n$ of $\mathfrak{n}$ (and so of right-invariant vector fields on $N$), the linear maps $\Upsilon_g,O_g$ become
matrices $\Upsilon^{\alpha}_{g,\beta}$ and $O^{\alpha}_{g,\beta\gamma}$, where
\begin{eqnarray*}
\Upsilon_g(Y_{\beta})=\Upsilon_{g,\beta}^{\alpha}Y_{\alpha}\\
O_g(Y_{\beta},Y_{\gamma})=O^{\alpha}_{g,\beta\gamma}Y_{\alpha}.
\end{eqnarray*}

\begin{theorem}\label{theorem_characteristic2}
Let $Z$ be a semimartingale on a Lie group $N$ with characteristic triplet $(b(\omega_A),A(\omega_A),\nu(\omega_A))$. Suppose that $Z$ admits
$\mathcal{G}$ with action $\Xi_g$ as gauge symmetry group with respect to any generalized natural filtration, then if $\mathbb{P}$ is a measure
on $\Omega^c$ such that $\tilde{\mathcal{F}}_t$ is a generalized natural filtration with respect both $Z_t$ and $d\tilde{Z}_t=\Xi_{G_t}(dZ_t)$, we have
\begin{equation}\label{equation_characteristic8}
\begin{array}{rcl}
db^{\alpha}_t(\omega)&=&\Upsilon^{\alpha}_{g(\omega_B),\beta}db_t^{\beta}(\pi_A(\Lambda'(\omega)))+\frac{1}{2}
O^{\alpha}_{g(\omega_B),\beta\gamma}dA^{\beta\gamma}_t(\pi_A(\Lambda'(\omega)))+\\
&&+\int_N{(h^{\alpha}(z')-h^{\beta}(\Xi_{g^{-1}(\omega_B)}(z'))\Upsilon^{\alpha}_{g(\omega_B),\beta})\nu(\pi_A(\Lambda'(\omega)),dt,dz')}
\end{array}
\end{equation}
\begin{eqnarray}
dA_t^{\alpha\beta}(\omega)&=&\Upsilon^{\alpha}_{g(\omega_B),\gamma}\Upsilon^{\beta}_{g(\omega_B),\delta}dA_t^{\gamma\delta}
(\pi_A(\Lambda'(\omega)))\label{equation_characteristic9}\\
\nu(\omega,dt,dz)&=&\Xi_{g(\omega_B)*}(\nu(\pi_A(\Lambda'(\omega)),dt,dz)),\label{equation_characteristic10}
\end{eqnarray}
up to a $\mathbb{P}'=\Lambda_*(\mathbb{P})$ null set. Furthermore, if $\tilde{b},\tilde{A},\tilde{\nu}$ are $\pi^{-1}_A(\mathcal{F}^A)$
measurable, the previous equalities hold with respect to $\mathbb{P}^Z$ null sets. Finally, if $(b,A,\nu)$ uniquely determines the law of
$Z$, the previous conditions are also sufficient for the existence of a gauge symmetry group.
\end{theorem}

Before proving the theorem we study the transformations of the characteristics under (canonical) semimartingale changes.

\begin{lemma}\label{lemma_characteristic1}
If $Z$ is a semimartingale with characteristics $(b,A,\nu)$, then
$d\tilde{Z}=\Xi_{G_t}(dZ)$ is a semimartingale with characteristics
\begin{eqnarray*}
d\tilde{b}^{\alpha}_t&=&\Upsilon^{\alpha}_{G_t,\beta}db^{\beta}_t+\frac{1}{2}O^{\alpha}_{G_t,\beta\gamma}dA^{\beta\gamma}_t+\int_N{(h^{\alpha}(z')-h^{\beta}(\Xi_{G_t^{-1}}(z'))\Upsilon^{\alpha}_{G_t,\beta})\nu(dt,dz')}\\
d\tilde{A}^{\alpha\beta}_t&=&\Upsilon^{\alpha}_{G_t,\gamma}\Upsilon^{\beta}_{G_t,\delta}dA^{\gamma\delta}_t\\
\tilde{\nu}&=&\Xi^*_{G_t}(\nu).
\end{eqnarray*}
\end{lemma}
\begin{proof}
We denote by $(\tilde{b},\tilde{A},\tilde{\nu})$ the characteristic triplet of $\tilde{Z}$. Since the jumps of $\tilde{Z}$ are $\Delta
\tilde{Z}_t=\Xi_{G_t}(\Delta Z_t)$ and the jump times of $\tilde{Z}$ are the same of $Z$  we have, using the notation of  Theorem
\ref{theorem_characteristic1},
$$\mu^{\tilde{Z}}(dt,dz)=\sum_{s \geq 0}I_{\Delta \tilde{Z}_s \not =0}(s)\delta_{(s,\Delta \tilde{Z}_s)}(dt,dz)=\sum_{s \geq 0}
I_{\Delta Z_s \not =0}(s)\delta_{(s,\Xi_{G_s}(\Delta Z_s))}(dt,dz).$$ 
If we identify, with a slight abuse of notation, the push-forward of the
map $(s,z) \rightarrow (s,\Xi_{G_s}(z))$ with the push-forward of the map $(s,z) \rightarrow \Xi_{G_s}(z)$, we have
$$\delta_{(s,\Xi_{G_s}(\Delta Z_s))}(du,dz)=\Xi_{G_s,*}(\delta_{(s,\Delta Z_s)})(du,dz),$$
and so
$$\mu^{\tilde{Z}}=\Xi_{G_t,*}(\mu^Z).$$
If we consider a function $h:N \rightarrow \mathbb{R}$ which is identically zero in a neighbourhood of $1_N$, by definition of push-forward of a measure we have
$$\int_0^t{\int_N{h(z)\Xi_{G_s,*}(\mu^Z-\nu)(ds,dz)}}=\int_0^t{\int_N{h(\Xi_{G_s}(z))(\mu^Z-\nu)(ds,dz)}}.$$
Furthermore $\int_0^t{\int_N{h(\Xi_{G_s}(z))(\mu^Z-\nu)(ds,dz)}}$ is a martingale, since $h(\Xi_{G_s}(z))$ is a predictable function and $\nu$
is the predictable projection of the random measure $\mu^Z$. Since $\mu^{\tilde{Z}}=\Xi_{G_t,*}(\mu^Z)$ we have that $\Xi_{G_t,*}(\nu)$ is the
predictable projection of the measure $\mu^{\tilde{Z}}$ and $\tilde{\nu}=\Xi_{G_t,*}(\nu)$.\\
For the formulas of $\tilde{A}$ and $\tilde{b}$ we use the definition of solution to a canonical SDE, Lemma \ref{lemma_geometrical1} and the
properties of $\Upsilon_g$ and $O_g$. We make the proof only for $\tilde{A}$,  the proof for $\tilde{b}$ being entirely similar. \\
Fixing an immersion $K:N \rightarrow \mathbb{R}^{k_N}$, by
definition and Lemma \ref{lemma_geometrical1}, for any functions
$f,g \in \cinf(N)$, the properties of $\Upsilon_g$ ensure that
\begin{eqnarray*}
[f(\tilde{Z}),g(\tilde{Z})]_t^c&=&\int_0^t{\partial_{k'^i}(\tilde{f} \circ \overline{\Xi})(\tilde{Z}_{s_-},Z_{s_-},Z_{s_-})\partial_{k'^j}(\tilde{g}\circ \overline{\Xi})(\tilde{Z}_{s_-},Z_{s_-},Z_{s_-})d[k^i(Z),k^j(Z)]_s^c}\\
&=&\int_0^t{Y^{z'}_{\alpha}(\tilde{f} \circ \overline{\Xi})(\tilde{Z}_{s_-},Z_{s_-},Z_{s_-})Y^{z'}_{\beta}(\tilde{g} \circ \overline{\Xi})(\tilde{Z}_{s_-},Z_{s_-},Z_{s_-})}\\
&&\tilde{P}^{\alpha}_i(Z_{s_-})\tilde{P}^{\beta}_j(Z_{s_-})d[k^i(Z),k^j(Z)]^c_s=\\
&=&\int_0^t{Y_{\gamma}(f)(\tilde{Z}_{s_-})Y_{\delta}(g)(\tilde{Z})\Upsilon^{\gamma}_{G_s,\alpha}\Upsilon^{\delta}_{G_s,\beta}\tilde{P}^{\alpha}_i(Z_{s_-})\tilde{P}^{\beta}_j(Z_{s_-})d[k^i(Z),k^j(Z)]^c_s},
\end{eqnarray*}
where $\tilde{g},\tilde{f}$ are two extensions of $f,g$ on
$\mathbb{R}^{k_N}$, and $\tilde{P}$ is a pseudoinverse matrix of
$P=(Y_{\alpha}(k^i))$ (see Theorem \ref{theorem_characteristic1}).
By definition of characteristics we have that
\begin{eqnarray*}
&[k^i(Z),k^j(Z)]_t^c-\int_0^t{Y_{\alpha}(k^i)(Z_{s_-})Y_{\beta}(k^j)(Z_{s_-})dA^{\alpha\beta}_s}=&\\
&=[k^i(Z),k^j(Z)]_t^c-\int_0^t{P^i_{\alpha}(Z_{s_-})P^j_{\beta}(Z_{s_-})dA^{\alpha\beta}_s}&
\end{eqnarray*}
is a local martingale. This means that
$$[f(\tilde{Z}),g(\tilde{Z})]^c_t-\int_0^t{Y_{\gamma}(f)(\tilde{Z}_{s_-})Y_{\delta}(g)(\tilde{Z}_{s_-})\Upsilon^{\gamma}_{G_s,\alpha}\Upsilon^{\delta}_{G_s,\beta}dA^{\alpha\beta}_s}$$
is a local martingale. ${}\hfill$
\end{proof}\\

\begin{proof}[Proof of Theorem \ref{theorem_characteristic2}]
We cannot directly use  Lemma \ref{lemma_characteristic1} to compare $(b,A,\nu)$ with $(\tilde{b},\tilde{A},\tilde{\nu})$, since
 $Z$ and $\tilde{Z}$, where $d\tilde{Z}_t=\Xi_{G_t}(dZ_t)$, are two different processes being two different
 functions  from $\Omega^c \times \mathbb{R}_+$ into $N$. Indeed $Z_t(\omega)=\pi_A(\omega)(t)$,
while $\tilde{Z}_t(\omega)=\pi_A(\Lambda(\omega))(t)$.\\
Since $\Lambda'$ is the $\mathbb{P}'$ inverse of $\Lambda$, $\tilde{Z}(\Lambda'(\omega))$ is exactly the same process as $Z$ (as functions
defined on $\Omega^c$). If $\tilde{Z}(\Lambda')$ and $Z$ have the same law, and since both the filtration $\hat{\mathcal{F}}^c_t$ and $\tilde{\mathcal{F}}^c_t$ are canonical, they necessarily have the same characteristics up to a $\mathbb{P}'$
null set and therefore $b(\omega)=\tilde{b}(\Lambda'(\omega))$, $A(\omega)=\tilde{A}(\Lambda'(\omega))$ and
$\nu(\omega)=\tilde{\nu}(\Lambda'(\omega))$. If $\tilde{b}(\Lambda')$, $\tilde{A}(\Lambda'(\omega))$ and $\tilde{\nu}(\Lambda'(\omega))$ are
$\pi_A^{-1}(\mathcal{F}^A_t)$ measurable
(usually they are only $\hat{\mathcal{F}}^c_t$ measurable) they are then equal to $b,a$ and $\nu$ up to a null set with respect to
 $\pi_{A*}(\mathbb{P})=\pi_{A*}(\mathbb{P}')$.\\
Obviously if $(b,A,\nu)$ uniquely  identifies  in $\Omega_A$ the law of $Z$, the condition stated in the theorem is also sufficient.${}\hfill$
\end{proof}

\subsection{Gauge symmetries of L\'evy processes}\label{subsection_gauge_Levy}

Generalizing \cite{Jacod2003} we introduce the following definition.

\begin{definition}\label{definition_Levy}
A c\acc{a}dl\acc{a}g semimartingale $Z$ on a Lie group $N$ is called an \emph{independent increments process} if its characteristics $(b,A,\nu)$ are deterministic.\\
The process $Z$ is a L\'evy process if $b_t=b_0t,A_t=A_0 t,\nu(dt,dx)=\nu_0(dx)dt$ for some $b_0 \in \mathbb{R}^n$, $A_0$ $n \times n$ symmetric
 positive semidefinite matrix and some $\sigma$-finite measure $\nu_0$ on $N$ such that $\int_N{(h^{\alpha}(z))^2\nu_0(dz)}<+\infty$ and
$\int_N{f(z)\nu_0(dz)}<+\infty $ for any smooth and bounded function $f \in \cinf(N)$ which is identically zero in a neighbourhood of $1_N$.
\end{definition}

It is evident that the definition of independent increments process depends on the filtration $\mathcal{F}_t$ used for defining the characteristics $(b,A,\nu)$. Furthermore since $(b,A,\nu)$ are deterministic the filtration $\mathcal{F}_t$ should always be canonical.

\begin{remark}
The characteristics of a  L\'evy process introduced in Definition \ref{definition_Levy} are the same as those discussed in Subsection \ref{subsubsection_Levy}.
Furthermore if $Z$ is a L\'evy process, then $Z$ is also an homogeneous Markov process. Its generator $L$ has the following form on $f \in
\cinf(N)$
$$
L(f)(z)=b_0^{\alpha}Y_{\alpha}(f)(z)+\frac{1}{2}A^{\alpha\beta}_0
Y_{\alpha}(Y_{\beta}(f))(z)+\int_N{(f(z' \cdot
z^{-1})-f(z)-h^{\alpha}(z')Y_{\alpha}(f)(z))\nu_0(dz')},
$$
\noindent for any $z\in N$.
\end{remark}

\begin{theorem}\label{theorem_characteristic3}
If a semimartingale $Z$ is an independent increments
process such that its law is uniquely determined by its
characteristics, then $Z$ admits $\mathcal{G}$ as  gauge symmetry group  with action $\Xi_g$ if and only if, for any $g \in
\mathcal{G}$,
\begin{eqnarray}
b^{\alpha}_t&=&\Upsilon^{\alpha}_{g,\beta}b_t^{\beta}+\frac{1}{2}O^{\alpha}_{g,\beta\gamma}A^{\beta\gamma}
+\int_0^t{\int_N{(h^{\alpha}(z')-h^{\beta}(\Xi_{g^{-1}}(z'))\Upsilon^{\alpha}_{g,\beta})\nu(ds,dz')}}\label{equation_characteristic5}\\
A^{\alpha\beta}_t&=&\Upsilon^{\alpha}_{g,\gamma}\Upsilon^{\beta}_{g,\delta}A^{\gamma\delta}_t\label{equation_characteristic6}\\
\nu&=&\Xi_{g*}(\nu).\label{equation_characteristic7}
\end{eqnarray}
\end{theorem}
\begin{proof}
Let us consider the constant process $G_t=g_0$ for some $g_0 \in \mathcal{G}$. Since $\Xi_{g_0}$ is a diffeomorphism and since the constant
process $G_t=g_0$ is measurable with respect to both the natural filtrations of $Z_t$ and of $\tilde{Z}_t$, it is simple to prove that, if
$\tilde{\mathcal{F}}^c_t$ is a generalized natural
 filtration for $Z_t$, then it is a generalized natural filtration also for $d\tilde{Z}_t=\Xi_{g_0}(dZ_t)$. This fact implies that $\hat{\mathcal{F}}^c_t$ is a generalized natural filtration
for $\omega_A(t)$ with respect to the law $\mathbb{P}'$. For this reason since  $(b,A,\nu)$ and the process $G_t$
  do not depend on $\omega$, \refeqn{equation_characteristic5},
\refeqn{equation_characteristic6} and \refeqn{equation_characteristic7} follow from the necessary condition in Theorem \ref{theorem_characteristic1}.  \\
Conversely, if equations \refeqn{equation_characteristic5}, \refeqn{equation_characteristic6} and \refeqn{equation_characteristic7} hold, they
imply equations \refeqn{equation_characteristic8}, \refeqn{equation_characteristic9} and \refeqn{equation_characteristic10} to any elementary
process $G_t$. Using standard techniques we can extend \refeqn{equation_characteristic8}, \refeqn{equation_characteristic9}
and \refeqn{equation_characteristic10} for any locally bounded predictable process $G_t$.\\
Since the law of $Z$ is uniquely determined by its characteristics, the thesis follows  by the sufficient condition in Theorem
\ref{theorem_characteristic2}.
\hfill\end{proof}

\begin{remark}
It is important to recall that  the law of an independent increments semimartingale
 on the Lie group $N=\mathbb{R}^n$ is always
uniquely determined by its characteristics (see, e.g., \cite{Jacod2003},
Chapter II, Theorem 4.15 and the corresponding comments in that reference).
\end{remark}

We now  propose a general method for explicitly  constructing
L\'evy processes admitting  a gauge symmetry group $\mathcal{G}$ with action $\Xi_g$.\\
In order to show that our construction is a generalization of the Brownian motion case, we begin with  a standard example. Consider $N=\mathbb{R}^{n}$ and the L\'evy process with  generator given by
$$L(f)(z)=\sum_{\alpha=1}^n\frac{D}{2}\partial_{z^{\alpha}z^{\alpha}}(f)(z)+\int_{N}{(f(z+z')-f(z)-I_{|z'|<1}(z')z^{\alpha}\partial_{z^{\alpha}}(f))F(|z'|)dz'},$$
where $D \in \mathbb{R}_+$, $|\cdot |$ is the standard norm of $\mathbb{R}^n$ and $F:\mathbb{R}_+  \rightarrow \mathbb{R}_+$ is a measurable
locally bounded function such that $\int_1^{\infty}{F(r)r^{n-1}dr} < + \infty$ and $\int_0^1{F(r)r^{n+1}} < +\infty$. When $B \in O(n)$  we have
$$\Xi_{B}(z)=B \cdot z.$$
By definition, $B$ respects the standard metric in $\mathbb{R}^n$ and so
$$\Xi_{B*}(F(|z|)dz)=\det(B)F(|B^T \cdot z|)dz=F(|z|)dz.$$
Furthermore, since $\Upsilon^{\alpha}_{B,\beta}=B^{\alpha}_{\beta}$ we have
\begin{eqnarray*}
&\int_N{(z^{\alpha}I_{|z|<1}(z)-\Xi^{\beta}_{B^{-1}}(z){\Upsilon}^{\alpha}_{B,\beta}I_{|\Xi_{B^{-1}}(\cdot)|<1}(z))F(|z|)dz}=&\\
&=\int_N{(z^{\alpha}I_{|z|<1}(z)-(B^{-1})^{\beta}_{\gamma}B^{\alpha}_{\beta}I_{|z|<1}(z)z^{\gamma})F(|z|)dz}=0.&
\end{eqnarray*}
Hence, by Theorem \ref{theorem_characteristic2}, $Z$ admits $O(n)$ as a gauge symmetry group with action $\Xi_B$.\\
In this case  the equation $dZ_t'=\Xi_{B_t}(dZ_t)$ is simply
$$Z'^{\alpha}_t=\int_0^t{B^{\alpha}_{\beta,s} dZ^{\beta}_s}.$$
This example can be easily generalized to the case of a group $\mathcal{G} \subset O(n)$ which is a strict subgroup of $O(n)$ with a faithful
action. Indeed in this case we can consider the   polynomial $k_1(z),...,k_l(z)$ as $\mathcal{G}$-invariant with respect to the action $\Xi_{B}$,
where $B \in \mathcal{G}$. If $G:\mathbb{R}^l \rightarrow \mathbb{R}$ is a non-negative smooth function such that $\partial_{y^i}(G) \not =0$
for $i=1,...,l$ and  $F$ is a measurable, locally bounded function satisfying the previous conditions,  then $\nu_G(dz)=F(|z|)G(k_1(z),...,k_l(z))dz$ is a L\'evy measure strictly invariant
with
respect to $\mathcal{G}$. So the L\'evy process with measure $\nu_G$  admits $\mathcal{G}$, but not all $O(n)$, as a  gauge symmetry group. \\

In order to extend the above construction to a general Lie group $N$, we introduce a special set of Hunt functions. Let $Y_1,...,Y_n$ be a
basis of right-invariant vector fields and consider $a^1,...,a^n \in \mathbb{R}$. It is possible to define the exponential $\exp(a^{\alpha}
Y_{\alpha}) \in N$, which is a point in $N$ defined as the evolution at time $1$ of $1_N$ with respect to the vector field
$a^{\alpha}Y_{\alpha}$. The map $\exp:\mathbb{R}^n \rightarrow N$ is a local diffeomorphism, so there exist a neighbourhood $U$ of $1_N$ and $n$
smooth functions $\hat{h}^1,...,\hat{h}^n$ such that, for any $z \in U$
\begin{equation}\label{equation_Levy1}
\exp(\hat{h}^{\alpha}(z)Y_{\alpha})=z.
\end{equation}
From equation \refeqn{equation_Levy1} and the implicit function theorem we deduce that $\hat{h}^{\alpha}$ are smooth and form a set of Hunt functions.\\
We introduce a special class of Lie group actions $\Xi_g$ on $N$. Suppose that $\Xi_g$ is a Lie group action of endomorphisms of $N$, which means
that, for any $z,z' \in N$,  $\Xi_g(z \cdot z')=\Xi_g(z) \cdot \Xi_g(z')$. Since  the derivative map $T\Xi_g$ is an automorphism of the Lie algebra
$\mathfrak{g}$ of right-invariant vector fields,  there are some functions $\Upsilon^{\alpha}_{g,\beta}$ from $\mathcal{G}$ into $\mathbb{R}$
such that
\begin{equation}\label{equation_Levy2}
T\Xi_g(Y_{\alpha})=\Upsilon^{\beta}_{g,\alpha} Y_{\beta}.
\end{equation}
We remark that the previous equality holds in all $N$, and not only at $1_N$ as it was the case for general group actions. Moreover,
in this case, since equality \refeqn{equation_Levy1} holds in all
$N$, the map $O_g$ associated with $\Xi_g$ is identically equal to
$0$.

\begin{lemma}\label{lemma_gauge1}
There exists a small enough neighbourhood $U$ of $1_N$ such that,
for any $y \in U$,
$$\Upsilon^{\beta}_{g,\alpha}\hat{h}^{\alpha}(\Xi_{g^{-1}}(z))=\hat{h}^{\beta}(z).$$
\end{lemma}
\begin{proof}
Write
$$f(a^1,...,a^n,z)=\exp(a^{\alpha}Y_{\alpha})(z).$$
Since $f(a,x)$, where $a \in \mathbb{R}^n$, is the flow at time $1$ of $a^{\alpha}Y_{\alpha}$, $\Xi_g(f(a,\Xi_{g^{-1}}(z)))$ is the flow at time $1$ of
$\Xi_{g,*}(a^{\alpha}Y_{\alpha})$. Moreover, the fact that $\Xi_g$ is an automorphism of $N$ ensures that
$$\Xi_{g,*}(a^{\alpha}Y_{\alpha})=a^{\alpha}\Xi_{g,*}(Y_{\alpha})=a^{\alpha}\Upsilon_{g,\alpha}^{\beta}Y_{\beta} $$
which means
$$\Xi_g(f(a,\Xi_{g^{-1}}(z)))=f(a^{\alpha}\Upsilon_{g,\alpha}^{\beta},z).$$
Since the $\hat{h}^{\alpha}$ solve equation \refeqn{equation_Levy1}, the $\hat{h}^{\alpha}(\Xi_{g^{-1}}(z))$ solve the equation
$$\Xi_g(f(\hat{h}^{\alpha}(\Xi_{g^{-1}}(z)),\Xi_{g^{-1}}(z)))=z.$$
Using the properties of $f$, from the previous equation follows that the $\hat{h}^{\alpha}(\Xi_{g^{-1}}(z))\Upsilon_{g,\alpha}^{\beta}$ solve equation \refeqn{equation_Levy1}. If we choose the neighbourhood $U$  small enough, by uniqueness of the solutions to equation \refeqn{equation_Levy1}, we have  $\hat{h}^{\beta}(z)=\hat{h}^{\alpha}(\Xi_{g^{-1}}(z))\Upsilon^{\beta}_{g,\alpha}$.
\hfill\end{proof}\\

Suppose that there exists a complete symmetric positive definite
matrix $K^{\alpha\beta}$ such that
\begin{equation}\label{equation_Levy3bis}
\Upsilon^{\gamma}_{g,\alpha} K^{\alpha\beta}
\Upsilon_{g,\beta}^{\delta}=K^{\gamma\delta}
\end{equation}
for any $g \in \mathcal{G}$ and define
$$U_R=\{\exp(a^{\alpha}Y_{\alpha})|a^{\alpha}K_{\alpha\beta}a^{\beta}<R^2\},$$
where $K_{\alpha\beta}$ is the inverse matrix of $K^{\alpha\beta}$. It is simple to verify that the closure of $U_R$ is a compact set.
A consequence of Lemma \ref{lemma_gauge1} is that, for $R$ small enough and for any $g \in \mathcal{G}$, we have $\Xi_g(U_R)=U_R$. \\

An automorphism $\Xi_g$ and a right-invariant metric $K$ which satisfy equation \refeqn{equation_Levy3bis} exist for a large class of Lie groups.
Indeed the set of endomorphisms of a Lie group $N$, which we denote by $Aut(N)$, forms a Lie group itself and we can consider $\mathcal{G}$ as a
maximal compact subgroup of $Aut(N)$.  Since the representation $\Upsilon_g$ of $\mathcal{G}$ is the representation of a compact subgroup in the Lie
algebra $\mathfrak{n}$ of $N$,  there  exists a metric $K$ on $\mathfrak{n}$
such that $\mathcal{G}$ is a subgroup of $O(\mathfrak{n})$ with respect to $K$.\\

\begin{corollary}\label{corollary_characteristic1}
If $(b_0t,A_0t,\nu_0dt)$ are the characteristics of a semimartingale $Z$ with respect to the Hunt functions $\hat{h}^{\alpha}$ and $\mathcal{G}$ is a
subgroup of $Aut(N)$ with an action satisfying the previous hypothesis, then $\mathcal{G}$ is a gauge symmetry group of $Z$ if and only if
\begin{eqnarray*}
b^{\alpha}_0&=&b_0^{\beta}\Upsilon_{g,\beta}^{\alpha}\\
A^{\alpha\beta}_0&=&\Upsilon_{g,\gamma}^{\alpha}A^{\gamma\delta}_0\Upsilon_{g,\delta}^{\beta}\\
\nu_0&=&\Xi_{g*}(\nu_0).
\end{eqnarray*}
\end{corollary}
\begin{proof}
Since $O_{g,\beta\gamma}^{\alpha}=0$ the only thing to prove is
that
$$\int_N{(\hat{h}^{\alpha}(z')-\Upsilon^{\alpha}_{g,\beta}\hat{h}^{\beta}(\Xi_{g^{-1}}(z')))\nu_0(dz')=0}.$$
But the last equality follows easily from Lemma \ref{lemma_gauge1}. ${}\hfill$ \end{proof}

\begin{remark}
Although all Lie groups $\mathcal{G}$
constructed with the previous method are compact,   not
 all gauge symmetry groups of a L\'evy process are
compact. For example, using Hamiltonian actions on
$\mathbb{R}^n$, it is possible to construct L\'evy processes with gauge
symmetry group $\mathcal{G}=\mathbb{R}^l$.
\end{remark}

\begin{remark}
The construction proposed here for general Lie groups is equivalent to the one considered in \cite{Albeverio2007} for L\'evy processes taking values in the matrix Lie groups.
\end{remark}

\subsection{Gauge symmetries of non-Markovian processes}

In this section we propose a method for the explicit construction of non-Markovian semimartingales admitting  gauge symmetries. We remark that the class of semimartingales obtained in this way does not exhaust  all the possible  non-Markovian semimartingales with gauge symmetries. \\
The main idea of our construction consists in generalizing the following fact: given  three
independent Brownian motions $W^0,W^1,W^2$,  the non-Markovian
process on $\mathbb{R}^2$ defined by the equations
\begin{eqnarray*}
\tilde{W}^1_t&=&\int_0^t{G(W^0_{[0,s]},s)dW^1_s}\\
\tilde{W}^2_t&=&\int_0^t{G(W^0_{[0,s]},s)dW^2_s},
\end{eqnarray*}
where $G$ is a continuous predictable functional on $C^0(\mathbb{R}_+)$, admits the gauge symmetry group  $SO(2)$ of two dimensional rotations. Indeed, if $B_s=(B^{\alpha}_{\beta,s}), s\ge 0,$ is a predictable
process taking values in $SO(2)$, the process $(\hat{W}^1,\hat{W}^2)$ defined by
$\hat{W}^{\alpha}_t=\int_0^t{B^{\alpha}_{\beta,s}d\tilde{W}^{\beta}_s}, t\ge 0,$ has the same law as $(\tilde{W}^1,\tilde{W}^2)$. In fact, if we put
$W'^{\alpha}_t=\int_0^t{B^{\alpha}_{\beta,s}dW^{\beta}_s}$, it is easy to prove that
$$\hat{W}^{\alpha}_t=\int_0^t{G(W^0_{[0,s]},s) dW'^{\alpha}_s}.$$
Since $[W'^{\alpha},W^0]_t=0$, and since $W'^{\alpha}$ is a Brownian motion, $W^0,W'^1,W'^2$ are all independent Brownian motions.
Since $\hat{W}^1,\hat{W}^2$ are the integrals with respect to two independent Brownian motions of a function of a third independent
Brownian motion $W^0$, we know that $\hat{W}^1$ and $\hat{W}^2$ have the same law as $\tilde{W}^1,\tilde{W}^2$.\\

Working in a more general setting, we consider the Lie group $N=N_1 \times N_2$, where $N_1,N_2$ are two Lie groups and the multiplication on $N$ is defined by
$$(z_1,z_2) \cdot (z'_1,z'_2)= (z_1 \cdot_1 z'_1, z_2 \cdot_2 z_2),$$
where   $\cdot_1,\cdot_2$ denote the multiplication on
$N_1,N_2$, respectively.
Moreover, we introduce the space $\Omega_A=\Omega_A^1 \times \Omega_A^2$, where $\Omega_A^i=\mathcal{D}_{1_{N_i}}([0, + \infty),N_i)$,
and we denote by $\omega_A^1,\omega^2_A$ the elements of $\Omega_A^1,\Omega_A^2$, respectively.\\

\begin{theorem}\label{theorem_characteristic4}
Consider $\Xi_g=(\Xi^1_g,id_{N_2})$ and suppose that the characteristics of a semimartingale $Z$ in $N$ depend only on $\omega_A^2$. If the
semimartingale $Z$ admits the Lie group $\mathcal{G}$ with action $\Xi_g$ as a gauge symmetry group then, for any $g \in \mathcal{G}$,
 \small
\begin{eqnarray}
b^{\alpha}_t(\omega^2_A)&=&\Upsilon^{\alpha}_{g,\beta}b_t^{\beta}(\omega^2_A)+\frac{1}{2}O^{\alpha}_{g,\beta\gamma}A^{\beta\gamma}(\omega^2_A)+
\int_0^t{\int_N{(h^{\alpha}(z')-h^{\beta}(\Xi_{g^{-1}}(z'))\Upsilon^{\alpha}_{g,\beta})\nu(\omega^2_A,ds,dz')}}\label{equation_characteristic11}\\
A^{\alpha\beta}_t(\omega^2_A)&=&\Upsilon^{\alpha}_{g,\gamma}\Upsilon^{\beta}_{g,\delta}A^{\gamma\delta}_t(\omega^2_A)\label{equation_characteristic12}\\
\nu(\omega^2_A,dt,dz)&=&\Xi_{g*}(\nu(\omega^2_A,dt,dz)).\label{equation_characteristic13}
\end{eqnarray}
\normalsize Moreover, if the triplet $(b,A,\nu)$ uniquely determines  the law of $Z$ on $\Omega_A$, then equations \refeqn{equation_characteristic11},
\refeqn{equation_characteristic12} and \refeqn{equation_characteristic13} provide  a sufficient condition too.
\end{theorem}
\begin{proof}
The proof is based on Theorem \ref{theorem_characteristic2}
and on the fact that the map $\Lambda'$ appearing in  Theorem \ref{theorem_characteristic2}  has here the form
$$\Lambda'=\left(\begin{array}{c}
\Lambda'^1_A\\
id_{\Omega_A^2}\\
id_{\Omega_B}\end{array} \right).$$
In particular, for proving the necessity it is enough to consider the constant process $G_t=g_0$ and apply Theorem \ref{theorem_characteristic2}.\\
The proof of the sufficiency of equations
\refeqn{equation_characteristic11},
\refeqn{equation_characteristic12} and
\refeqn{equation_characteristic13} is  similar to
the proof of Theorem \ref{theorem_characteristic3}.
\hfill\end{proof}\\

Let us apply Theorem \ref{theorem_characteristic4} to the  example described at the beginning of this subsection. In this case $(\tilde{W}^1,\tilde{W}^2,W^0)$, as a semimartingale
on $\mathbb{R}^3$,  has characteristics
\begin{eqnarray*}
db_t&=&0\\
dA_t&=&\left(\begin{array}{ccc}
(G(W^0_{[0,t]},t))^2dt & 0 & 0 \\
0 & (G(W^0_{[0,t]},t))^2dt & 0 \\
0 & 0 & dt
\end{array} \right)\\
\nu&=&0,
\end{eqnarray*}
where the Hunt functions can be chosen arbitrarily.\\
Here $\Xi^1_B(z)=B \cdot z$, where $B \in SO(2)$ and so $\Upsilon_B=\left(\begin{array}{cc}B & 0\\
0 &0 \end{array} \right)$ and $O_B=0$. It is easy to prove that
$\Upsilon_B \cdot b=0=b$, $\Upsilon_B \cdot A \cdot \Upsilon_B^T=A$ and $\Xi_g(\nu)=0=\nu$.
For a suitable choice of $G$, for example by choosing $G$ Lipschitz with respect to the natural seminorms of $C^0(\mathbb{R}_+)$,
the triplet $(b,A,\nu)$  uniquely determines the law of $(\tilde{W}^1,\tilde{W}^2,W^0)$ and, therefore, we can apply Theorem \ref{theorem_characteristic4}.\\

All the results of  Subsection \ref{subsection_gauge_Levy} can be generalized in many ways which still permit to apply Theorem
\ref{theorem_characteristic4},  obtaining thus other examples of non-Markovian semimartingales with gauge symmetries.

\section{Time symmetries}\label{section_time}

In this section we briefly discuss the time symmetries of a semimartingale on a Lie group. After recalling some properties of the absolutely
continuous time change, we introduce the definition
of time symmetry of a semimartingale, and  we prove some results analogous to those holding for  gauge symmetries. \\
Finally we study time symmetries of L\'evy processes,  constructing some explicit examples of L\'evy processes with non-trivial
time symmetry. Our construction mainly follows \cite{Kunita1994,Kunita1994(1)}.\\

\subsection{Time symmetries of semimartingales}

Given  a positive adapted stochastic process  $\beta$ such that, for
any $\omega \in \Omega$, the function $\beta(\omega):t \mapsto
\beta_t(\omega)$ is absolutely continuous with strictly positive
locally bounded derivative, we define
$$\alpha_t=\inf\{s|\beta_s>t\},$$
where, as usual, by convention $\inf(\mathbb{R}_+)=+\infty$. The process
$\alpha$ is an adapted process such that
$$\beta_{\alpha_t}=\alpha_{\beta_t}=t.$$
If $X$ is a stochastic process adapted to the filtration $\mathcal{F}_t$, we denote by $H_{\beta}(X)$ the stochastic process adapted to the
filtration $\mathcal{F}'_t=\mathcal{F}_{\alpha_t}$ such that
$$H_{\beta}(X)_t=X_{\alpha_t}.$$

Since, by assumption,
$\beta_t$ is absolutely continuous and strictly increasing, then also $\alpha_t$ is absolutely continuous and strictly increasing.  Furthermore,
 denoting by $\alpha'_t$ respectively $\beta'_t$ the time derivative of $\alpha_t$ respectively $\beta_t$,  we have
$$\alpha'_t=\frac{1}{\beta'_{\alpha_t}}.$$
If $\mu$ is a random measure on $N$ adapted to the filtration $\mathcal{F}_t$, we can introduce a time changed random measure $H_{\beta}(\mu)$
adapted to the filtration $\mathcal{F}'_t$ such that, for any Borel set $E \subset N$,
$$H_{\beta}(\mu)([0,t] \times E)=\mu([0,\alpha_t] \times E).$$
In order to introduce a good  concept of symmetry with respect to time transformations, we have to recall some fundamental properties of absolutely
continuous random time changes    with a locally bounded derivative.

\begin{theorem}\label{theorem_time1}
Let $\beta_t$ be the process described above and let $Z, Z'$ be two real semimartingales, $K_t$ be a predictable process which is integrable with
respect to $Z$ and  $\mu$ be a random measure. Then
\begin{enumerate}
\item $H_{\beta}(Z)$ is a semimartingale,
\item if $Z$ is a local $\mathcal{F}_t$-martingale, then $H_{\beta}(Z)$ is a local $\mathcal{F}'_t$-martingale,
\item $H_{\beta}([Z,Z'])=[H_{\beta}(Z),H_{\beta}(Z')]$
\item $H_{\beta}(K)$ is integrable with respect to $H_{\beta}(Z)$ and $\int_0^{\alpha_t}{K_sdZ_s}=\int_0^t{H_{\beta}(K)_s dH_{\beta}(Z)_s}.$
\item if $\mu^p$ is the compensator of $\mu$, then $H_{\beta}(\mu^p)$ is the compensator of $H_{\beta}(\mu)$.
\end{enumerate}
\end{theorem}
\begin{proof}
Since the random time change $\beta$ is continuous, $\beta$ is an \emph{adapted change of time} in the meaning of \cite{Jacod1979}( Chapter X, Section b)).\\
Thank to  this remark the proofs of assertions 1, ..., 5 can be found in \cite{Jacod1979}( Chapter X, Sections b) and c)). ${}\hfill$
\end{proof}\\

Taking into account Theorem \ref{theorem_time1}, a quite natural definition of time symmetry seems at first view to be the following: a semimartingale $Z$ has time symmetries if, for any $\beta$  satisfying  the
previous hypotheses, $Z$ and $H_{\beta}(Z)$ have the same law. Unfortunately, using  for example standard deterministic time changes, it is
possible to prove that the only process satisfying the previous definition is the process almost surely equal to a constant. For this reason we
introduce the following, different,  definition, which has the advantage of admitting non-trivial examples.

\begin{definition}\label{definition_time}
Let $Z$ be a semimartingale on a Lie group $N$ and let $\Gamma_\cdot:N \times \mathbb{R}_+ \rightarrow N$ be an $\mathbb{R}_+$ action such that
$\Gamma_r(1_N)=1_N$ for any $r \in \mathbb{R}_+$. We say that $Z$ has a time symmetry with action $\Gamma_r$ with respect to the filtration $\mathcal{F}_t$ if
$$dZ'_t=H_{\beta}(\Gamma_{\beta'_t}(dZ_t)),$$
has the same law of $Z$ for any $\beta_t$  satisfying the previous hypotheses and such that $\beta'_t$ is a $\mathcal{F}_t$-predictable locally
bounded process in $\mathbb{R}_+$.
\end{definition}

\begin{remark}
The request that $\beta'_t$ is a locally bounded process in $\mathbb{R}_+$ ensures  that $\beta'_t(\omega) \geq c(\omega) >0$ for some $c(\omega)
\in \mathbb{R}_+$ and for $t$ in compact subsets of $\mathbb{R}_+$.
\end{remark}

\begin{lemma}\label{lemma_time2}
If $(X,Z)$ is a solution to the SDE $\overline{\Psi}_{K_t}$ and  $\beta$ is an absolutely continuous process such that $\beta'_t$ is locally bounded in
$\mathbb{R}_+$, then $(H_{\beta}(X),H_{\beta}(Z))$ is a solution to the SDE $\overline{\Psi}_{H_{\beta}(K)_t}$.
\end{lemma}
\begin{proof}
The thesis is a simple consequence of Definition \ref{definition_solution} and Theorem \ref{theorem_time1}, point 4. ${}\hfill$ \end{proof}\\

We now prove the analogue of Proposition \ref{proposition_gauge2} for our time symmetry.

\begin{proposition}\label{proposition_time1}
Given two Lie groups $N$ and $\tilde{N}$, let $Z$ be a semimartingale on  $N$ with the time symmetry $\Gamma_r$ and let $\Theta:N \rightarrow \tilde{N}$ be a
diffeomorphism such that $\Theta(1_N)=1_{\tilde{N}}$. Then the process $d\tilde{Z}_t=\Theta(dZ_t)$ is a semimartingale with  the time symmetry $\Theta
\circ \Gamma_r \circ \Theta^{-1}$.
\end{proposition}
\begin{proof}
From Corollary \ref{corollary_gauge1} we have that $dZ_t=\Theta^{-1}(d\tilde{Z}_t)$ and, since $\Gamma_r$ is a time symmetry for $Z$, if
$dZ'_t=\Gamma_{\beta'_t}(\Theta^{-1}(d\tilde{Z}_t))$, then $H_{\beta}(Z')$ has the same law as $Z$. Hence, by the uniqueness of the solution to a
geometrical SDE, $d\tilde{Z}_t'=\Theta(dH_{\beta}(Z')_t)$ has the same law as $\tilde{Z}$. On the other hand from Lemma \ref{lemma_time2}, we
have $H_{\beta}(\Theta(dZ'_t))=\Theta(dH_{\beta}(Z)_t)$. ${}\hfill$ \end{proof}

\begin{lemma}\label{lemma_time3}
Let $Z$ be a semimartingale with characteristics $(b,A,\nu)$. Then $H_{\beta}(Z)$ has characteristics
$(H_{\beta}(b),H_{\beta}(A),H_{\beta}(\nu))$.
\end{lemma}
\begin{proof}
First we recall that $\nu$ is the compensator of the random measure $\mu^Z$ defined by
$$\mu^Z(\omega,dt,dz)=\sum_{s \geq 0}I_{\Delta Z_s \not =1_N}\delta_{(s,\Delta Z_s(\omega))}(dt,dz)$$
(see the proof of Theorem \ref{theorem_characteristic1}). This means that the random measure associated with $\bar{Z}=H_{\beta}(Z)$ is
\begin{eqnarray*}
\mu^{\bar{Z}}(\omega,dt,dz)&=&\sum_{s \geq 0}I_{\Delta Z_{\alpha_s} \not =1_N}\delta_{(s,\Delta Z_{\alpha_s}(\omega))}(dt,dz)\\
&=&H_{\beta}(\mu^{Z}).
\end{eqnarray*}
Since, by Theorem \ref{theorem_time1}, $H_{\beta}(\mu^Z)$ has $H_{\beta}(\nu)$ as compensator, the characteristic measure of
$H_{\beta}(Z)$ is $H_{\beta}(\nu)$.\\
The proof for $b$ and $A$ is similar and follows from the definition of characteristics and points 2, 3 and 4 of Theorem \ref{theorem_time1}.
${}\hfill$ \end{proof}\\

We  shall now discuss a version of Theorem \ref{theorem_characteristic2} for time symmetries, considering  $\Omega_B$ as the set of locally bounded
functions from $\mathbb{R}_+$ into itself and the process $\beta_t$ defined by
$$\beta_t=\int_0^t{\omega_B(s)ds}.$$
The map $\Lambda: \Omega^c \rightarrow \Omega^c$ (see  Subsection \ref{subsection_gauge_main}) is the composition of two functions: the map $\Lambda_{\Gamma_r}$ induced by the solution to the
SDE $\Gamma_{\beta'_t}(dZ_t)$, as in Subsection \ref{subsection_gauge_main}, and the map $H_{\beta}$, induced by the time transformation, from
$\Omega^c$ into itself, defined  by
$$H_{\beta}(\omega_A(t),\omega_B(t))=(\omega_A(\alpha_t),\omega_B(\alpha_t)).$$
Since both $\Lambda_{\Gamma_r}$ and $H_{\beta}$ are invertible, also $\Lambda$ is invertible and we denote by $\Lambda'$ its inverse.\\
In the same way  we introduced the linear maps $\Upsilon_g$ and $O_g$ for the $\mathcal{G}$-action, in the present case we  introduce two maps $\gamma_r:
\mathfrak{n} \rightarrow \mathfrak{n}$ and $Q_r:\mathfrak{n} \times \mathfrak{n} \rightarrow \mathfrak{n}$ such that, for any smooth function
$f:N \rightarrow N$ and for any right invariant vector fields $Y,Y'$,
\begin{eqnarray*}
Y^z(f(\Gamma_r(z) \cdot \tilde{z}))|_{z=1_N}&=&\gamma_r(Y)(f)(\tilde{z})\\
Y'^z(Y^z(f(\Gamma_r(z) \cdot \tilde{z})))|_{z=1_N}&=&\gamma_r(Y')(\gamma_r(Y)(f))(\tilde{z})+
Q_r(Y,Y')(f)(\tilde{z}).
\end{eqnarray*}
If $Y_1,...,Y_n$ is a basis of right-invariant vector fields, we denote by $\gamma^{\alpha}_{r,\beta},Q^{\alpha}_{r,\beta\gamma}$ the components
of the maps $\gamma_r,Q_r$ with respect to the basis $Y_1,...,Y_n$.

\begin{theorem}\label{theorem_time2}
Let $Z$ be a semimartingale on a Lie group $N$ with characteristics $(b(\omega_A),A(\omega_A),\nu(\omega_A))$. If $Z$ has a time symmetry with
action $\Gamma_r$ then, for any probability measure on $\Omega^c$ such that $\tilde{\mathcal{F}}_t$ is a generalized natural filtration with
respect to both $Z_t$ and $d\tilde{Z}_t=\Xi_{G_t}(dZ_t)$, we have that
\begin{eqnarray*}
db^{\alpha}_t(\omega)&=&\gamma^{\alpha}_{H_{\beta}(\beta'_t(\omega_B)),\beta}dH_{\beta}(b)^{\beta}(\Lambda'(\omega))+\frac{1}{2}Q^{\alpha}_{H_{\beta}(\beta')_t(\omega_B),\beta\gamma}
dH_{\beta}(A^{\beta\gamma})_t(\Lambda'(\omega)))+\\
&&+\left(\int_N{(h^{\alpha}(z')-h^{\beta}(\Gamma_{r^{-1}}(z'))\gamma^{\alpha}_{\beta'_t(\omega_B),\beta})H_{\beta}(\nu(\Lambda'(\omega),dt,dz'))}\right)\\
dA_t^{\alpha\beta}(\omega)&=&\gamma^{\alpha}_{H_{\beta}(\beta')_t(\omega_B),\gamma}\gamma^{\beta}_{H_{\beta}(\beta')_t(\omega_B),\delta}
dH_{\beta}(A)_t^{\gamma\delta}(\Lambda'(\omega))\\
\nu(\omega,dt,dz)&=&\Gamma_{H_{\beta}(\beta')_t(\omega_B)*}(H_{\beta}(\nu(\Lambda'(\omega),dt,dz))),
\end{eqnarray*}
up to a $\mathbb{P}'=\Lambda_*(\mathbb{P})$ null set. Furthermore, if $\tilde{b},\tilde{A},\tilde{\nu}$ are $\pi^{-1}_A(\mathcal{F}^A)$
measurable,  the previous equalities hold with respect to null sets of the law of $Z$.\\
Finally, if the triplet $(b,A,\nu)$ uniquely determines  the law of $Z$,
the previous conditions are also sufficient for the existence of a time
symmetry.
\end{theorem}
\begin{proof}
The proof is completely similar to the proof of Theorem \ref{theorem_characteristic2}, using  Lemma \ref{lemma_time3} in addition to Lemma
\ref{lemma_characteristic1}. ${}\hfill$ \end{proof}

\subsection{L\'evy processes with time symmetries}

In this section we restrict our attention to L\'evy processes on $N$, proving  some general results about L\'evy processes with  time symmetries and providing
 explicit  examples.

\begin{theorem}\label{theorem_time3}
If $Z$ is a L\'evy process with characteristic triplet $(b_0 t,A_0 t, \nu_0dt)$, which uniquely determines the law of $Z$, then $Z$ admits a time symmetry with action $\Gamma_r$ if and only
if,  for any fixed  $r \in \mathbb{R}_+$,
\begin{eqnarray*}
b_0^{\alpha}&=&\frac{1}{r}\left(\gamma^{\alpha}_{r,\beta}b^{\beta}_0+Q^{\alpha}_{r,\beta\gamma}A^{\beta\gamma}_0\right)+\\
&&+\frac{1}{r}\int_N{(h^{\alpha}(z')-h^{\beta}(\Gamma_{r^{-1}}(z))\gamma^{\alpha}_{r,\beta})\nu_0(dz')}\\
A_0^{\alpha\beta}&=&\frac{1}{r} \gamma^{\alpha}_{r,\gamma}\gamma^{\beta}_{r,\delta}A^{\gamma\delta}_0\\
\nu_0(dz)&=&\frac{1}{r}\Gamma_{r*}(\nu_0(dz)).
\end{eqnarray*}
\end{theorem}
\begin{proof}
The proof is similar to the one of Theorem \ref{theorem_characteristic3}, where Theorem \ref{theorem_characteristic2} is replaced by Theorem
\ref{theorem_time2}. ${}\hfill$ \end{proof}\\

As in the case of gauge symmetries,  also in the case of time symmetries the most difficult task  is the construction of suitable Hunt functions
satisfying  the relations in Theorem \ref{theorem_time3}. For this reason we start by considering  stable processes on
nilpotent Lie groups. In the case where $N=\mathbb{R}^n$, $\alpha$-stable processes are well known since their generator is the
fractional
Laplacian, and  they can be obtained by a subordination from a Brownian motion (see, e.g., \cite{Albeverio2007,Applebaum2004}). \\
The homogeneous $\alpha$-stable processes are L\'evy processes in $\mathbb{R}^n$  depending on a parameter $\alpha \in (0,2]$. If the
parameter $\alpha=2$, then $Z$ is a $n$ dimensional Brownian motion with  generator
$$L_2=\frac{1}{2}\sum_{\beta=1}^n \partial_{z^{\beta}z^{\beta}}.$$
For $\alpha \in (0,2)$ $Z$ is a pure jump L\'evy process with L\'evy measure
$$\nu_{\alpha}(dz)=\frac{1}{|z|^{n+\frac{\alpha}{2}}}dz,$$
where $|\cdot|$ is the standard norm of $\mathbb{R}^n$ and $dz$ is the Lebesgue measure.\\
The generator $L_{\alpha}$ of an $\alpha$-stable process is
$$L_{\alpha}(f)(z)=\int_{\mathbb{R}^n}{\left(f(z+z')-f(z)-I_{|z'|<1}(z')\left(
z'^{\beta}\partial_{z^{\beta}}(f)(z)\right)\right)\nu_{\alpha}(dz')}.$$
\\
Given $B \in O(n)$, let $\Xi_{B}$ be the standard action of $B$ on $\mathbb{R}^n$. Since, by definition, $\Xi_{B}$
preserves the standard metric on $\mathbb{R}^n$,  Corollary \ref{corollary_characteristic1}
implies that $O(n)$ is a gauge symmetry of $Z$ with $\alpha=2$.
Using Corollary \ref{corollary_characteristic1} we  obtain the same result for $\alpha \in (0,2)$, with $\Xi_{B}^*(\nu)=\nu$.\\
Furthermore, the $\mathbb{R}_+$ action
$$\Gamma^{\alpha}_r(z)=r^{\frac{1}{\alpha}}z,$$
is a time symmetry for $Z$. For the Brownian motion case, Theorem \ref{theorem_time2} can be applied directly. For $\alpha \in (0,2)$ it is enough
to observe that the space homogeneity of $\nu(z)$ ensures that
$$\int_{\mathbb{R}^n}(I_{B}(z')-I_{\Gamma_{1/r}(B)})z'^{\alpha}\nu(dz')=0.$$
Moreover,  it is easy to see that $Q^{\alpha}_{\beta\gamma}=0$ and $\Gamma^{\alpha}_{r*}(\nu(z))=r\nu$.
Hence, as a consequence of Theorem \ref{theorem_time2}, the homogeneous $\alpha$-stable processes have time symmetry with respect to the action $\Gamma^{\alpha}_r$.\\

In the following we generalize this construction to some nilpotent group which admits dilations.
The presence of dilations is essential to construct L\'evy measures  satisfying the hypotheses of
Theorem \ref{theorem_time2}. Although the construction proposed is well known and can be found in
\cite{Kunita1994,Kunita1994(1)}, for the convenience of the reader, in the following  we summarize the main steps. \\
Given a  simply connected nilpotent group $N$ and its
 Lie algebra $\mathfrak{n}$, the exponential map
$\exp:\mathfrak{n} \rightarrow N$ is a diffeomorphism. Let
$\Gamma_r:N \rightarrow N$ be a subset of automorphisms of $N$ such
that
$$\Gamma_r \circ \Gamma_s=\Gamma_{rs},$$
and $\Gamma_1=Id_N$. We say that $\Gamma_r$ is a dilation on $N$ if, for any $n \in N$,  $\Gamma_r(n) \rightarrow 1_N$ uniformly on compact sets as
$r \rightarrow 0$.

\begin{remark}
It is important to note that not all Lie groups admit a dilation.  Indeed a necessary condition for  $N$ to admit a dilation is that $N$ is
simply connected and nilpotent (this condition is only necessary, but not sufficient, see, e.g., \cite{Folland1982}).
\end{remark}

Using  the properties of composition of $\Gamma_r$, we can
prove that there exists a linear transformation $S$ of
$\mathfrak{n}$ such that
$$\Gamma_r=\exp(\log(r)S).$$
Moreover,  $S$ is a derivation
of $\mathfrak{n}$,  which means
$$S([Y_1,Y_2])=[S(Y_1),Y_2]+[Y_1,S(Y_2)]$$
and the linear transformation $S$ decomposes in a natural way the Lie
algebra $\mathfrak{n}$. Indeed, let $g$ be  the minimal polynomial
of $S$ and  factorize $g=g_1^{n_1}...g_p^{n_p}$, where
$g_1,...,g_p$ are monic irreducible factors of $g$ and $n_j$ are
positive integers.  If we write $\mathfrak{n}_j=\ker(g_j(S)^{n_j})$,
it is simple to prove that $\mathfrak{n}_j$ are invariant subspaces
for $S$ and $\mathfrak{n}=\bigoplus_{j=1}^p \mathfrak{n}_j$. Let
$\kappa_j=\alpha_j \pm i \beta_j$ (where $\alpha_j,\beta_j \in
\mathbb{R}$), be the eigenvalue associated with the space
$\mathfrak{n}_j$ and put
\begin{eqnarray*}
I&=&\{1\leq j \leq p | \alpha_j=\frac{1}{2}\}\\
J&=&\{1\leq j \leq p | \frac{1}{2}<\alpha_j \}\\
I_1&=&\{1\leq j \leq p|\alpha_j=1\}\\
J_1&=&\{1\leq j \leq p|\frac{1}{2}<\alpha_j<1\}.
\end{eqnarray*}
If $K \subset \{1,...,p\}$, we write $\mathfrak{n}_K=\bigoplus_{j \in K}\mathfrak{n}_j$. We denote by $P_{\mathfrak{n}_K}$ the projection onto
the space $\mathfrak{n}_K$ given by the decomposition of $\mathfrak{n}$ into the subspaces $\mathfrak{n}_j$. If $1$ is not eigenvalue of $S$,
then $S-I$ is invertible. If $1$ is eigenvalues of $S$ we can suppose that $\kappa_1=1$, and we can decompose the space $\mathfrak{n}_1$ into
two subspaces $\tilde{\mathfrak{n}}_1=\{(S-I)(Y)|Y\in \mathfrak{n}_1$ and $\hat{\mathfrak{n}}_1=\{Y \in \mathfrak{n}_1|S(Y)=Y\}$. We can define
a pseudo-inverse $(S-I)^{-1}$ of $(S-I)$ such that, fixing $K_1,...,K_m \in \mathfrak{n}_1 \backslash \hat{\mathfrak{n}}_1$ linearly independent
such that $\spann\{(S-I)(K_1),...,(S-I)(K_m)\}=\tilde{\mathfrak{n}}_1$ and putting $V=\bigoplus_{j \not= 1} \mathfrak{n}_j$,  we have
$$(S-I)^{-1} \circ (Q-I)=P_{V \oplus \spann\{K_1,...,K_m\}}.$$
Choose on $\mathfrak{n}$ a metric $\langle \cdot, \cdot \rangle$ with norm $|\cdot |$ and define
$$K=\{Y \in \mathfrak{n}||Y|=1,|r^S(Y)|>1\text{ for any }r>1\}.$$
If $Y \in \mathfrak{n} \backslash \{0\}$, there exist an unique $\theta \in S$ and an unique $r \in \mathbb{R}_+$ such that $r^Q(\theta)=Y$. The
relation described above defines two smooth functions $\theta:\mathfrak{n} \backslash \{0\} \rightarrow S$ and
$r:\mathfrak{n} \backslash \{0\} \rightarrow \mathbb{R}_+$.\\

\begin{theorem}\label{theorem_time4}
A L\'evy process $Z$ on a nilpotent Lie group $N$ with dilation
 $\Gamma_r$ has the time symmetry with respect to the action $\Gamma_r$ if and only if,
denoting by $(A_0t,b_0t,\nu_0dt)$ the characteristics of $Z$ with respect to the Hunt functions
$h^{\alpha}(z)=\frac{\log^{\alpha}(z)}{1+|z^{\alpha}Y_{\alpha}|^2}$ and writing $M=\log_*(\nu_0)$ where $\log=\exp^{-1}:N \rightarrow
\mathfrak{n}$ and $\log^{\alpha}$ is the component of $\log$ with respect to the basis $Y_1,...,Y_n$ of $\mathfrak{n}$, the following conditions
hold
\begin{enumerate}
\item \label{condition_time_A} $P_{\mathfrak{n}_I} \cdot A \cdot
P^T_{\mathfrak{n}_I}=A$ where $P^T_{\mathfrak{n}_I}$ is the
transpose of $P_{\mathfrak{n}_I}$,
\item \label{condition_time_M}
the support of the measure $M$ is contained in the subspace
$\mathfrak{n}_J$ and
$$dM(Y)=\frac{d\lambda(\theta(Y))d(r(Y))}{(r(Y))^2},$$
where $\lambda$ is a measure on the set $K$, \item \label{condition_time_b} if $\kappa_1 \not = 1$ then $b^{\alpha} Y_{\alpha}
=B_1=\int_{\mathfrak{n}}{\frac{\langle S (Y),Y \rangle }{(1+|Y|^2)^2}(S-I)^{-1}(Y)dM(Y)}$ otherwise $b^iY_i -B_1 \in \hat{\mathfrak{n}}_1$,
\item \label{condition_time_if1_1} if $\kappa_1=1$
$$\int_{\mathfrak{n}}{\frac{\langle S(Y), Y \rangle}{(1+|Y|^2)^2}P_{\mathfrak{n}_1}(Y)dM(Y)} \in \tilde{\mathfrak{n}}_1,$$
\end{enumerate}
\end{theorem}
\begin{proof}
Thank to Theorem \ref{theorem_time3}, the statements of this theorem are equivalent to the corresponding statements on the stable processes in
\cite{Kunita1994,Kunita1994(1)}. ${}\hfill$ \end{proof}\\

It is important to note the big  difference between Theorem \ref{theorem_time4} and the construction of L\'evy processes with
gauge symmetries proposed in subsection \ref{subsection_gauge_Levy}. In fact, if we have a compact Lie group $\mathcal{G}$ of
automorphisms on a Lie group $N$, we can construct several  L\'evy processes on $N$  admitting $\mathcal{G}$ as group of gauge symmetries. On the other hand,
it is not true that any dilation $\Gamma_r$ on a nilpotent Lie group $N$ gives rise to  L\'evy processes with the time symmetry with respect
to the action $\Gamma_r$. Indeed in this case the spectral decomposition of the linear
operator $S$ associated with $\Gamma_r$ plays an important role.\\
Furthermore, in the case of gauge symmetries, it is possible to construct L\'evy processes on $N$ with a
 gauge symmetry and continuous and discontinuous parts can be non trivial. On the contrary, in the case of time symmetries
 this is not possible. Indeed the space $\mathfrak{n}_J$ (where
 the jumps are supported) and the space $\mathfrak{n}_I$ (where the continuous martingale part is supported) are
 complementary. The reason is that the behaviour of the transformation $\Gamma_r$ as $r \rightarrow 0$
 is essential for the characterization of the kind of L\'evy process with the time symmetry $\Gamma_r$.\\
This property of time symmetric L\'evy processes seems to keep holding   even if we drop the request that $\Gamma_r$ is an automorphism of $N$.
Indeed, although we are able to construct various L\'evy processes with different behaviour at infinity of the measure $\nu_0$, the behaviour of
the measure $\nu_0$ at $1_N$ is similar to the case in which $\Gamma_r$ is an automorphism of $N$. This fact gives  strong restrictions on the
form of L\'evy processes admitting  time symmetries.

\section{Symmetries and invariance properties of an SDE with jumps}\label{section_symmetry}

\subsection{Stochastic transformations}

Let $\mathcal{C}(\mathbb{P}_0)$ (or simply $\mathcal{C}$) be the class of c\acc{a}dl\acc{a}g semimartingales $Z$ on a Lie group $N$  inducing
the same probability measure on $\mathcal{D}([0,T],N)$ (the metric space of c\acc{a}dl\acc{a}g functions taking values in $N$). In order to
generalize to the semimartingale case the notion of weak solution to an SDE driven by a Brownian motion, we introduce the following definition.

\begin{definition}\label{definition_transformation1}
 Given  a semimartingale $X$ on $M$ and  a semimartingale $Z$ on $N$ such that
$Z \in \mathcal{C}$, the pair $(X,Z)$  is called a \emph{process of class $\mathcal{C}$} on $M$. \\
A process $(X,Z)$ of class $\mathcal{C}$ which is a solution to the canonical SDE $\Psi$ is called a \emph{solution of class $\mathcal{C}$ } to
$\Psi$.
\end{definition}

We remark that if $(X,Z)$ and $(X',Z')$ are two solutions of
 class $\mathcal{C}$ and if $X_0$ and  $X'_0$ have the same law, then also  $X$ and $X'$ have the same law. Hereafter we suppose that the filtration $\mathcal{F}_t$, for which $X$ and $Z$ are semimartingales, is a generalized natural filtration for $Z$. In the usual case considered in the following, where $X$ is a solution of a geometrical SDEs driven by $Z$, and thus where $X$ could be chosen adapted with respect to the natural filtration $\mathcal{F}^Z_t$ of $Z$, the above restriction on $\mathcal{F}_t$ is not relevant. \\

In this section we  define a set of transformations which transform a process of class $\mathcal{C}$ into a new process of class $\mathcal{C}$.
This set of transformations depends on the properties of the processes
belonging to the class $\mathcal{C}$.\\
We start by describing the case of processes in $\mathcal{C}$ admitting a gauge symmetry group $\mathcal{G}$ with action
$\Xi_g$ and a time symmetry with action $\Gamma_r$. Afterwards,  we discuss how to extend our approach to more general situations.\\

\begin{definition}
A stochastic transformation from $M$ into $M'$ is a triad $(\Phi,B,\eta)$, where $\Phi$ is a diffeomorphism of $M$ into $M'$, $B:M \rightarrow
\mathcal{G}$ is a smooth function and $\eta:M \rightarrow \mathbb{R}_+$ is a positive smooth function. We denote the set of stochastic
transformations of $M$ into $M'$ by $S_{\mathcal{G}}(M,M')$.
\end{definition}

A stochastic transformation defines a map between the set of stochastic processes of class $\mathcal{C}$ on $M$ into the set of stochastic
processes of class $\mathcal{C}$ on $M'$. The action of the stochastic transformation $T \in S_{\mathcal{G}}(M,M')$ on the stochastic process
$(X,Z)$ is denoted by $(X',Z')=P_T(X,Z)$, and is defined as follows:
\begin{eqnarray*}
X'&=&\Phi\left[H_{\beta^{\eta}}(X)\right]\\
dZ_t'&=&H_{\beta^{\eta}}\left\{\Xi_{B(X_t)}\left[\Gamma_{\eta(X_t)}(dZ_t)\right]\right\},
\end{eqnarray*}
 where $\beta^{\eta}$ is the random time change given by
$$\beta^{\eta}_t=\int_0^t{\eta(X_s)ds}.$$
The second step is to define  an action of a stochastic transformation $T$ on the set of canonical SDEs. This action transforms a canonical SDE
$\Psi$ on $M$ into the canonical SDE $\Psi'=E_T(\Psi)$ on $M'$ defined by
$$\Psi'(x,z)=\Phi\left\{\Psi\left[\Phi^{-1}(x),(\Gamma_{(\eta(\Phi^{-1}(x)))^{-1}} \circ \Xi_{(B(\Phi^{-1}(x)))^{-1}})(z)\right]\right\}.$$

\begin{theorem}\label{theorem_symmetry3}
If $T \in S_{\mathcal{G}}(M,M')$ is a stochastic transformation and  $(X,Z)$ is a  class $\mathcal{C}$ solution  to the canonical SDE $\Psi$,
then $P_T(X,Z)$ is a class $\mathcal{C}$ solution  to the canonical SDE $E_T(\Psi)$.
\end{theorem}
\begin{proof}
The fact that $P_T(X,Z)$ is a process of class $\mathcal{C}$ follows from the symmetries  of $Z$, which are the gauge
symmetry group $\mathcal{G}$ with action $\Xi_g$ and the time symmetry with action $\Gamma_r$.\\
The fact that, if $(X,Z)$ is a solution to $\Psi$, then $P_T(X,Z)$ is a solution to $E_T(\Psi)$, follows from Theorem
\ref{theorem_geometrical1}, Theorem \ref{theorem_gauge1} and Lemma \ref{lemma_time2}.
${}\hfill$\end{proof}\\

If $\mathcal{C}$ contains semimartingales which have only the gauge symmetry group $\mathcal{G}$ but without time symmetry,  the stochastic
transformation $T$ reduces to a pair $(\Phi,B)$ and the action on processes and SDEs is the same as in the general case with $\Gamma_r=Id_N$. The
same argument can be applied in the case of $\mathcal{C}$ containing semimartingales which possess only the time symmetry property. In the case of semimartingales
without neither  gauge  nor  time symmetries,  the stochastic transformations can be identified with the diffeomorphisms $\Phi:M \rightarrow M'$
and the action on the processes is $P_T(X,Z)=(\Phi(X),Z)$. Since these kinds of transformations do not change the driving process $Z$ and  play
a special role in the theory of symmetries we call \emph{strong stochastic transformation} a stochastic transformation of the form
$(\Phi,1_N,1)$.

\subsection{The geometry of stochastic transformations}

In this subsection we  prove that stochastic transformations have some interesting geometric properties. These are an extension to c\acc{a}dl\acc{a}g-semimartingales-driven SDEs of the properties proposed in \cite{DMU1} for SDEs driven by  Brownian motions. \\
In order to keep holding some crucial  geometric properties, in the following we require   an additional property on the maps $\Xi_g$ and
$\Gamma_r$, i.e. the commutation of the   two group actions $\Xi_g$ and $\Gamma_r$. In particular we suppose that
\begin{equation}\label{pippo}
\Xi_g(\Gamma_r(z))=\Gamma_r(\Xi_g(z)),
\end{equation}
for any $z \in N$, $g \in \mathcal{G}$ and $r \in \mathbb{R}_+$.

\begin{remark}
Condition \eqref{pippo} can be weakened by requiring  that the set of diffeomorphisms $\Theta_{(r,g)}=\Gamma_r \circ \Xi_g$ is an action of the
semidirect product $\mathbb{R}_+ \rtimes \mathcal{G}$.
 This means that there exists a smooth action $h_{\cdot}:\mathbb{R}_+ \times \mathcal{G} \rightarrow \mathcal{G}$ of $\mathbb{R}_+$
 on $\mathcal{G}$ such that
$$\Gamma_r \circ \Xi_g=\Xi_{h_{r}(g)} \circ \Gamma_r.$$
The commutative case is included in this general setting by taking $h_r(g)=g$. Since we are not able to construct any concrete semimartingale
with gauge symmetries and time symmetry admitting non trivial  $h_r$ and, on the other hand,  condition $h_r(g)=g$ quite  simplifies  the
exposition, we prefer working with the commutativity assumption.
\end{remark}

We can define a composition between two stochastic transformations $T \in S_{\mathcal{G}}(M,M')$ and 
$T' \in S_{\mathcal{G}}(M',M'')$, where
$T=(\Phi,B,\eta)$ and $T'=(\Phi',B',\eta')$,  by
\begin{equation}\label{equation_symmetry1}
T' \circ T=(\Phi' \circ \Phi,(B' \circ \Phi) \cdot B, (\eta' \circ \Phi)\eta ).
\end{equation}

The above composition has a nice geometrical interpretation. 
A stochastic transformation from $M$ into $M'$ can be identified with an isomorphism from the
trivial right principal bundle $M \times \mathcal{H}$ into the trivial right principal bundle $M' \times \mathcal{H}$, $\mathcal{H}=\mathcal{G} \times \mathbb{R}_+$,  which preserves the
principal bundle structure. If we exploit this identification and the natural isomorphism composition we obtain  formula
\refeqn{equation_symmetry1}
(see \cite{DMU1} for the case $\mathcal{G}=SO(n)$).\\
Composition \refeqn{equation_symmetry1}, for any $T \in S_{\mathcal{G}}(M,M')$, permits to define an inverse $T^{-1} \in S_{\mathcal{G}}(M',M)$
as follows
$$T^{-1}= (\Phi^{-1},(B \circ \Phi^{-1})^{-1},(\eta \circ \Phi^{-1})^{-1}).$$
Hence the set $S_{\mathcal{G}}(M):=S_{\mathcal{G}}(M,M)$ is a group with respect to the composition $\circ$ and the identification
 of $S_{\mathcal{G}}(M)$ with $\Iso(M \times \mathcal{H}, M \times \mathcal{H})$ (which is a closed subgroup of the group of diffeomorphisms of $M \times \mathcal{H}$) suggests to consider  the corresponding Lie algebra $V_{\mathcal{G}}(M)$.\\
Given a one parameter group  $T_a=(\Phi_a,B_a,\eta_a) \in S_{\mathcal{G}}(M)$,  there exist a vector field $Y$ on $M$, a smooth function $C:M
\rightarrow \mathfrak{g}$ (where $\mathfrak{g}$ is the Lie algebra of $\mathcal{G}$), and a smooth function $\tau:M \rightarrow \mathbb{R}$ such
that
\begin{equation}\label{equation_infinitesimal_SDE1}\begin{array}{ccc}
Y(x)&:=&\partial_a(\Phi_a(x))|_{a=0}\\
C(x)&:=&\partial_a(B_a(x))|_{a=0}\\
\tau(x)&:=&\partial_a(\eta_a(x))|_{a=0}.
\end{array}
\end{equation}
So if  $Y,C,\tau$ are  as above, the one parameter solution $(\Phi_a,B_a,\eta_a)$ to the equations
\begin{equation}\label{equation_infinitesimal_SDE2}\begin{array}{rcl}
\partial_a(\Phi_a(x))&=&Y(\Phi_a(x))\\
\partial_a(B_a(x))&=&R_{B_a(x)*}(C(\Phi_a(x)))\\
\partial_a(\eta_a(x))&=&\tau(\Phi_a(x))\eta_a(x),
\end{array}
\end{equation}
with initial condition $\Phi_0=id_M$, $B_0=1_{\mathcal{G}}$ and $\eta_0=1$, is a one parameter group in $S_{\mathcal{G}}(M)$. For this reason we identify the elements of $V_{\mathcal{G}}(M)$ with the triads $(Y,C,\tau)$.\\

\begin{definition}
A triad $V=(Y,C,\tau)\in V_{\mathcal{G}}(M)$, where  $Y$ is a vector field on $M$,  $C:M \rightarrow \mathfrak{g}$  and  $\tau:M \rightarrow
\mathbb{R}$  are smooth functions, is an  infinitesimal stochastic transformation. If $V$ is of the form  $V=(Y,0,0)$,  we call $V$ a strong
infinitesimal stochastic transformation, as the corresponding one-parameter group is a group of strong stochastic transformations.
\end{definition}

Since $V_{\mathcal{G}}(M)$ is a Lie subalgebra of the set of vector fields on $M \times \mathcal{H}$, the standard Lie brackets between vector
fields on $M \times \mathcal{H}$ induce some Lie brackets on $V_{\mathcal{G}}(M)$. Indeed, if $V_1=(Y_1,C_1,\tau_1),V_2=(Y_2,C_2,\tau_2) \in
V_m(M)$ are two infinitesimal stochastic transformations, we have
\begin{equation}\label{equation_infinitesimal_SDE5}
\left[V_1,V_2\right]=(\left[Y_1,Y_2\right],Y_1(C_2)-Y_2(C_2)-\{C_1,C_2\},Y_1(\tau_2)-Y_2(\tau_1)),
\end{equation}
where $\{\cdot,\cdot\}$ denotes the usual commutator between elements of $\mathfrak{g}$.\\
Furthermore, the identification of  $T=(\Phi,B,\eta) \in S_{\mathcal{G}}(M,M')$  with  $F_T\in \Iso(M \times \mathcal{H}, M' \times \mathcal{H})$
allows us to define  the push-forward $T_*(V)$  of $V \in V_{\mathcal{G}}(M)$ as
\begin{equation}\label{equation_infinitesimal_SDE4}
(\Phi_*(Y) ,(Ad_{B}(C)+R_{B^{-1}*}(Y(B))) \circ \Phi^{-1},(\tau+Y(\eta)\eta^{-1})\circ \Phi^{-1}),
\end{equation}
where $Ad$ denotes the adjoint operation and the symbol $Y(B)$  the push-forward of $Y$ with respect to the map $B:M \rightarrow \mathcal{G}$.\\
Analogously, given $V' \in V_{\mathcal{G}}(M')$, we can consider  the pull-back of $V'$  defined as $T^*(V')=(T^{-1})_*(V')$. The following
theorem shows that any Lie algebra of general infinitesimal stochastic transformations satisfying a non-degeneracy condition, can be locally
transformed, by the action of the push-forward of a suitable stochastic transformation $T \in S_{\mathcal{G}}(M)$,  into a Lie algebra of strong
infinitesimal stochastic transformations.

\begin{theorem}\label{theorem_infinitesimal_SDE1}
Let $K=\spann \{V_1,...,V_k\}$ be a Lie algebra of $V_{\mathcal{G}}(M)$ and let $x_0 \in M$ be such that $Y_1(x_0),...,Y_k(x_0)$ are linearly
independent (where $V_i=(Y_i,C_i,\tau_i)$). Then there exist an open neighbourhood $U$ of $x_0$ and a stochastic transformation $T \in
S_{\mathcal{G}}(U)$ of the form $T=(Id_U,B,\eta)$ such that $T_*(V_1),...,T_*(V_k)$ are strong infinitesimal stochastic transformations in
$V_{\mathcal{G}}(U)$. Furthermore the smooth functions $B,\eta$ are solutions to the equations
\begin{eqnarray*}
Y_i(B)&=&-L_{B*}(C_i)\\
Y_i(\eta)&=&-\tau_i \eta,
\end{eqnarray*}
 where $L_g$ is the diffeomorphism given by the left multiplication for $g \in \mathcal{G}$ and  $i=1,...,k$.
\end{theorem}
\begin{proof}
The proof of this theorem, for the   case $\mathcal{G}=SO(n)$, can be found in \cite{DMU1}. Since the proof in \cite{DMU1} does not ever use the
specific group properties of $SO(n)$ but  only the fact that  $SO(n)$ is a Lie group, the
 proof given in \cite{DMU1} holds also in our case.
${}\hfill$\end{proof}\\

Theorem \ref{theorem_infinitesimal_SDE1} plays a very important role in the applications of the symmetry analysis to concrete SDEs. For example,
in the case of SDEs driven by Brownian motion it permits to apply the reduction and reconstruction procedure, using the existence of strong
infinitesimal symmetries, in the general case (see \cite{DMU2}). We give an
 example of application of Theorem \ref{theorem_infinitesimal_SDE1} in Subsection \ref{subsection_example} and we plan to provide further applications  in a forthcoming paper.

\subsection{Symmetries of an SDE with jumps}

\begin{definition}\label{definition_symmetries}
A stochastic transformation $T \in S_{\mathcal{G}}(M)$ is a symmetry of the SDE $\Psi$ if,
for any process $(X,Z)$ of class $\mathcal{C}$ solution to the SDE $\bar{\Psi}$,  also $P_T(X,Z)$ is a solution to the SDE $\Psi$.\\
An infinitesimal stochastic transformation $V \in V_{\mathcal{G}}(M)$ is a symmetry of the SDE $\Psi$ if the one-parameter group of
stochastic transformations $T_a$ generated by $V$ is a group of symmetry of the SDE $\Psi$.
\end{definition}

\begin{remark}
We can give also a local version of Definition \ref{definition_symmetries}: a stochastic transformation $T \in S_{\mathcal{G}}(U,U')$, where
$(U,U')$ are two open sets of $M$, is a symmetry of $\Psi$ if $P_T$ transforms solutions to $\Psi|_{U}$ into solutions to
$\Psi|_{U'}$.\\
In this case it is necessary to stop the solution process $X$ and the driving semimartingale $Z$ with respect to a suitably adapted  stopping
time.
\end{remark}

\begin{theorem}\label{theorem_symmetry1}
A sufficient condition for a stochastic transformation $T\in S_\mathcal{G}(M)$ to be a symmetry of the SDE $\Psi$ is that $E_T(\Psi)=\Psi$.
\end{theorem}
\begin{proof}
This is an easy application of Theorem \ref{theorem_symmetry3}.
${}\hfill$\end{proof}\\

A natural question arising from  previous discussion is whether the condition in Theorem \ref{theorem_symmetry1} is also necessary. Unfortunately,
 even for Brownian motion driven SDEs there are counterexamples
(see \cite{DMU1} where the determining equations for symmetries of Brownian-motion-driven SDEs are different from the equations found here). The reason for this fact is
that, for a general law of the driving semimartingale in the class $\mathcal{C}$, it is possible to find two different canonical SDEs $\Psi \not
= \Psi'$ with the same set of solutions of class $\mathcal{C}$, i.e. any solution $(X,Z)$ of $\Psi$
is also a solution of $\Psi'$ and viceversa. \\
Exploiting this fact it is possible to find sufficient conditions in order to prove  the converse of Theorem \ref{theorem_symmetry1}.\\
In  the following we say that a semimartingale $Z$ in the class $\mathcal{C}$ and with characteristic triplet $(b,A,\nu)$ has \emph{jumps of
any size} if the support of $\nu$ is all $N \times \mathbb{R}_+$ with positive probability.

\begin{lemma}\label{lemma_full}
Given a semimartingale $Z$ in the class $\mathcal{C}$ with jumps of any size and such that the stopping time $\tau$ of the first jump
 is almost surely strictly positive,  if $(X,Z)$ is a solution to both the SDEs $\Psi$ and $\Psi'$ such that $X_0=x_0 \in M$ almost surely,  then
$\Psi(x_0,z)=\Psi'(x_0,z)$ for any $z \in N$.
\end{lemma}
\begin{proof}
Consider the semimartingale $S^f_t=f(X_t)$, where $f \in \cinf(M)$ is a bounded smooth function. Given  a bounded smooth function $h \in
\cinf(\mathbb{R})$ such that $h(x)=0$ for $x $ in a neighbourhood of $0$, we define the (special) semimartingale
$$H^{h,f}_t=\sum_{0 \leq s \leq t}h(\Delta S^f_s).$$
Since the jumps $\Delta S^f_t$ of $S^f$ are exactly $\Delta S^f_t=f(\Psi(X_{t_-},\Delta Z_t))-f(X_{t_-})$ or, equivalently, $\Delta
S^f_t=f(\Psi'(X_{t_-},\Delta Z_t))-f(X_{t_-})$, we have that \small
$$H^{h,f}_t=\int_{N \times [0,t]}{h(f(\Psi(X_{s_-},
z))-f(X_{s_-}))\mu^Z(ds,dz)}=\int_{N \times [0,t]}{h(f(\Psi'(X_{s_-}, z))-f(X_{s_-}))\mu^Z(ds,dz)}.$$ \normalsize
Since $H^{h,f}$ is a special
semimartingale there exists a unique (up to $\mathbb{P}$ null sets) predictable process $R^{h,f}$ of bounded variation such that
$H^{h,f}_t-R^{h,f}_t$ is a local martingale. By the definition of characteristic measure $\nu$  it is simple to prove that
\small
$$R^{h,f}_t=\int_{N \times [0,t]}{h(f(\Psi(X_{s_-}, z))-f(X_{s_-}))\nu(ds,dz)}=\int_{N \times [0,t]}{h(f(\Psi'(X_{s_-}, z))-f(X_{s_-}))\nu(ds,dz)}.$$
\normalsize This means that
$$\int_{N \times [0,t]}{(h(f(\Psi(X_{s_-}, z))-f(X_{s_-}))-h(f(\Psi'(X_{s_-}, z))-f(X_{s_-})))\nu(ds,dz)}$$
is a semimartingale almost surely equal to $0$. Since $X_{t_-}$ is a continuous function for $t\leq \tau$ and  the support of $\nu$ is all $N
\times \mathbb{R}_+$, in a set of positive measure, there exists a set of positive probability such that $h(f(\Psi(X_{t_-},
z))-f(X_{t_-}))-h(f(\Psi'(X_{t_-}, z))-f(X_{t_-}))=0$ for any $z \in N$. Taking the limit $t \rightarrow 0$ we obtain $h(f(\Psi(x_0,
z))-f(x_0))=h(f(\Psi'(x_0, z))-f(x_{0}))$. Since $h,f$ are generic functions, we deduce that $\Psi(x_0,z)=\Psi'(x_0,z)$ for any $z \in N$.
${}\hfill$\end{proof}

\begin{theorem}\label{theorem_full}
In the hypotheses of Lemma \ref{lemma_full}, a stochastic transformation $T \in S_{\mathcal{G}}(M)$ is a symmetry of an SDE $\Psi$ if and only if
$E_T(\Psi)=\Psi$.
\end{theorem}
\begin{proof}
The if part is exactly Theorem \ref{theorem_symmetry1}. \\
Conversely, suppose  that $T$ is a symmetry of $\Psi$ and put $\Psi'=E_T(\Psi)$. If $X^{x_0}$ denotes  the unique solution to the SDE $\Psi$
driven by the semimartingale $Z$ such that $X^{x_0}=x_0$ almost surely, put $(X',Z')=E_T(X^{x_0},Z)$.  By definition of symmetry $(X',Z')$ is a
solution to $\Psi$ and, by Theorem \ref{theorem_symmetry3}, it is a solution to $\Psi'$. Since $X'_0=\Phi(x_0)$ almost surely, using Lemma
\ref{lemma_full} we obtain that $\Psi(\Phi(x_0),z)=\Psi'(\Phi(x_0),z)$. Since $\Phi$ is a diffeomorphism and $x_0 \in M$ is a generic point this
concludes the proof.
${}\hfill$\end{proof}

\begin{remark}
We propose here two possible generalizations of Theorem \ref{theorem_full}\\
First we can suppose that $Z$ is a purely discontinuous semimartingale and that $b^{\alpha}_t=A^{\alpha\beta}_t=0, \forall t \geq 0$ with Hunt
functions $h^{\alpha}=0$. In this case, if  the support of $\nu$ is $J \times \mathbb{R}_+$ almost surely,  the stochastic transformation  $T$
is a symmetry of the SDE $\Psi$ if and only if $E_T(\Psi)(x,z)=\Psi(x,z)$ for any $z \in J$. The proof of the necessity of the condition is
equal to the one in Lemma \ref{lemma_full} and Theorem \ref{theorem_full}, instead the proof of the sufficiency part is essentially based on the
fact that $Z$ is a pure jump process.
This case includes, for example, the Poisson process.\\
The second generalization covers the important case of continuous semimartingales. An example of the theorem which could be obtained in this
case is Theorem 17 in \cite{DMU1} that, in our language,   can be reformulated as follows: $T$ is a symmetry of $\Psi$ driven by a Brownian
motion $Z^2,...,Z^m$ and by the time $Z^1_t=t$ if and only if $\partial_{z^{\alpha}}(\Psi)(x,0)=\partial_{z^{\alpha}}(E_T(\Psi))(x,0)$ for
$\alpha=2,...,m$ and $\partial_{z^1}(\Psi)(x,0)+\frac{1}{2}\sum_{\alpha=2}^m
\partial_{z^{\alpha}z^{\alpha}}(\Psi)(x,0)=\partial_{z^1}(E_T(\Psi))(x,0)
+\frac{1}{2}\sum_{\alpha=2}^m \partial_{z^{\alpha}z^{\alpha}}(E_T(\Psi))(x,0)$. \\
\end{remark}

In order to provide an explicit formulation of the \emph{determining equations} for the infinitesimal symmetries of an SDE $\Psi$, we prove the
following proposition.

\begin{proposition}\label{proposition_determining}
A sufficient condition for an infinitesimal stochastic transformation $V$, generating a one-parameter group $T_a$ of stochastic transformations,
to be an infinitesimal symmetry of an SDE $\Psi$ is that
\begin{equation}\label{equation_determining1}
\partial_a(E_{T_a}(\Psi))|_{a=0}=0.
\end{equation}
When the hypotheses of Theorem \ref{theorem_full} hold, condition \refeqn{equation_determining1} is also necessary.
\end{proposition}
\begin{proof}
We prove that if equation \refeqn{equation_determining1} holds, then $E_{T_a}(\Psi)=\Psi$ for any $a \in \mathbb{R}$. Defining
$\tilde{\Psi}(a,x,z)=E_{T_a}(\Psi)$, the function $\tilde{\Psi}(a,x,z)$ solves a partial differential equation of the form
\begin{equation}\label{equation_determining3}
\partial_a(\tilde{\Psi}(a,x,z))=\mathcal{L}(\tilde{\Psi}(a,x,z))+F(\tilde{\Psi}(a,x,z),x,z),
\end{equation}
where $\mathcal{L}$ is a linear first order scalar differential operator in $\partial_x,\partial_z$ and $F$ is a smooth function. It is possible to prove,
exploiting standard techniques of characteristics for first order PDEs (see \cite{DM2016,DMU1}), that equation \refeqn{equation_determining3}
admits a unique local solution as evolution PDE in the time parameter $a$ for any smooth initial value $\tilde{\Psi}(0,x,z)$.\\
Since   $\tilde{\Psi}(0,x,z)=E_{T_0}(\Psi)(x,z)=\Psi(x,z)$ and
$\mathcal{L}(\Psi(x,z))+F(\Psi(x,z),x,z)=\partial_a(E_{T_a}(\Psi))|_{a=0}=0$, we have that $E_{T_a}(\Psi)(x,z)=\tilde{\Psi}(a,x,z)=\Psi(x,z)$. \\
The necessity of condition \refeqn{equation_determining1} under the hypotheses of Theorem \ref{theorem_full} is trivial since, by Theorem
\ref{theorem_full}, we must have $E_{T_a}(\Psi)=\Psi$.
${}\hfill$\end{proof}\\

In the following we use Proposition \ref{proposition_determining} to rewrite  equations \refeqn{equation_determining1} in any given  coordinate
systems $x^i$ on $M$ and  $z^{\alpha}$ on $N$. We denote by $K_1,...,K_r$ the vector fields  on $N$ generating the action $\Xi_g$ of
$\mathcal{G}$ on $N$ and by $H$ the vector field generating  the action $\Gamma_r$ of $\mathbb{R}_+$ on $N$. Using these  notations, with any
infinitesimal stochastic transformation $V=(Y,C,\tau)$ we associate a vector field $Y$ on $M$, a function $\tau$ and $r$ functions
$C^1(x),...,C^r(x)$
 which correspond to the components of $C$ with respect to the basis $K_1,...,K_r$ of generators of the action $\Xi_g$. In the chosen  coordinate systems on $M,N$
  the vector fields $Y$ and  $K_1,...,K_r,H$ are of the form
$$
Y=Y^i(x)\partial_{x^i} \ \  K_{\ell}=K_{\ell}^{\alpha}(z)\partial_{z^{\alpha}} \ \ H=H^{\alpha}(z)\partial_{z^{\alpha}}.
$$
Therefore, we can rewrite \refeqn{equation_determining1} as
\begin{equation}\label{equation_determining2}
Y^i(\Psi(x,z))-Y^j(x)\partial_{x^j}(\Psi^i)(x,z)-\tau(x) H^{\alpha}(z)\partial_{z^{\alpha}}(\Psi^i)(x,z)-C^{\ell}(x) K^{\alpha}_{\ell}(z)
\partial_{z^{\alpha}}(\Psi^i)(x,z)=0,
\end{equation}
where $\Psi^i(x,z)=x^i \circ \Psi$ and $i=1,...,m$.  In the literature of symmetries of deterministic differential equations,  equations
\refeqn{equation_determining2} are usually called \emph{determining equations} (see, e.g., \cite{Olver1993,Stephani1989}). It is however important to
note some differences with respect to the determining equations of ODEs or also of Brownian-motion-driven SDEs (see \cite{DMU1}). Indeed, in the
deterministic case and in the Brownian  motion case the determining equations are linear and local overdetermined first order differential
equations both in the infinitesimal transformation coefficients and in the equation coefficients. Instead equations
\refeqn{equation_determining2} are \emph{linear non-local} differential equations in the coefficients $Y^i,\tau,C^{\ell}$ of the infinitesimal
transformation $V$, and they are \emph{non-linear local} differential equations in the coefficient
$\Psi^i$ of the SDE. \\

\subsection{An example}\label{subsection_example}

In  order to give an idea of the generality and of the flexibility of our approach, we propose an example of an application of the previous theory.
Further  examples  of  SDEs interesting
for mathematical applications will be given in a forthcoming paper.\\
We consider $M=\mathbb{R}^2$, $N=GL(2) \times \mathbb{R}^2$ (with the natural multiplication), and the canonical SDE
\begin{equation}\label{equation_example}
\Psi(x,z_{(1)},z_{(2)})=z_{(1)} \cdot x+z_{(2)}.
\end{equation}
The SDE associated with $\Psi$ is an affine SDE and its solution $(X,Z)$ satisfies the following stochastic differential relation
\begin{equation}\label{equation_example1} dX^i_t=X^j_{t_-}(Z^{-1}_{(1)})^k_{j,t_-}{dZ^i_{k,(1),t}+dZ^i_{(2),t}},
\end{equation}
where $Z^{-1}_{(1)}$ is the inverse matrix of $Z_{(1)}$ and  $GL(2)$ is naturally embedded in the set of the two by two matrices. If we set
\begin{equation}\label{equation_example2}
\overline{Z}^i_{j,t}=\int_0^t{(Z^{-1}_{(1)})^k_{j,s_-}{dZ^i_{k,(1),s}}},
\end{equation}
equation \refeqn{equation_example1} becomes the most general equation affine both in the noises $\overline{Z},Z_{(2)}$ and in the unknown
process $X$. Furthermore, if the noises $Z_{(1)},Z_{(2)}$ are discrete time semimartingales (i.e. semimartingales with fixed time jumps at times
$n \in \mathbb{N}$) equation \refeqn{equation_example1} becomes $X_n=Z_{(1),n-1}^{-1} \cdot Z_{(1),n} \cdot X_{n-1}+Z_{(2),n}-Z_{(2),n-1}$, that
is an affine type iterated
random map (see Subsection \ref{subsection_iterated} and references therein). \\
The SDE $\Psi$ does not have strong symmetries, in the sense that, for general semimartingales $(Z_{(1)},Z_{(2)})$, equation \eqref{equation_example1} does not admit symmetries. \\
For this reason we suppose that the semimartingales $Z_{(1)},Z_{(2)}$ have the gauge symmetry  group $O(2)$ with the natural action
\begin{equation}\label{equation_example3}
\Xi_{B}(z_{(1)},z_{(2)})=(B \cdot z_{(1)} \cdot B^T, B \cdot z_{(2)}),
\end{equation}
where $B \in O(2)$.\\
In order to use the determining equation \refeqn{equation_determining2} for calculating the infinitesimal symmetries of the SDE $\Psi$, we need
to explicitly write the infinitesimal generator $K$ of the action $\Xi_B$ on $N$. In the standard coordinate system of $N$ we have that $\Xi_B$
is generated by
\begin{eqnarray*}
K&=&(-z^2_{1,(1)}-z^1_{2,(1)})\partial_{z^1_{1,(1)}}+(z^1_{1,(1)}-z^2_{2,(1)})\partial_{z^1_{2,(1)}}+(z^1_{1,(1)}-z^2_{2,(1)})\partial_{z^2_{1,(1)}}+\\
&&+(z^1_{2,(1)}+z^2_{1,(1)})\partial_{z^2_{2,(1)}}-z^2_{(2)}\partial_{z^1_{(2)}}+z^1_{(2)}\partial_{z^2_{(2)}}. \end{eqnarray*} If we set
$$R=\left(\begin{array}{cc}
0 & -1 \\
1 & 0 \end{array} \right) $$ we have that

\begin{eqnarray*}
K(z_{(1)})&=&R \cdot z_{(1)}+ z_{(1)} \cdot R^T\\
K(z_{(2)})&=&R \cdot z_{(2)},
\end{eqnarray*}
where the vector field $K$ is applied componentwise to the matrix $z_{(1)}$ and the vector $z_{(2)}$. Using this property of $K$ we can easily
prove that
$$V=(Y,C)=\left(-x^2\partial_{x^1}+x^1\partial_{x^2}, 1 \right),$$
(where $C=1$ is the component of the gauge symmetry with respect to the generator $K$) is a symmetry of the equation $\Psi$. Indeed, recalling
that $Y$ is a linear vector field whose components satisfy the relation
$$Y=R \cdot x,$$
we have that, in this case, the determining equations \refeqn{equation_determining2} read
\begin{eqnarray*}
Y \circ \Psi-Y(\Psi)-C(x)K(\Psi)&=& R \cdot (z_{(1)} \cdot x+z_{(2)} )-z_{(1)} \cdot (R \cdot x)- K(\Psi)\\
&=&R \cdot (z_{(1)} \cdot x+z_{(2)} )+z_{(1)} \cdot R^T  \cdot x- (R \cdot z_{(1)}+ z_{(1)} \cdot R^T) \cdot x+\\
&&-R \cdot z_{(2)}=0.
\end{eqnarray*}
Since $V$ satisfies the determining equations \refeqn{equation_determining2}, $V$ is an infinitesimal symmetry of $\Psi$. The infinitesimal
stochastic transformation $V$ generates a one-parameter group of symmetries of $\Psi$ given by
$$T_a=(\Phi_a,B_a)=\left(\left(
\begin{array}{cc}
\cos(a) & -\sin(a)\\
\sin(a) & \cos(a) \end{array} \right) \cdot x, \left(
\begin{array}{cc}
\cos(a) & -\sin(a)\\
\sin(a) & \cos(a)
\end{array} \right)  \right).$$
In other words if the law of $(Z_{(1)},Z_{(2)})$ is gauge invariant with respect to rotations then the SDE $\Psi$ is invariant with respect to
rotations.
\\
Once we have found an infinitesimal symmetry, we can exploit it to transform the SDE $\Psi$ in an equation of a simpler form as  done, for
example,  in \cite{DMU2} for Brownian-motion-driven SDEs.\\ The first step consists in looking for a stochastic transformation $T=(\Phi,B)$ such
that $T_*(V)$ is a strong symmetry (the existence of the transformation $T$ is guaranteed by Theorem \ref{theorem_infinitesimal_SDE1}). In this
specific case the transformation $T$ has the following form (for $x=(x^1,x^2)\ne (0,0)$)
\begin{equation}\label{equation_transformation}
T=(\Phi(x),B(x))=\left(\left(\begin{array}{c}
x^1 \\
x^2\end{array} \right),\left(\begin{array}{cc}
\frac{x^1}{\sqrt{(x^1)^2+(x^2)^2}} & \frac{x^2}{\sqrt{(x^1)^2+(x^2)^2}}\\
\frac{-x^2}{\sqrt{(x^1)^2+(x^2)^2}} & \frac{x^1}{\sqrt{(x^1)^2+(x^2)^2}}\end{array}\right)\right) \end{equation} and the SDE $\Psi'=E_T(\Psi)$
becomes for such $x$
$$\Psi'(x,z_{(1)},z_{(2)})=\left(\begin{array}{cc}
x^1 & -x^2\\
x^2 & x^1 \end{array} \right)\cdot \left(\begin{array}{c}
z^{1}_{1,(1)}\\
z^2_{1,(1)} \end{array}\right)+ \left(\begin{array}{cc}
\frac{x^1}{\sqrt{(x^1)^2+(x^2)^2}} & \frac{-x^2}{\sqrt{(x^1)^2+(x^2)^2}}\\
\frac{x^2}{\sqrt{(x^1)^2+(x^2)^2}} & \frac{x^1}{\sqrt{(x^1)^2+(x^2)^2}}\end{array}\right) \cdot \left(\begin{array}{c} z^1_{(2)}\\
z^2_{(2)} \end{array}\right).$$ Note that $\Psi'$ does not depend on $z^{1}_{2,(1)},z^{2}_{2,(1)}$, which means that the noise has been reduced by the transformation. The
transformation $T$ has an effect similar to the reduction of redundant Brownian motions in Brownian-motion-driven SDE (see \cite{Elworthy1999}).
 Moreover, if we rewrite the transformed SDE in (pseudo)-polar coordinates
\begin{eqnarray*}
\rho&=&(x^1)^2+(x^2)^2\\
\theta&=&\argg(x^1,x^2),
\end{eqnarray*}
where $\argg(a,b)$ is the function  giving the measure of the angle between $(0,1)$ and $(a,b)$ in $\mathbb{R}^2$, we find
\begin{equation}\label{equation_polar1}
\begin{array}{ccl}
\Psi'^{\rho}(\rho,\theta,z)&=&(\sqrt{\rho}z^1_{1,(1)}+z^1_{(2)})^2+(\sqrt{\rho}z^2_{1,(1)}+z^2_{(2)})^2\\
\Psi'^{\theta}(\rho,\theta,z)&=&\theta+\argg(\sqrt{\rho}z^1_{1,(1)}+z^1_{(2)},\sqrt{\rho}z^2_{1,(1)}+z^2_{(2)}),
\end{array}
\end{equation}
The canonical SDE defined by $(\Psi^{\rho},\Psi^{\theta})$ is a triangular SDE with respect to the solutions processes $(R_t,\Theta_t)$. Indeed
we have
\begin{equation}\label{equation_R}
\begin{array}{rcl}
 dR_t&=&\left(d\left[Z'^1_{(2)},Z'^1_{(2)}\right]^c_t+d\left[Z'^2_{(2)},Z'^2_{(2)}\right]^c_t+(\Delta Z'^1_{(2),t})^2+(\Delta Z'^2_{(2),t})^2\right)+\\
 &&+\sqrt{R_{t_-}}\left(2dZ'^1_{(2),t}+2d\left[Z'^1_{(2)},\overline{Z}'^1_1\right]^c_t+2d\left[Z'^2_{(2)},\overline{Z}'^2_1\right]^c_t+2\Delta
\overline{Z}'^{1}_{1,t} \Delta Z'^1_{(2),t}+2\Delta \overline{Z}'^{2}_{1,t} \Delta Z'^2_{(2),t}\right)+\\
&&+
R_{t_-}\left(2d\overline{Z}'^1_{1,t}+d\left[\overline{Z}'^1_1,\overline{Z}'^{1}_1\right]_t^c+
d\left[\overline{Z}'^2_1,\overline{Z}'^2_1\right]_t^c+(\Delta
\overline{Z}^{1}_{1,t})^2+(\Delta \overline{Z}^2_{1,t})^2\right)
\end{array}
\end{equation}
\begin{equation}\label{equation_Theta}
\begin{array}{rcl}
d\Theta_t&=&\left(d\overline{Z}'^2_{1,t}-2d\left[\overline{Z}'^{1}_1,\overline{Z}'^2_1\right]^c_t
\right)+\\
&&+\frac{1}{\sqrt{R}_{t_-}}\left(dZ'^2_{(2),t}-d\left[\overline{Z}'^{1}_1,Z'^2_{(2)}\right]^c_t
-2d\left[Z'^{1}_{(2)},\overline{Z}'^2_1\right]^c_t\right)-\frac{1}{R_{t_-}}d\left[Z'^{1}_{(2)},Z'^2_{(2)}\right]^c_t+\\
&&+\left(\argg\left(\sqrt{R_{t_-}}(1+\Delta \overline{Z}'^{1}_{1,t})+\Delta Z'^1_{(2),t},\sqrt{R_{t_-}}(\Delta \overline{Z}'^{2}_{1,t})+\Delta
Z'^2_{(2),t}\right) -\Delta \overline{Z}'^2_{1,t}-\frac{\Delta Z'^2_{(2),t}}{\sqrt{R_{t_-}}}\right),
\end{array}
\end{equation}
where
\begin{eqnarray*}
dZ'^i_{(2),t}&=&B^{i}_j(X_{t_-})dZ^j_{(2),t}\\
dZ'^i_{j,(1),t}&=&Z'^i_{k(1),{t_-}}B^{k}_l(X_{t_-})B^j_r(X_{t_-})(Z^{-1}_{(1)})^{l}_{m,t_-} dZ^m_{r,(1),t}\\
d\overline{Z}'^{i}_{j,t}&=&(Z'^{-1}_{(1)})^i_{k,t_-}dZ'^k_{j(1),t} =B^{i}_k(X_{t_-})B^{j}_r(X_{t_-})d\overline{Z}^{k}_{r,t}.
\end{eqnarray*}
Here $B(x)$ is given in \refeqn{equation_transformation}, $X^1_t=\sqrt{R_t}\cos(\Theta_t)$, $X^2_t=\sqrt{R_t}\sin(\Theta_t)$ and
$\overline{Z}^i_j$ are given by equation \refeqn{equation_example2}. It is evident that the SDEs \refeqn{equation_R} and \refeqn{equation_Theta}
 are in triangular form. Indeed, the equation for $R_t$ depends only on $R_t$, while the equation for $\Theta_t$ is independent from $\Theta_t$ itself.
 This means that the process $\Theta_t$ can be reconstructed from the process $R_t$ and the semimartingales $(Z'_{(1)},Z'_{(2)})$ using only
 integrations. Furthermore, using the inverse of the stochastic transformation \refeqn{equation_transformation}, we
 can recover both the solution process $X^1_t,X^2_t$ and the initial noise $(Z_{(1)},Z_{(2)})$  using only inversion of functions and
 It\^o integrations. This situation is very similar to what happens in the
 deterministic setting (see \cite{Olver1993,Stephani1989}) and in the Brownian motion case (see \cite{DMU2}),
 where  the presence of a one-parameter symmetry group allows us  to
split the differential system into a system of lower dimension and an integration (the so called reduction and reconstruction by quadratures).
Also the equation for $R_t$ is recognized to have a familiar form. Indeed, in the case where $Z_{(1)}=I_2$ (the two dimensional identity matrix) and
$Z_{(2)}$ is a two dimensional Brownian motion, equation \refeqn{equation_R} becomes the equation of the two dimensional Bessel process. This
fact should not surprise since  the proposed reduction procedure  is the usual reduction procedure of a two dimensional Brownian motion with
respect to the rotation group. For  generic $(Z_{(1)},Z_{(2)})$ the equation for $R_t$ has the form
$$dR_t=d\mathfrak{Z}^1_t+\sqrt{R_{t_-}}d\mathfrak{Z}^2_t+R_{t_-}d\mathfrak{Z}^2_t,$$
 where\small
\begin{eqnarray*}
\mathfrak{Z}^1_t&=&\left[Z'^1_{(2)},Z'^1_{(2)}\right]^c_t+\left[Z'^2_{(2)},Z'^2_{(2)}\right]^c_t+\sum_{0 \leq s \leq t}\left((\Delta
Z'^1_{(2),s})^2+(\Delta Z'^2_{(2),s})^2\right)\\
\mathfrak{Z}^2_t&=& 2 Z'^1_{(2),t}+2\left[Z'^1_{(2)},\overline{Z}'^1_1\right]^c_t+2\left[Z'^2_{(2)},\overline{Z}'^2_1\right]^c_t+\sum_{0 \leq s
\leq t}\left(2\Delta \overline{Z}'^{1}_{1,s} \Delta Z'^1_{(2),s}+2\Delta \overline{Z}'^{2}_{1,s} \Delta Z'^2_{(2),s}\right)\\
\mathfrak{Z}^3_t&=&2 \overline{Z}'^1_{1,t}+\left[\overline{Z}'^1_1,\overline{Z}'^{1}_1\right]_t^c+
\left[\overline{Z}'^2_1,\overline{Z}'^2_1\right]_t^c+\sum_{0 \leq s \leq t}\left((\Delta \overline{Z}'^{1}_{1,t})^2+(\Delta
\overline{Z}'^2_{1,t})^2\right).
\end{eqnarray*}
\normalsize
Equation \refeqn{equation_R}  can be considered as a kind of generalization of affine processes (see \cite{Filipovic2003}): indeed,
in the case where $Z_{(1)}$ is deterministic and $Z_{(2)}$ is a two dimensional Brownian motion, equation \refeqn{equation_R} reduces to a CIR
model equation (appearing in  mathematical finance), with time dependent coefficients.

\begin{remark}
 The gauge symmetry group $O(2)$ with action $\Xi_B$ on the pair $(Z_{(1)},Z_{(2)})$  has
interesting applications in the iterated map theory. Indeed, let   $Z_{(1)},Z_{(2)}$ be discrete-time semimartingales with independent
increments,  $Z_{(1),n}=K_n \cdot Z_{(1),n-1}$ and
 $Z_{(2),n}=Z_{(2),n-1}+ H_n$, where $K_n \in GL(2), H_n \in \mathbb{R}^2$ are random variables independent from
 $Z_{(1),1},...,Z_{(1),n-1},Z_{(2),1},...,Z_{(2),n-1}$. Therefore we have that $(Z_{(1)},Z_{(2)})$ have $O(2)$ as gauge symmetry group with action $\Xi_B$
  if and only if the distribution of $(K_n,H_n) \in GL(2) \times \mathbb{R}^2$ is invariant with respect to the action
 $\Xi_B$. Indeed in the present case the characteristic triplet $(b,A,\nu)$ of $(Z_{(1)},Z_{(2)})$ is
 \begin{eqnarray*}
b&=&0\\
A&=&0\\
\nu(dt,dz)&=& \sum_{n \in \mathbb{N}} \delta_{n}(dt) m_n(dz)\\
 \end{eqnarray*}
 where $\delta_n$ is the Dirac delta distribution on $\mathbb{R}$ with the mass concentrated in $n \in \mathbb{R}$ and $m_n$ is the probability
 distribution on $N=GL(2) \times \mathbb{R}^2$ of the pair of random variables $(K_n,H_n)$. By Theorem \ref{theorem_characteristic3}, $\Xi_B$ is
 a gauge symmetry of $(Z_{(1)},Z_{(2)})$ if and only if, for any $B \in O(2)$, $\Xi_{B*}(\nu)=\nu$ which is equivalent to request that
 $\Xi_{B*}(m_n)=m_n$. This implies that the law of $(K_n,H_n)$ is invariant with respect to the action $\Xi_B$.\\
The invariance of the law of $K_n \in GL(2)$ with respect to $\Xi_B$ is exactly the invariance of the matrix random variable $K_n$ with respect
to orthogonal conjugation, and the law of the $\mathbb{R}^2$ random variable $H_n$ is rotationally invariant. This kind of random variables and
 related processes are deeply studied in random matrix theory (see, e.g., \cite{Anderson2010,Mehta2004}).\\
\end{remark}

\section*{Acknowledgements}
The first author gratefully acknowledges the Department of Mathematics, Universit\`a degli Studi di Milano, for financial support through a grant of the fourth author. The hospitality of IAM,HCM, University of Bonn, given by to second and fourth author is gratefully acknowledged. This work was also supported  by Gruppo Nazionale Fisica Matematica (GNFM-INdAM).

\bibliographystyle{plain}

\bibliography{levy_symmetries(6)}

\end{document}